\documentclass[reqno]{amsart}

\usepackage{a4wide}
\usepackage{amssymb}
\usepackage{graphicx}
\usepackage[backref]{hyperref}
\usepackage{mathscinet}

\usepackage{mathtools}

\usepackage{subfigure}
\usepackage{calc}
\usepackage{tikz}
\usetikzlibrary{cd,arrows}
\usepackage[english]{babel}

\hypersetup{colorlinks=true,linkcolor=blue}

\numberwithin{equation}{section}

\theoremstyle{plain}
\newtheorem{theorem}{Theorem}[section]
\newtheorem{lemma}[theorem]{Lemma}
\newtheorem{proposition}[theorem]{Proposition}
\newtheorem{corollary}[theorem]{Corollary}

\newtheorem{bigtheorem}{Theorem}

\theoremstyle{definition}
\newtheorem{definition}[theorem]{Definition}
\newtheorem{example}[theorem]{Example}
\newtheorem{assumption}[theorem]{Assumption}

\theoremstyle{remark}
\newtheorem{remark}[theorem]{Remark}
\newtheorem{question}[theorem]{Question}

\newcommand{\R}{\mathbb{R}}
\newcommand{\Z}{\mathbb{Z}}
\newcommand{\Zmod}[1]{\Z/#1\Z}

\newcommand{\bL}{\mathbf{L}}

\newcommand{\bR}{\mathbf{R}}
\newcommand{\bU}{\mathbf{U}}
\newcommand{\bV}{\mathbf{V}}

\newcommand{\bh}{\mathbf{h}}
\newcommand{\bn}{\mathbf{n}}
\newcommand{\bp}{\mathbf{p}}

\newcommand{\bx}{\mathbf{x}}
\newcommand{\bv}{\mathbf{v}}
\newcommand{\bw}{\mathbf{w}}

\newcommand{\cA}{\mathcal{A}}

\newcommand{\cI}{\mathcal{I}}

\newcommand{\cM}{\mathcal{M}}

\newcommand{\cP}{\mathcal{P}}
\newcommand{\cR}{\mathcal{R}}

\newcommand{\sB}{\mathsf{B}}
\newcommand{\sC}{\mathsf{C}}

\newcommand{\sG}{\mathsf{G}}

\newcommand{\sfI}{\mathsf{I}}
\newcommand{\sK}{\mathsf{K}}
\newcommand{\sL}{\mathsf{L}}

\newcommand{\sV}{\mathsf{V}}

\newcommand{\sfa}{\mathsf{a}}

\newcommand{\sfc}{\mathsf{c}}
\newcommand{\sfe}{\mathsf{e}}

\newcommand{\sfh}{\mathsf{h}}

\newcommand{\sfv}{\mathsf{v}}

\newcommand{\fR}{\mathfrak{R}}

\newcommand{\fr}{\mathfrak{r}}

\renewcommand{\bar}{\overline}
\renewcommand{\hat}{\widehat}
\renewcommand{\tilde}{\widetilde}

\newcommand{\wt}{\operatorname{wt}}

\newcommand{\lcusp}{\prec}
\newcommand{\rcusp}{\succ}

\DeclareMathOperator{\aug}{Aug}
\DeclareMathOperator{\sgn}{sgn}
\DeclareMathOperator{\val}{val}

\DeclareMathOperator{\res}{Res}

\newcommand{\vPi}{{\bV\Pi}}

\renewcommand{\emptyset}{\varnothing}

\newcommand*{\DecorationScale}{0.3}

\DeclareRobustCommand*{\crossing}{{
\vcenter{\hbox{\begin{tikzpicture}[scale=\DecorationScale]
\draw[thick,-](-1,1)--(1,-1);
\draw[thick,-](-1,-1)--(1,1);
\end{tikzpicture}}}
}}

\DeclareRobustCommand*{\crossingblack}{{
\vcenter{\hbox{\begin{tikzpicture}[scale=\DecorationScale]
\draw[densely dotted, fill=gray] (0,0) circle (0.5);
\draw[thick,-](-1,1)--(1,-1);
\draw[thick,-](-1,-1)--(1,1);
\end{tikzpicture}}}
}}

\DeclareRobustCommand*{\crossingpositive}{{
\vcenter{\hbox{\begin{tikzpicture}[scale=\DecorationScale]
\draw[thick,-](-1,1)--(1,-1);
\draw[thick,-](-1,-1)--(-0.3,-0.3);
\draw[thick,-](0.3,0.3)--(1,1);
\end{tikzpicture}}}
}}

\DeclareRobustCommand*{\crossingnegative}{{
\vcenter{\hbox{\begin{tikzpicture}[scale=\DecorationScale]
\draw[thick,-](-1,-1)--(1,1);
\draw[thick,-](-1,1)--(-0.3,0.3);
\draw[thick,-](0.3,-0.3)--(1,-1);
\end{tikzpicture}}}
}}

\DeclareRobustCommand*{\crossingsingular}{{
\vcenter{\hbox{\begin{tikzpicture}[scale=\DecorationScale]
\draw[thick,-](-1,1)--(1,-1);
\draw[thick,-](-1,-1)--(1,1);
\draw[fill] (0,0) circle (5pt);
\end{tikzpicture}}}
}}

\DeclareRobustCommand*{\crossinghorizontal}{{
\vcenter{\hbox{\begin{tikzpicture}[scale=\DecorationScale]
\draw[thick,-](-1,1) arc (225:315:1.414);
\draw[thick,-](-1,-1) arc (-225:-315:1.414);
\end{tikzpicture}}}
}}

\DeclareRobustCommand*{\crossingvertical}{{
\vcenter{\hbox{\begin{tikzpicture}[scale=\DecorationScale]
\draw[thick,-](-1,-1) arc (-45:45:1.414);
\draw[thick,-](1,-1) arc (-135:-225:1.414);
\end{tikzpicture}}}
}}

\DeclareRobustCommand*{\positivekink}{{
\vcenter{\hbox{\begin{tikzpicture}[scale=\DecorationScale]
\draw[thick,-](-1,1) -- (0.3, -0.3) arc (45: -225: 0.4242);
\draw[thick,-](0.3,0.3)--(1,1);
\end{tikzpicture}}}
}}

\DeclareRobustCommand*{\negativekink}{{
\vcenter{\hbox{\begin{tikzpicture}[scale=\DecorationScale]
\draw[thick,-](0.3, -0.3) arc (45: -225: 0.4242) -- (1,1);
\draw[thick,-] (-0.3, 0.3)--(-1,1);
\end{tikzpicture}}}
}}

\DeclareRobustCommand*{\rightkink}{{
\vcenter{\hbox{\begin{tikzpicture}[scale=\DecorationScale]
\draw[thick,-](-1,1) -- (0.3, -0.3) arc (-135: 135: 0.4242);
\draw[thick,-](-0.3,-0.3)--(-1,-1);
\end{tikzpicture}}}
}}

\DeclareRobustCommand*{\leftarc}{{
\vcenter{\hbox{\begin{tikzpicture}[scale=\DecorationScale]
\draw[thick,-](0,1) arc ( 90: 275: 1);
\end{tikzpicture}}}
}}

\DeclareRobustCommand*{\horizontalline}{{
\vcenter{\hbox{\begin{tikzpicture}[scale=\DecorationScale]
\draw[thick,-] (-1, 0)--(1,0);
\end{tikzpicture}}}
}}

\DeclareRobustCommand*{\theunknot}{{
\vcenter{\hbox{\begin{tikzpicture}[scale=\DecorationScale]
\draw[thick,-] (0,0) circle (1);
\end{tikzpicture}}}
}}

\DeclareRobustCommand*{\Lcusp}{{
\vcenter{\hbox{\begin{tikzpicture}[scale=\DecorationScale]
\draw[thick,-] (1,1) arc (-45:-90:2.414) arc (90:45:2.414);
\end{tikzpicture}}}
}}

\DeclareRobustCommand*{\Lcuspfour}{{
\vcenter{\hbox{\begin{tikzpicture}[scale=\DecorationScale]
\draw[thick,-] (1,1) arc (-30:-90:2) arc (90:30:2);
\draw[thick,-] (1,0.464) arc (-60:-90:3.464) arc (90:60:3.464);
\draw[fill] (-0.732,0) circle (5pt);
\end{tikzpicture}}}
}}

\DeclareRobustCommand*{\Ldoublecusp}{{
\vcenter{\hbox{\begin{tikzpicture}[scale=\DecorationScale]
\draw[thick,-] (1,1) arc (-45:-90:2.414) arc (90:45:2.414);
\draw[thick,-] (1,-1) arc (45:90:2.414) arc (-90:-45:2.414);
\end{tikzpicture}}}
}}

\DeclareRobustCommand*{\Ldoublecuspblack}{{
\vcenter{\hbox{\begin{tikzpicture}[scale=\DecorationScale]
\draw[fill=gray] (0.45,0) circle (0.4);
\draw[thick,-] (1,1) arc (-45:-90:2.414) arc (90:45:2.414);
\draw[thick,-] (1,-1) arc (45:90:2.414) arc (-90:-45:2.414);
\end{tikzpicture}}}
}}

\DeclareRobustCommand*{\Lstackedcusp}{{
\vcenter{\hbox{\begin{tikzpicture}[scale=\DecorationScale]
\draw[thick,-] (1,1) arc (-55:-90:2) arc (90:55:2);
\draw[thick,-] (1,-1) arc (55:90:2) arc (-90:-55:2);
\end{tikzpicture}}}
}}

\DeclareRobustCommand*{\Lnestedcusp}{{
\vcenter{\hbox{\begin{tikzpicture}[scale=\DecorationScale]
\draw[thick,-] (1,1) arc (-30:-90:2) arc (90:30:2);
\draw[thick,-] (1,0.2) arc (-60:-90:1.5) arc (90:60:1.5);
\end{tikzpicture}}}
}}

\DeclareRobustCommand*{\Rcusp}{{
\vcenter{\hbox{\begin{tikzpicture}[scale=\DecorationScale]
\draw[thick,-] (-1,1) arc (-135:-90:2.414) arc (90:135:2.414);
\end{tikzpicture}}}
}}

\title{Ruling invariants for Legendrian graphs}

\author{Byung Hee An}
\address{Center for Geometry and Physics, Institute for Basic Science, Pohang 37673, Republic of Korea}
\email{anbyhee@ibs.re.kr}
\author{Youngjin Bae}
\address{Research Institute for Mathematical Sciences, Kyoto University, Kyoto 606-8502, Japan}
\email{ybae@kurims.kyoto-u.ac.jp}
\author{Tam\'as K\'alm\'an}
\address{Department of Mathematics, Tokyo Institute of Technology, Tokyo 152-8551, Japan}
\email{kalman@math.titech.ac.jp}

\begin{document}

\begin{abstract}
We define ruling invariants for even-valence Legendrian graphs in standard contact three-space. We prove that rulings exist if and only if the DGA of the graph, introduced by the first two authors, has an augmentation. We set up the usual ruling polynomials for various notions of gradedness and prove that if the graph is four-valent, then the ungraded ruling polynomial appears in Kauffman--Vogel's graph version of the Kauffman polynomial. 
Our ruling invariants are compatible with certain vertex-identifying operations as well as vertical cuts and gluings of front diagrams. 
We also show that Leverson's definition of a ruling of a Legendrian link in a connected sum of $S^1\times S^2$'s can be seen as a special case of ours.
\end{abstract}

\maketitle
\tableofcontents

\section{Introduction}

Ruling invariants for Legendrian knots and links were introduced by Chekanov and Pushkar \cite{ChekPush2005}, and independently by Fuchs \cite{Fuchs2003}.
The motivation comes from a {\em generating family}, which is a family of functions whose critical values give the front of a Legendrian knot. Rulings can be used to distinguish smoothly isotopic Legendrians even if share the same Thurston--Bennequin number and rotation number, such as Chekanov's famous pair of Legendrians of knot type $5_2$. For that reason we call ruling invariants {\em non-classical}.
 
There is another non-classical construction, the so called Chekanov--Eliashberg DG-algebra, originating from a relative version of contact homology, i.e., holomorphic curve techniques \cite{Chekanov2001}.
The homology of the DG-algebra is invariant under Legendrian isotopy and also distinguishes the above pair of Legendrians via a method called linearization of DG-algebras.

There is a deep relation between the two approaches: the existence of a ruling and the linearizability of the DG-algebra, i.e, the existence of a so called augmentation, are equivalent. This is established by Fuchs \cite{Fuchs2003}, Fuchs--Ishkhanov \cite{FI2004}, and Sabloff \cite{Sabloff2005} and extended by Leverson \cite{Leverson2016,Leverson2017}.

On the other hand, the so called ungraded ruling polynomial, which is a weighted (by genus) count of all rulings, appears as a certain sequence of coefficients of the Kauffman polynomial. 
These are leading coefficients when the upper bound for the Thurston--Bennequin number given by the Kauffman polynomial is sharp, and otherwise all zeros \cite{Rutherford2006}.
(Hence the ungraded ruling polynomial is in fact a classical invariant; to access the full power of rulings, one has to narrow their counts to only $\Z$-graded ones.)

Legendrian graphs have been studied using classical invariants \cite{OP2012}. Recently they have also drawn attention as singular Legendrians appearing in the study of Lagrangian skeleta of Weinstein manifolds \cite{Nadler2017, GPS2018}.
The first two authors developed a DG-algebra invariant for Legendrian graphs via a careful consideration of the algebraic issues that arise near the vertices of graphs \cite{AB2018}.

In this article, we extend the definition of ruling from Legendrian links to Legendrian graphs.
Of course, the main issue will be to analyze the behavior of each ruling near the vertices.
We restrict ourselves to Legendrian graphs with only even-valent vertices and demand that the ruling at each vertex be parametrized by the set of perfect matchings of the incident edges. In other words, we regard a Legendrian graph as a set of Legendrian links (with markings) which can be obtained by resolutions of vertices, indexed by a perfect matching at each vertex.

With this extension, we show the equivalence between the existence of ($\rho$-graded) rulings and of ($\rho$-graded) augmentations for Legendrian graphs. 

\begin{bigtheorem}[Theorem~\ref{theorem:equivalence of existence for graphs}]
Let $\sL$ be a bordered Legendrian graph. Then a $\rho$-graded normal ruling for $\sL$ exists if and only if a $\rho$-graded augmentation for the DG-algebra $\cA(\sL)$ exists.
\end{bigtheorem}

Kauffman and Vogel introduced a polynomial invariant for four-valent graphs embedded in $\R^3$ which generalizes the two-variable Kauffman polynomial of links. We also show that the ungraded
ruling polynomial can be realized as a certain sequence of coefficients of this topological graph invariant.

\begin{bigtheorem}[Theorem~\ref{theorem:ruling and kauffman polynomial for graphs}]
Let $\sL$ be a regular front projection of a four-valent Legendrian graph.
The ungraded $(\rho=1)$ ruling polynomial $R_1(\sL)$ for $\sL$ is the same as the coefficient of $a^{-\mathbf{tb}(\sL)-1}$ ($a^{-1}$, resp.) in the shifted Kauffman--Vogel polynomial $z^{-1}F_\sL$ (unnormalized polynomial $z^{-1}[\sL]$, resp.) after replacing $A$ and $B$ with $(z-1)$ and $-1$, respectively.
\end{bigtheorem}

The paper is organized as follows: In Section~2, we introduce basic concepts of (bordered) Legendrian graphs with Maslov potential. For each {\em matching} datum at the vertex, we assign a corresponding {\em resolution}, a bordered smooth Legendrian with {\em marking}.

In Section~3, we define ruling invariants for Legendrian graphs by considering all resolutions of bordered Legendrian graphs respecting the grading condition.
We also discuss the relation between our and Leverson's ruling invariant for Legendrian links in $\#(S^1\times S^2)$.

In Section~4, we first recall the DGA associated to a Legendrian graph, and establish the equivalence between the existence of a ruling for the Legendrian graph and the existence of an augmentation of its DGA.
In particular, when the Legendrian graph is four-valent, we show that the ruling polynomial appears as a certain coefficient of the Kauffman--Vogel polynomial of the underlying graph.

Section~5 is devoted to showing that the resolutions defined in Section~2 are compatible with the Reidemeister moves at the vertex, which implies the invariance of the ruling invariant.

\emph{Acknowledgments}: 
The first author is supported by IBS-R003-D1.
The second author is supported by Japan Society for the Promotion of Science (JSPS) International Research Fellowship Program, and he thanks Research Institute for Mathematical Sciences, Kyoto University for its warm hospitality.
The third author is supported by JSPS Grant-in-Aid for Scientific Research C (no.\ 17K05244).

\section{Bordered Legendrian graphs}

We work with the standard tight contact structure $\ker(dz-y\,dx)$ of $\R^3$. We use the front projection to $\R^2_{xz}$. A {\em cusp} along a Legendrian curve is a point where the tangent is parallel to the $y$-axis.

\subsection{Bordered Legendrian graphs}

\begin{definition}
A {\em bordered Legendrian graph $\sL$ of type $(\ell,r)$} is a Legendrian graph embedded in $[-M,M]\times \R^2_{yz}$ for some $M>0$ such that all cusps and vertices are contained in the interior and $\sL$ intersects the $x=-M$ and $x=M$ planes exactly at $\ell$ and $r$ points, respectively. We also assume these intersections to be perpendicular, which implies that they occur at points with $y$-coordinate equal to $0$. In symbols,
\begin{align*}
\#(\sL\cap (\{-M\}\times \R^2_{yz})) &= \#(\sL\cap (\{(-M,0)\}\times \R_z))=\ell \quad\text{and}\\
\#(\sL\cap (\{M\}\times \R^2_{yz})) &= \#(\sL\cap (\{(M,0)\}\times \R_z))=r
\end{align*}
We say that two bordered Legendrian graphs of the same type are {\em equivalent} if they are isotopic through bordered Legendrian graphs.

We denote the sets of vertices, double points (of the front projection), left and right cusps by $\sV_\sL, \sC_\sL, \lcusp_\sL$ and $\rcusp_\sL$, respectively.
If $\sV_\sL=\emptyset$, then we call $\sL$ a {\em bordered Legendrian link}.
\end{definition}

\begin{assumption}
We assume that $\sL$ has a {\em regular} front projection, so that 
\begin{enumerate}
\item there are no triple points;
\item cusps and vertices are not double points;
\item all cusps, vertices and double points have pairwise different $x$-coordinates.
\end{enumerate}
\end{assumption}

\begin{remark}
If $\ell=r=0$, then $\sL$ is a usual Legendrian graph.
\end{remark}

\begin{example}
As a special case, one may consider the bordered Legendrian graphs $\mathsf{0}_\ell$ and $\mathsf{\infty}_r$ of types $(0,\ell)$ and $(r,0)$, respectively, whose underlying graphs are the $\ell$- and $r$-corollas as follows:
\begin{align*}
\mathsf{0}_\ell&\coloneqq\vcenter{\hbox{\includegraphics{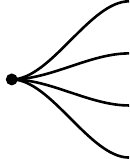}}}~;&
\mathsf{\infty}_r&\coloneqq\vcenter{\hbox{\includegraphics{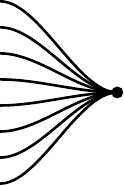}}}~.
\end{align*}
Furthermore, we will consider the bordered Legendrian graph $\sfI_n$ of type $(n,n)$ consisting of $n$ parallel (arbitrarily short) arcs:
\[
\sfI_n\coloneqq \vcenter{\hbox{\includegraphics{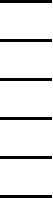}}}~.
\]
\end{example}

Notice that for each Legendrian graph $\sL$ of type $(\ell,r)$, there are two natural inclusions $\iota_L\colon \sfI_\ell\to \sL$ and $\iota_R\colon \sfI_r\to \sL$ ``near the border'':
\[
\sfI_\ell\stackrel{\iota_L}\longrightarrow \sL \stackrel{\iota_R}\longleftarrow \sfI_r.
\]

It is not hard to see that two bordered Legendrian graphs are equivalent if and only if their front projections are related through a sequence of the following Reidemeister moves:
\begingroup
\allowdisplaybreaks
\begin{align*}
\vcenter{\hbox{\includegraphics{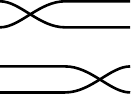}}}&\stackrel{\rm(0_a)}\longleftrightarrow
\vcenter{\hbox{\includegraphics{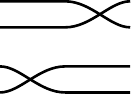}}}&
\vcenter{\hbox{\includegraphics{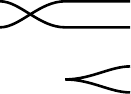}}}&\stackrel{\rm(0_b)}\longleftrightarrow
\vcenter{\hbox{\includegraphics{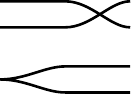}}}&
\vcenter{\hbox{\includegraphics{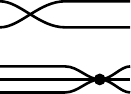}}}&\stackrel{\rm(0_c)}\longleftrightarrow
\vcenter{\hbox{\includegraphics{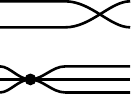}}}\\[1pc]
\vcenter{\hbox{\includegraphics{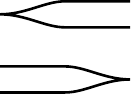}}}&\stackrel{\rm(0_d)}\longleftrightarrow
\vcenter{\hbox{\includegraphics{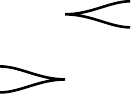}}}&
\vcenter{\hbox{\includegraphics{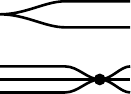}}}&\stackrel{\rm(0_e)}\longleftrightarrow
\vcenter{\hbox{\includegraphics{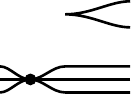}}}&
\vcenter{\hbox{\includegraphics{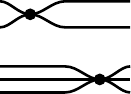}}}&\stackrel{\rm(0_f)}\longleftrightarrow
\vcenter{\hbox{\includegraphics{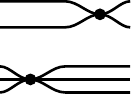}}}\\[1pc]
\vcenter{\hbox{\includegraphics{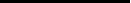}}}&\stackrel{\rm(I)}\longleftrightarrow
\vcenter{\hbox{\includegraphics{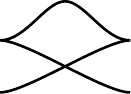}}}&
\vcenter{\hbox{\includegraphics{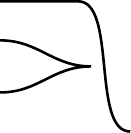}}}&\stackrel{\rm(II)}\longleftrightarrow
\vcenter{\hbox{\includegraphics{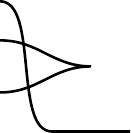}}}&
\vcenter{\hbox{\includegraphics{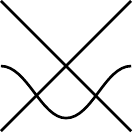}}}&\stackrel{\rm(III)}\longleftrightarrow
\vcenter{\hbox{\includegraphics{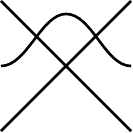}}}\\
\vcenter{\hbox{\includegraphics{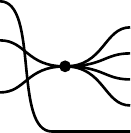}}}&\stackrel{\rm(IV)}\longleftrightarrow
\vcenter{\hbox{\includegraphics{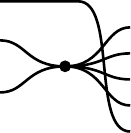}}}&
\vcenter{\hbox{\includegraphics{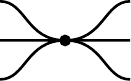}}}&\stackrel{\rm(V)}\longleftrightarrow
\vcenter{\hbox{\includegraphics{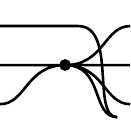}}}~.
\end{align*}
\endgroup

\subsubsection{Concatenations and closures}
Let $\sL_1$ and $\sL_2$ be two bordered Legendrian graphs of types $(\ell, r)$ and $(r,s)$, respectively. Then there is a canonical operation, called {\em gluing}, which is a concatenation of $\sL_1$ and $\sL_2$, and can also be regarded as a {\em push-out} of the following diagram:
\[
\begin{tikzcd}
\sfI_r\arrow[r]\arrow[d] & \sL_2\ar[d,dashed]\\
\sL_1\ar[r,dashed]& \sL
\end{tikzcd}
\]
We will write $\sL = \sL_1\coprod_{\sfI_r}\sL_2$ or simply $\sL=\sL_1\cdot \sL_2$.

\begin{definition}
Let $\sL$ be a bordered Legendrian graph of type $(\ell,r)$.
The {\em closure} $\hat \sL$ of $\sL$ is the Legendrian graph obtained by gluing $\mathsf{0}_\ell$ and $\mathsf{\infty}_r$ to $\sL$ on the left and right, respectively.
\end{definition}

\begin{example}
Let $\sL$ be the following bordered Legendrian graph of type $(4,8)$:
\[
\sL=\vcenter{\hbox{\includegraphics{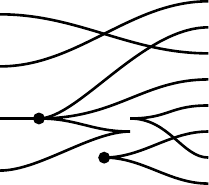}}}~.
\]
Then its closure looks as follows:
\[
\hat \sL = \mathsf{0}_4 \cdot \sL \cdot \mathsf{\infty}_8=
\vcenter{\hbox{\includegraphics{S_ell.pdf}}}
\vcenter{\hbox{\includegraphics{L_5_7.pdf}}}
\vcenter{\hbox{\includegraphics{S_r.pdf}}}~.
\]
\end{example}

The closures of $\mathsf{0}_\ell$ and $\mathsf{\infty}_l$ are equivalent and will be denoted by $\Theta_\ell$:
\[
\hat {\mathsf{0}}_\ell = \Theta_\ell = 
\vcenter{\hbox{\includegraphics{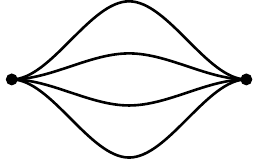}}}\ \ =\hat \sfI_\ell=
\hat {\mathsf{\infty}}_\ell.
\]

\subsubsection{Maslov potentials and markings}
Let $\fR$ denote either $\Z$ or $\Z_m$ for some $m\ge 2$, generated by $1_\fR$.

\begin{definition}
A {\em Maslov potential $\mu$} is an $\fR$-valued function on the components of $\sL\setminus \left(\sV_\sL\cup \lcusp_\sL\cup \rcusp_\sL\right)$ such that
\begin{align*}
\mu(\alpha)&=\mu(\alpha')+1_\fR
\end{align*}
whenever $\alpha$ meets $\alpha'$ at a cusp so that locally, $z$-values along $\alpha$ exceed those along $\alpha'$.
\end{definition}

It is easy to see that a Maslov potential $\mu$ given on $\sL$ induces a Maslov potential $\hat\mu$ on $\hat \sL$.
Moreover, a Maslov potential defines a grading for each double point of the front projection (the set of which will be denoted with $\sC$) by
\begin{align}
|\sfc|\coloneqq \mu(\alpha)-\mu(\alpha')\in\fR,\label{equation:grading}
\end{align}
where $\alpha$ and $\alpha'$ are the arcs of $\sL$ whose projections intersect at $\sfc$, furthermore the preimage of $\sfc$ on $\alpha$ has a lower $y$-coordinate than the preimage on $\alpha'$.

For each Legendrian graph $(\sL,\mu)$ of type $(\ell,r)$, with Maslov potential, we obtain two trivial Legendrian graphs $(\sfI_\ell, \iota_L^*(\mu))$ and $(\sfI_r, \iota_R^*(\mu))$, with potentials, by pulling $\mu$ back via the canonical inclusions $\iota_L$ and $\iota_R$. In other words, we have
\[
\iota_L^*(\sL,\mu) \longleftarrow (\sL,\mu)\longrightarrow \iota_R^*(\sL,\mu),
\]
where 
\[
\iota^*_L(\sL,\mu)=(\sfI_\ell, \iota_L^*(\mu))\quad\text{and}\quad
\iota^*_R(\sL,\mu)=(\sfI_r, \iota_R^*(\mu)).
\]

Let $(\sL_1,\mu_1)$ and $(\sL_2,\mu_2)$ be two Legendrian graphs, of respective types $(\ell,r)$ and $(r,s)$, with Maslov potentials. 
Assume furthermore that the two induced Maslov potentials $\iota_R^*(\mu_1)$ and $\iota_L^*(\mu_2)$ on $\sfI_r$ coincide.
In this case we define the gluing $(\sL,\mu)\coloneqq(\sL_1,\mu_1)\cdot(\sL_2,\mu_2)$ by
$\sL\coloneqq\sL_1\cdot \sL_2$ and 
$\mu\coloneqq \mu_1\amalg \mu_2$.

\begin{definition}[Marked bordered Legendrian graphs]\label{def:marking}
Let $\sC=\sC(\sL)$ be the set of crossings of $\sL$. For a subset $\sB$ of $\sC$, the pair $\bL=((\sL,\mu), \sB)$ is called a {\em marked} bordered Legendrian graph.

For simplicity we put $\bL=(\sL,\mu)$ if $\bL$ has no markings, i.e., when $\sB=\emptyset$.
\end{definition}

We will call crossings in $\sB$ and $\sC\setminus \sB$ {\em marked} and {\em regular} crossings, respectively, and represent them in our diagrams as follows:
\begin{align*}
\crossingblack &\in\sB,& \crossing &\in\sC\setminus \sB.
\end{align*}

The canonical inclusions $\iota_L$ and $\iota_R$ induce two marked bordered Legendrian graphs
\[
\iota_L^*(\bL)=(\sfI_\ell, \iota_L^*(\mu))\longleftarrow \bL\longrightarrow \iota_R^*(\bL)=(\sfI_r, \iota_R^*(\mu)).
\]

Let $\bL_1=((\sL_1,\mu_1),\sB_1)$ and $\bL_2=((\sL_2,\mu_2),\sB_2)$ be marked bordered Legendrian graphs of types $(\ell,r)$ and $(r,s)$, respectively. If $\iota_R^*(\bL_1)=\iota_L^*(\bL_2)$, then we define their concatenation $\bL\coloneqq((\sL,\mu),\sB)$ by
\begin{align*}
(\sL,\mu)&\coloneqq (\sL_1,\mu_1)\cdot(\sL_2,\mu_2),&
\sB&\coloneqq \sB_1\amalg\sB_2.
\end{align*}
We will often shorten the notation to $\bL=\bL_1\cdot\bL_2$.

For $\bL$ of type $(\ell,r)$, by gluing $\mathbf{0}_\ell$ and $\pmb{\infty}_r$ to the left and right of $\bL$, we obtain the closure $\hat\bL$ as before:
\begin{align*}
\hat\bL&\coloneqq \mathbf{0}_\ell(\iota_L^*(\mu))\cdot \bL\cdot \pmb{\infty}_r(\iota_R^*(\mu)),&
\mathbf{0}_\ell(-)&\coloneqq(\mathsf{0}_\ell, -),&
\pmb{\infty}_\ell(-)&\coloneqq(\mathsf{\infty}_\ell, -).
\end{align*}

\begin{definition}[Equivalences of marked bordered Legendrian graphs]\label{Definition:equivalence of marked graphs}
We say that two marked bordered Legendrian graphs $\bL_1$ and $\bL_2$ are {\em equivalent} if one can be transformed to the other via a sequence of usual Reidemeister moves and \emph{marked Reidemeister moves} depicted in Figure~\ref{figure:marked Reidemeister moves}.
\end{definition}

\begin{figure}[ht]
\begin{align*}
\tag{$0$}
\vcenter{\hbox{\includegraphics{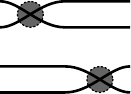}}}&\longleftrightarrow
\vcenter{\hbox{\includegraphics{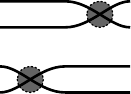}}}&
\vcenter{\hbox{\includegraphics{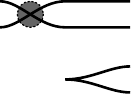}}}&\longleftrightarrow
\vcenter{\hbox{\includegraphics{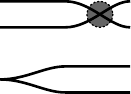}}}
\end{align*}

\begin{align*}
\tag{$\mathrm{II}$}
\begin{tikzcd}[ampersand replacement=\&]
\ 
\&\quad\vcenter{\hbox{\input{R2_B_5_gen_input.tex}}}\arrow[d,leftrightarrow]\\
\quad\vcenter{\hbox{\input{R2_B_2_gen_input.tex}}}\arrow[r,leftrightarrow] \& \quad\vcenter{\hbox{\input{R2_B_1_gen_input.tex}}}\arrow[r,leftrightarrow] \& \quad\vcenter{\hbox{\input{R2_B_4_gen_input.tex}}}\\
\ \&\quad\vcenter{\hbox{\input{R2_B_3_gen_input.tex}}}\arrow[u,leftrightarrow]
\end{tikzcd}
\end{align*}

\begin{align*}
\tag{$\mathrm{III}$}
\vcenter{\hbox{\includegraphics{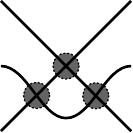}}}&\longleftrightarrow
\vcenter{\hbox{\includegraphics{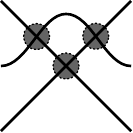}}}&
\vcenter{\hbox{\includegraphics{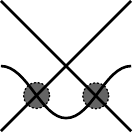}}}&\longleftrightarrow
\vcenter{\hbox{\includegraphics{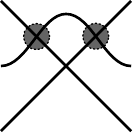}}}\\
\vcenter{\hbox{\includegraphics{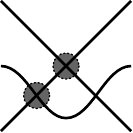}}}&\longleftrightarrow
\vcenter{\hbox{\includegraphics{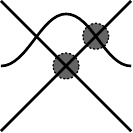}}}&
\vcenter{\hbox{\includegraphics{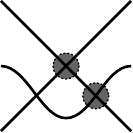}}}&\longleftrightarrow
\vcenter{\hbox{\includegraphics{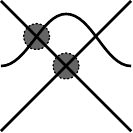}}}
\end{align*}

\begin{align*}
\tag{$\rm S$}
\vcenter{\hbox{\includegraphics{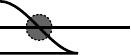}}}\longleftrightarrow
\vcenter{\hbox{\includegraphics{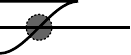}}}
\end{align*}

\begin{align*}
\tag{$\rm T$}
\vcenter{\hbox{\includegraphics{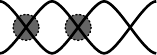}}}&\longleftrightarrow
\vcenter{\hbox{\includegraphics{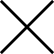}}}&
\vcenter{\hbox{\includegraphics{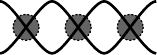}}}&\longleftrightarrow
\vcenter{\hbox{\includegraphics{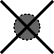}}}
\end{align*}
\caption{Reidemeister moves for marked Legendrian graphs: In the move (II), $n=0,1,2,\dots$.}
\label{figure:marked Reidemeister moves}
\end{figure}

\begin{remark}\label{remark:marked Reidemeister move T}
The marked Reidemeister move $(\mathrm{T})$ does not imply one can cancel out two subsequent black crossings.
\[
\begin{tikzpicture}[baseline=-.5ex]
\draw[thick,rounded corners] (-0.5,0.25) -- (0,-0.25) -- (0.5,0.25);
\draw[thick,rounded corners] (-0.5,-0.25) -- (0,0.25) -- (0.5,-0.25);
\draw[fill,densely dotted,opacity=0.6] (-0.25,0) circle (4pt) (0.25,0) circle (4pt);
\end{tikzpicture}
\neq
\begin{tikzpicture}[baseline=-.5ex]
\draw[thick] (-0.5,0.25) to[out=-45,in=180] (0,0.1) to[out=0,in=225] (0.5,0.25);
\draw[thick] (-0.5,-0.25) to[out=45,in=180] (0,-0.1) to[out=0,in=-225] (0.5,-0.25);
\end{tikzpicture}
\]
\end{remark}


\subsection{Resolution of a vertex}
For convenience's sake, let us denote the set of integers $\{1,\dots, 2n\}$ by $[2n]$.

\begin{definition}[Matchings]
Let $X$ be a finite set. A {\em matching} $\phi$ on $X$ is an involution which can be expressed as
\[
\phi=\{\{x_1,\phi(x_1)\},\dots,\{x_m,\phi(x_m)\}\},
\]
where $x_i$ is not necessarily different from $\phi(x_i)$.

We say that $\phi$ is {\em perfect} if $\phi$ has no fixed points, and denote the set of all perfect matchings on $X$ by $\cP_X$.
\end{definition}

Let $\sL$ be a (bordered) Legendrian graph and $\sfv\in \sV_\sL$ be a vertex.
We say that $\sfv$ is of type $(\ell,r)$ if $\sfv$ looks locally as follows:
\[
(\sL_\sfv,\sB_\sfv=\emptyset)\coloneqq\vcenter{\hbox{
\begingroup%
  \makeatletter%
  \providecommand\color[2][]{%
    \errmessage{(Inkscape) Color is used for the text in Inkscape, but the package 'color.sty' is not loaded}%
    \renewcommand\color[2][]{}%
  }%
  \providecommand\transparent[1]{%
    \errmessage{(Inkscape) Transparency is used (non-zero) for the text in Inkscape, but the package 'transparent.sty' is not loaded}%
    \renewcommand\transparent[1]{}%
  }%
  \providecommand\rotatebox[2]{#2}%
  \ifx\svgwidth\undefined%
    \setlength{\unitlength}{97.20234362bp}%
    \ifx\svgscale\undefined%
      \relax%
    \else%
      \setlength{\unitlength}{\unitlength * \real{\svgscale}}%
    \fi%
  \else%
    \setlength{\unitlength}{\svgwidth}%
  \fi%
  \global\let\svgwidth\undefined%
  \global\let\svgscale\undefined%
  \makeatother%
  \begin{picture}(1,0.48924235)%
    \put(0,0){\includegraphics[width=\unitlength,page=1]{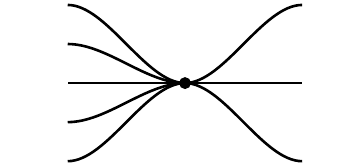}}%
    \put(0.52492355,0.28166138){\color[rgb]{0,0,0}\makebox(0,0)[lb]{\smash{$\sfv$}}}%
    \put(0.91776693,0.47263861){\color[rgb]{0,0,0}\makebox(0,0)[lb]{\smash{\tiny$1$}}}%
    \put(0.92166402,0.23978682){\color[rgb]{0,0,0}\makebox(0,0)[lb]{\smash{\tiny$2$}}}%
    \put(0.9236126,0.00596053){\color[rgb]{0,0,0}\makebox(0,0)[lb]{\smash{\tiny$r$}}}%
    \put(0.01851121,0.47458673){\color[rgb]{0,0,0}\makebox(0,0)[lb]{\smash{\tiny$r+1$}}}%
    \put(0.01851121,0.35182786){\color[rgb]{0,0,0}\makebox(0,0)[lb]{\smash{\tiny$r+2$}}}%
    \put(0.01851121,0.0040124){\color[rgb]{0,0,0}\makebox(0,0)[lb]{\smash{\tiny$r+\ell$}}}%
  \end{picture}%
\endgroup%
}}~.
\]
We note that here, the labels of the incident edges are dictated by consideration of the Lagrangian projection.

We require that $\ell+r$ be {\em even}, say $2n$, but this is not necessary for $\ell$ and $r$. 
Let us denote the set of perfect matchings of half-edges adjacent to $\sfv$ by $\cP_\sfv$.
Then the labelling convention described above induces the bijection
\[
\cP_\sfv\simeq\cP_{[2n]}.
\]

Next, we describe ways of resolving $\sfv$, indexed by the set of perfect matchings.
For a given perfect matching $\phi \in \cP_{[2n]}$, we split $[2n]$ into three $\phi$-invariant subsets, 
\(
[2n]=L\amalg B\amalg R
\), where
\begin{align*}
L=L(\phi)&\coloneqq\{ i\in[2n] \mid i, \phi(i) \ge r+1\};\\
B=B(\phi)&\coloneqq\{ i\in[2n] \mid (i\ge r+1\iff r\ge\phi(i))\};\\
R=R(\phi)&\coloneqq\{ i\in[2n] \mid r\ge i, \phi(i)\}.
\end{align*}

If we define the integers $a$, $b$ and $c$ as
\begin{align*}
2a&\coloneqq \#(L), &
2b&\coloneqq \#(B), &
2c&\coloneqq \#(R),
\end{align*}
then it is obvious that $\ell=2a+b$ and $r=b+2c$.

Let us fix an order of the set of matched pairs whose union is $L$.
For the first pair $\{i,j\}$ in $L$, we consider a bordered Legendrian as depicted in Figure~\ref{figure:marked right cusp}, which we call a {\em marked right cusp}. Here the $i$-th and $j$-th edges are made to form a cusp and the resulting crossings are marked as in Definition~\ref{def:marking}. The endpoints on the right retain their labels in $\{r+1,\ldots,r+\ell\}\setminus\{i,j\}$.

Then, we concatenate another marked right cusp for the next pair in the order and so on.
One can easily show that the resulting Legendrian $(\sL_\sfv^\rcusp,\sB_\sfv^\rcusp)$ of type $(\ell,b)$ is invariant under the changes of the order of pairs in the sense of Definition~\ref{Definition:equivalence of marked graphs}.
For examples, according to whether two matched pairs are nested or linked, we have the following sequences of marked Reidemeister moves:
\[
\begin{tikzcd}
\begin{tikzpicture}[baseline=-.5ex]
\begin{scope}
\draw[thick] (-1,0.7) to[out=0,in=180] (0,0.5);
\draw[thick] (-1,0.5) to[out=0,in=180] (0,0.3);
\draw[thick] (-1,0.1) to[out=0,in=180] (0,0.1);
\draw[thick] (-1,-0.1) to[out=0,in=180] (0,-0.1);
\draw[thick] (-1,-0.5) to[out=0,in=180] (0,-0.3);
\draw[thick] (-1,-0.7) to[out=0,in=180] (0,-0.5);
\draw[thick] (-1,0.3) to[out=0,in=180] (-0.4,-0.3) -- (-1,-0.3);
\draw[densely dotted,fill,opacity=0.5] (-0.75,0.1) circle (2.5pt) (-0.65,-0.1) circle (2.5pt);
\end{scope}
\begin{scope}[xshift=1cm]
\draw[thick] (-1,0.5) to[out=0,in=180] (0,0.3);
\draw[thick] (-1,0.1) to[out=0,in=180] (0,0.1);
\draw[thick] (-1,-0.1) to[out=0,in=180] (0,-0.1);
\draw[thick] (-1,-0.5) to[out=0,in=180] (0,-0.3);
\draw[thick] (-1,0.3) to[out=0,in=180] (-0.4,-0.3) -- (-1,-0.3);
\draw[densely dotted,fill,opacity=0.5] (-0.75,0.1) circle (2.5pt) (-0.65,-0.1) circle (2.5pt);
\end{scope}
\end{tikzpicture}\arrow[r,"(\mathrm{0})","(\mathrm{II})"']&
\begin{tikzpicture}[baseline=-.5ex]
\begin{scope}
\draw[thick] (-1,0.7) to[out=0,in=180] (0,0.7) to[out=0,in=180] (1,0.3);
\draw[thick] (-1,0.5) to[out=0,in=180] (-0.6,0.5) to[out=0,in=180] (0.2,-0.5) -- (-0.5,-0.5);
\draw[thick] (-1,0.3) to[out=0,in=180] (-0.4,-0.1) -- (0.2,-0.1) to[out=0,in=180] (0.4,-0.3) -- (-0.5,-0.3) to[out=180,in=0] (-1,-0.3);
\draw[thick] (-1,0.1) to[out=0,in=180] (-0.5,0.3) -- (0,0.3) to[out=0,in=180] (1,0.1);
\draw[thick] (-1,-0.1) to[out=0,in=180] (-0.5,0.1) -- (0,0.1) to[out=0,in=180] (1,-0.1);
\draw[thick] (-1,-0.5) to[out=0,in=180] (-0.5,-0.5);
\draw[thick] (-1,-0.7) to[out=0,in=180] (-0.5,-0.7) to[out=0,in=180] (0,-0.7) to[out=0,in=180] (1,-0.3);
\draw[densely dotted,fill,opacity=0.5] (-0.78,0.18) circle (2.5pt) (-0.66,0.04) circle (2.5pt);
\draw[densely dotted,fill,opacity=0.5] (-0.32,0.3) circle (2.5pt) (-0.23,0.1) circle (2.5pt) (-0.17,-0.1) circle (2.5pt) (-0.1,-0.3) circle (2.5pt);
\end{scope}
\end{tikzpicture}\arrow[r,"(\mathrm{III})"]&
\begin{tikzpicture}[baseline=-.5ex]
\begin{scope}
\draw[thick] (-1,0.7) to[out=0,in=180] (0,0.5) to[out=0,in=180] (1,0.3);
\draw[thick] (-1,0.5) to[out=0,in=180] (-0.4,-0.5) -- (-1,-0.5);
\draw[thick] (-1,0.3) to[out=0,in=180] (0,0.3) to[out=0,in=180] (0.6,-0.3) -- (-1,-0.3);
\draw[thick] (-1,0.1) to[out=0,in=180] (1,0.1);
\draw[thick] (-1,-0.1) to[out=0,in=180] (1,-0.1);
\draw[thick] (-1,-0.7) to[out=0,in=180] (0,-0.5) to[out=0,in=180] (1,-0.3);
\draw[densely dotted,fill,opacity=0.5] (0.25,0.1) circle (2.5pt) (0.35,-0.1) circle (2.5pt);
\draw[densely dotted,fill,opacity=0.5] (-0.77,0.3) circle (2.5pt) (-0.72,0.1) circle (2.5pt) (-0.68,-0.1) circle (2.5pt) (-0.63,-0.3) circle (2.5pt);
\end{scope}
\end{tikzpicture}\\
\begin{tikzpicture}[baseline=-.5ex]
\begin{scope}
\draw[thick] (-1,0.7) to[out=0,in=180] (0,0.5);
\draw[thick] (-1,0.5) to[out=0,in=180] (0,0.3);
\draw[thick] (-1,0.1) to[out=0,in=180] (0,0.1);
\draw[thick] (-1,-0.1) to[out=0,in=180] (0,-0.1);
\draw[thick] (-1,-0.5) to[out=0,in=180] (0,-0.3);
\draw[thick] (-1,-0.7) to[out=0,in=180] (0,-0.5);
\draw[thick] (-1,0.3) to[out=0,in=180] (-0.4,-0.3) -- (-1,-0.3);
\draw[densely dotted,fill,opacity=0.5] (-0.75,0.1) circle (2.5pt) (-0.65,-0.1) circle (2.5pt);
\end{scope}
\begin{scope}[xshift=1cm]
\draw[thick] (-1,0.3) to[out=0,in=180] (0,0.3);
\draw[thick] (-1,0.1) to[out=0,in=180] (0,0.1);
\draw[thick] (-1,-0.3) to[out=0,in=180] (0,-0.1);
\draw[thick] (-1,-0.5) to[out=0,in=180] (0,-0.3);
\draw[thick] (-1,0.5) to[out=0,in=180] (-0.4,-0.1) -- (-1,-0.1);
\draw[densely dotted,fill,opacity=0.5] (-0.75,0.3) circle (2.5pt) (-0.65,0.1) circle (2.5pt);
\end{scope}
\end{tikzpicture}\arrow[r,"(\mathrm{0})","(\mathrm{S})"']&
\begin{tikzpicture}[baseline=-.5ex]
\begin{scope}
\draw[thick] (-1,0.7) to[out=0,in=180] (-0.8,0.7) to[out=0,in=180] (0,-0.1) -- (-0.5,-0.1) to[out=180,in=0] (-1,-0.1);
\draw[thick] (-1,0.5) to[out=0,in=180] (0,0.5) to[out=0,in=180] (1,0.3);
\draw[thick] (-1,0.3) to[out=0,in=180] (-0.4,0.1) -- (0,0.1) to[out=0,in=180] (0.4,-0.3) -- (-0.5,-0.3) to[out=180,in=0] (-1,-0.3);
\draw[thick] (-1,0.1) to[out=0,in=180] (-0.5,0.3) -- (0,0.3) to[out=0,in=180] (1,0.1);
\draw[thick] (-1,-0.5) -- (0,-0.5) to[out=0,in=180] (1,-0.1);
\draw[thick] (-1,-0.7) to[out=0,in=180] (-0.5,-0.7) to[out=0,in=180] (0,-0.7) to[out=0,in=180] (1,-0.3);
\draw[densely dotted,fill,opacity=0.5] (-0.72,0.22) circle (2.5pt);
\draw[densely dotted,fill,opacity=0.5] (-0.5,0.5) circle (2.5pt) (-0.4,0.3) circle (2.5pt) (-0.3,0.1) circle (2.5pt);
\end{scope}
\end{tikzpicture}\arrow[r,"(\mathrm{III})"]&
\begin{tikzpicture}[baseline=-.5ex]
\begin{scope}
\draw[thick] (-1,0.7) to[out=0,in=180] (-0.4,-0.1) -- (-1,-0.1);
\draw[thick] (-1,0.5) to[out=0,in=180] (0,0.5);
\draw[thick] (-1,0.3) to[out=0,in=180] (0,0.3);
\draw[thick] (-1,0.1) to[out=0,in=180] (0,0.1);
\draw[thick] (-1,-0.3) to[out=0,in=180] (0,-0.1);
\draw[thick] (-1,-0.5) to[out=0,in=180] (0,-0.3);
\draw[thick] (-1,-0.7) to[out=0,in=180] (0,-0.5);
\draw[densely dotted,fill,opacity=0.5] (-0.75,0.5) circle (2.5pt) (-0.7,0.3) circle (2.5pt) (-0.65,0.1) circle (2.5pt);
\end{scope}
\begin{scope}[xshift=1cm]
\draw[thick] (-1,0.5) to[out=0,in=180] (0,0.3);
\draw[thick] (-1,0.1) to[out=0,in=180] (0,0.1);
\draw[thick] (-1,-0.3) to[out=0,in=180] (0,-0.1);
\draw[thick] (-1,-0.5) to[out=0,in=180] (0,-0.3);
\draw[thick] (-1,0.3) to[out=0,in=180] (-0.6,-0.1) -- (-1,-0.1);
\draw[densely dotted,fill,opacity=0.5] (-0.8,0.1) circle (2.5pt);
\end{scope}
\end{tikzpicture}
\end{tikzcd}
\]

%

Symmetrically, by using {\em marked left cusps} as depicted in Figure~\ref{figure:marked left cusp}, we use the matchings in $R$ to construct the Legendrian $(\sL_\sfv^\lcusp,\sB_\sfv^\lcusp)$ of type $(b,r)$.

\begin{figure}[ht]
\subfigure[Marked right cusp\label{figure:marked right cusp}]{\makebox[0.4\textwidth]{\includegraphics{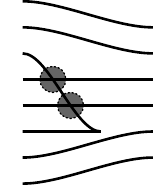}}}
\subfigure[Marked left cusp\label{figure:marked left cusp}]{\makebox[0.4\textwidth]{\includegraphics{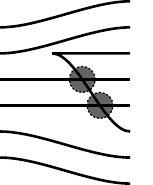}}}
\caption{Marked cusps}
\label{figure:marked cusps}
\end{figure}

Now it remains to construct $(\sL_\sfv^\times,\sB_\sfv^\times)$ of type $(b,b)$ out of the perfect matching $\phi|_B$ on $B$.
Since $\phi|_B$ is a perfect matching between the left edges and the right edges, there is a positive braid $\beta$ on $b$ strands and with a minimal number of crossings (a so called {\em permutation braid}) which induces $\phi|_B$. Note that there is a one-to-one correspondence between permutations and permutation braids.

Recall that any positive $b$-braid $\beta$ can be realized as a bordered Legendrian $\sL_\beta$ of type $(b,b)$ whose arcs have no cusps \cite{Tamas2006}.
In particular, any permutation braid $\beta$ can be regarded as a sub-braid of the {\em half-twist} 
\[
\Delta_b\coloneqq\vcenter{\hbox{\tiny
\begingroup%
  \makeatletter%
  \providecommand\color[2][]{%
    \errmessage{(Inkscape) Color is used for the text in Inkscape, but the package 'color.sty' is not loaded}%
    \renewcommand\color[2][]{}%
  }%
  \providecommand\transparent[1]{%
    \errmessage{(Inkscape) Transparency is used (non-zero) for the text in Inkscape, but the package 'transparent.sty' is not loaded}%
    \renewcommand\transparent[1]{}%
  }%
  \providecommand\rotatebox[2]{#2}%
  \ifx\svgwidth\undefined%
    \setlength{\unitlength}{45.00000073bp}%
    \ifx\svgscale\undefined%
      \relax%
    \else%
      \setlength{\unitlength}{\unitlength * \real{\svgscale}}%
    \fi%
  \else%
    \setlength{\unitlength}{\svgwidth}%
  \fi%
  \global\let\svgwidth\undefined%
  \global\let\svgscale\undefined%
  \makeatother%
  \begin{picture}(1,0.76666407)%
    \put(0,0){\includegraphics[width=\unitlength,page=1]{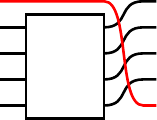}}%
    \put(0.24999999,0.25833074){\color[rgb]{0,0,0}\makebox(0,0)[lb]{\smash{$\Delta_{b-1}$}}}%
  \end{picture}%
\endgroup%
}}~.
\]
in the sense of the following: there is a positive braid $\beta^c$ such that $\beta\beta^c$ is equivalent to $\Delta_b$ as braids. We call $\beta^c$ the {\em right complement} of $\beta$.
Let $\bar\beta^c$ be the \emph{mirror} of $\beta^c$ which is a positive braid obtained by reversing a word representing $\beta^c$.

\begin{remark}
For the notions of complements and mirror, one can refer the paper by El-Rifai and Morton \cite{EM1994}.
\end{remark}

Let $\sL_{\beta^c}$, $\sL_{\bar {\beta^c}}$ be Legendrian permutation braids realizing $\beta^c$ and $\bar{\beta^c}$, respectively, 
and let all crossings in $\sL_{\beta^c}$ and $\sL_{\bar{\beta^c}}$ be marked as in Definition~\ref{def:marking}.
Then we define
\[
\sL_\sfv^\times\coloneqq \sL_{\beta}\cdot\sL_{\beta^c}\cdot\sL_{\bar{\beta^c}}\quad\text{and}\quad
\sB_\sfv^\times\coloneqq \sC(\sL_{\beta^c})\amalg\sC(\sL_{\bar{\beta^c}}),
\]
that is, we leave all crossings in the factor $\sL_\beta$ unmarked (regular).

For concreteness, let us fix standard forms of all permutation braids in the formula above 
inductively, as follows:
\begin{align*}
\sL_{\beta}&=\vcenter{\hbox{
\begingroup%
  \makeatletter%
  \providecommand\color[2][]{%
    \errmessage{(Inkscape) Color is used for the text in Inkscape, but the package 'color.sty' is not loaded}%
    \renewcommand\color[2][]{}%
  }%
  \providecommand\transparent[1]{%
    \errmessage{(Inkscape) Transparency is used (non-zero) for the text in Inkscape, but the package 'transparent.sty' is not loaded}%
    \renewcommand\transparent[1]{}%
  }%
  \providecommand\rotatebox[2]{#2}%
  \ifx\svgwidth\undefined%
    \setlength{\unitlength}{37.50000158bp}%
    \ifx\svgscale\undefined%
      \relax%
    \else%
      \setlength{\unitlength}{\unitlength * \real{\svgscale}}%
    \fi%
  \else%
    \setlength{\unitlength}{\svgwidth}%
  \fi%
  \global\let\svgwidth\undefined%
  \global\let\svgscale\undefined%
  \makeatother%
  \begin{picture}(1,1.01999941)%
    \put(0,0){\includegraphics[width=\unitlength,page=1]{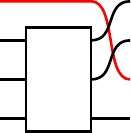}}%
    \put(0.30000001,0.30999919){\color[rgb]{0,0,0}\makebox(0,0)[lb]{\smash{$\sL'_{\beta}$}}}%
  \end{picture}%
\endgroup%
}}~;&
\sL_{\beta^c}&=\vcenter{\hbox{
\begingroup%
  \makeatletter%
  \providecommand\color[2][]{%
    \errmessage{(Inkscape) Color is used for the text in Inkscape, but the package 'color.sty' is not loaded}%
    \renewcommand\color[2][]{}%
  }%
  \providecommand\transparent[1]{%
    \errmessage{(Inkscape) Transparency is used (non-zero) for the text in Inkscape, but the package 'transparent.sty' is not loaded}%
    \renewcommand\transparent[1]{}%
  }%
  \providecommand\rotatebox[2]{#2}%
  \ifx\svgwidth\undefined%
    \setlength{\unitlength}{37.50000262bp}%
    \ifx\svgscale\undefined%
      \relax%
    \else%
      \setlength{\unitlength}{\unitlength * \real{\svgscale}}%
    \fi%
  \else%
    \setlength{\unitlength}{\svgwidth}%
  \fi%
  \global\let\svgwidth\undefined%
  \global\let\svgscale\undefined%
  \makeatother%
  \begin{picture}(1,1.01999991)%
    \put(0,0){\includegraphics[width=\unitlength,page=1]{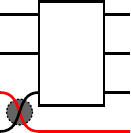}}%
    \put(0.39999997,0.40999844){\color[rgb]{0,0,0}\makebox(0,0)[lb]{\smash{$\sL'_{\beta^c}$}}}%
  \end{picture}%
\endgroup%
}}~;&
\sL_{\bar{\beta^c}}&=\vcenter{\hbox{
\begingroup%
  \makeatletter%
  \providecommand\color[2][]{%
    \errmessage{(Inkscape) Color is used for the text in Inkscape, but the package 'color.sty' is not loaded}%
    \renewcommand\color[2][]{}%
  }%
  \providecommand\transparent[1]{%
    \errmessage{(Inkscape) Transparency is used (non-zero) for the text in Inkscape, but the package 'transparent.sty' is not loaded}%
    \renewcommand\transparent[1]{}%
  }%
  \providecommand\rotatebox[2]{#2}%
  \ifx\svgwidth\undefined%
    \setlength{\unitlength}{37.50000262bp}%
    \ifx\svgscale\undefined%
      \relax%
    \else%
      \setlength{\unitlength}{\unitlength * \real{\svgscale}}%
    \fi%
  \else%
    \setlength{\unitlength}{\svgwidth}%
  \fi%
  \global\let\svgwidth\undefined%
  \global\let\svgscale\undefined%
  \makeatother%
  \begin{picture}(1,1.01999991)%
    \put(0,0){\includegraphics[width=\unitlength,page=1]{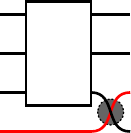}}%
    \put(0.30000005,0.40999844){\color[rgb]{0,0,0}\makebox(0,0)[lb]{\smash{$\sL'_{\bar{\beta^c}}$}}}%
  \end{picture}%
\endgroup%
}}~.&
\end{align*}

In conclusion, for a given perfect matching $\phi$ at $\sfv$, the resulting resolution $(\sL_\sfv^\phi,\sB_\sfv^\phi)$ of $(\sL_\sfv, \emptyset)$ is defined by 
\begin{align*}
(\sL_\sfv^\phi,\sB_\sfv^\phi)&\coloneqq(\sL_\sfv^\rcusp,\sB_\sfv^\rcusp)\cdot (\sL_\sfv^\times,\sB_\sfv^\times)\cdot (\sL_\sfv^\lcusp,\sB_\sfv^\rcusp)\\
&=(\sL_\sfv^\rcusp\cdot\sL_\sfv^\times\cdot\sL_\sfv^\lcusp, \sB_\sfv^\rcusp\amalg\sB_\sfv^\times\amalg \sB_\sfv^\lcusp).
\end{align*}

\begin{remark}
If $\sfv$ is of type $(\ell,0)$ or $(0,r)$, then $\beta$ is a $0$-braid and hence empty. Therefore all crossings in all resolutions are marked.
\end{remark}

\begin{example}
Let us consider a vertex $\sfv$ in a Legendrian graph $\sL$ of type $(5,3)$ as follows:
\[
\vcenter{\hbox{
\begingroup%
  \makeatletter%
  \providecommand\color[2][]{%
    \errmessage{(Inkscape) Color is used for the text in Inkscape, but the package 'color.sty' is not loaded}%
    \renewcommand\color[2][]{}%
  }%
  \providecommand\transparent[1]{%
    \errmessage{(Inkscape) Transparency is used (non-zero) for the text in Inkscape, but the package 'transparent.sty' is not loaded}%
    \renewcommand\transparent[1]{}%
  }%
  \providecommand\rotatebox[2]{#2}%
  \ifx\svgwidth\undefined%
    \setlength{\unitlength}{85.99241481bp}%
    \ifx\svgscale\undefined%
      \relax%
    \else%
      \setlength{\unitlength}{\unitlength * \real{\svgscale}}%
    \fi%
  \else%
    \setlength{\unitlength}{\svgwidth}%
  \fi%
  \global\let\svgwidth\undefined%
  \global\let\svgscale\undefined%
  \makeatother%
  \begin{picture}(1,0.55432148)%
    \put(0,0){\includegraphics[width=\unitlength,page=1]{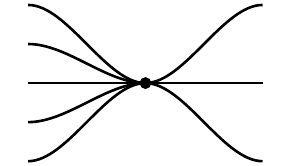}}%
    \put(0.46079023,0.31968035){\color[rgb]{0,0,0}\makebox(0,0)[lb]{\smash{$\sfv$}}}%
    \put(0.90484448,0.53555328){\color[rgb]{0,0,0}\makebox(0,0)[lb]{\smash{\tiny$1$}}}%
    \put(0.90924959,0.27234705){\color[rgb]{0,0,0}\makebox(0,0)[lb]{\smash{\tiny$2$}}}%
    \put(0.91145218,0.00803928){\color[rgb]{0,0,0}\makebox(0,0)[lb]{\smash{\tiny$3$}}}%
    \put(0.00110129,0.53775536){\color[rgb]{0,0,0}\makebox(0,0)[lb]{\smash{\tiny$4$}}}%
    \put(-0,0.39899371){\color[rgb]{0,0,0}\makebox(0,0)[lb]{\smash{\tiny$5$}}}%
    \put(0.00127032,0.00453545){\color[rgb]{0,0,0}\makebox(0,0)[lb]{\smash{\tiny$8$}}}%
    \put(0.00127032,0.26502353){\color[rgb]{0,0,0}\makebox(0,0)[lb]{\smash{\tiny$6$}}}%
    \put(0.00127032,0.13477949){\color[rgb]{0,0,0}\makebox(0,0)[lb]{\smash{\tiny$7$}}}%
  \end{picture}%
\endgroup%
}}
\]
Suppose that a perfect matching $\phi \in \cP_{[8]}$ is given by
\[
\{ \{1,6\}, \{2,4\}, \{3,8\}, \{5,7\} \}.
\]
Then we have $L=\{5,7\}$, $B=\{1,2,3,4,6,8\}$, and $R=\emptyset$.
It is straightforward to check that $\sL^R$ is the trivial braid of type $(3,3)$ and that $(\sL^L,\sB^L)$ becomes the following bordered Legendrian of type $(5,3)$:
\[
\vcenter{\hbox{
\begingroup%
  \makeatletter%
  \providecommand\color[2][]{%
    \errmessage{(Inkscape) Color is used for the text in Inkscape, but the package 'color.sty' is not loaded}%
    \renewcommand\color[2][]{}%
  }%
  \providecommand\transparent[1]{%
    \errmessage{(Inkscape) Transparency is used (non-zero) for the text in Inkscape, but the package 'transparent.sty' is not loaded}%
    \renewcommand\transparent[1]{}%
  }%
  \providecommand\rotatebox[2]{#2}%
  \ifx\svgwidth\undefined%
    \setlength{\unitlength}{53.08730336bp}%
    \ifx\svgscale\undefined%
      \relax%
    \else%
      \setlength{\unitlength}{\unitlength * \real{\svgscale}}%
    \fi%
  \else%
    \setlength{\unitlength}{\svgwidth}%
  \fi%
  \global\let\svgwidth\undefined%
  \global\let\svgscale\undefined%
  \makeatother%
  \begin{picture}(1,0.63694092)%
    \put(0,0){\includegraphics[width=\unitlength,page=1]{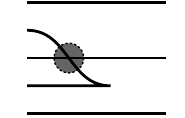}}%
    \put(-0.00312722,0.60831767){\color[rgb]{0,0,0}\makebox(0,0)[lb]{\smash{\tiny$4$}}}%
    \put(-0.00312722,0.4576225){\color[rgb]{0,0,0}\makebox(0,0)[lb]{\smash{\tiny$5$}}}%
    \put(-0.00312722,0.30692734){\color[rgb]{0,0,0}\makebox(0,0)[lb]{\smash{\tiny$6$}}}%
    \put(-0.00312722,0.15623218){\color[rgb]{0,0,0}\makebox(0,0)[lb]{\smash{\tiny$7$}}}%
    \put(-0.00312722,0.00553702){\color[rgb]{0,0,0}\makebox(0,0)[lb]{\smash{\tiny$8$}}}%
  \end{picture}%
\endgroup%
}}
\]
Since $\phi|_B=\{\{1,6\}, \{2,4\}, \{3,8\}\}$, the following claims are easy to check:
\begin{align*}
\sL_{\beta}&=\vcenter{\hbox{
\begingroup%
  \makeatletter%
  \providecommand\color[2][]{%
    \errmessage{(Inkscape) Color is used for the text in Inkscape, but the package 'color.sty' is not loaded}%
    \renewcommand\color[2][]{}%
  }%
  \providecommand\transparent[1]{%
    \errmessage{(Inkscape) Transparency is used (non-zero) for the text in Inkscape, but the package 'transparent.sty' is not loaded}%
    \renewcommand\transparent[1]{}%
  }%
  \providecommand\rotatebox[2]{#2}%
  \ifx\svgwidth\undefined%
    \setlength{\unitlength}{33.88730384bp}%
    \ifx\svgscale\undefined%
      \relax%
    \else%
      \setlength{\unitlength}{\unitlength * \real{\svgscale}}%
    \fi%
  \else%
    \setlength{\unitlength}{\svgwidth}%
  \fi%
  \global\let\svgwidth\undefined%
  \global\let\svgscale\undefined%
  \makeatother%
  \begin{picture}(1,0.99782136)%
    \put(0,0){\includegraphics[width=\unitlength,page=1]{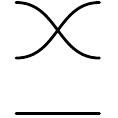}}%
    \put(0.89219207,0.0086742){\color[rgb]{0,0,0}\makebox(0,0)[lb]{\smash{\tiny$3$}}}%
    \put(0.89219207,0.48082742){\color[rgb]{0,0,0}\makebox(0,0)[lb]{\smash{\tiny$2$}}}%
    \put(0.89219207,0.95298064){\color[rgb]{0,0,0}\makebox(0,0)[lb]{\smash{\tiny$1$}}}%
    \put(-0.00489905,0.95298064){\color[rgb]{0,0,0}\makebox(0,0)[lb]{\smash{\tiny$4$}}}%
    \put(-0.00489905,0.48082742){\color[rgb]{0,0,0}\makebox(0,0)[lb]{\smash{\tiny$6$}}}%
    \put(-0.00489905,0.0086742){\color[rgb]{0,0,0}\makebox(0,0)[lb]{\smash{\tiny$8$}}}%
  \end{picture}%
\endgroup%
}}~;&
(\sL_{\beta^c},\sC(\sL_{\beta^c}))&=\vcenter{\hbox{
\begingroup%
  \makeatletter%
  \providecommand\color[2][]{%
    \errmessage{(Inkscape) Color is used for the text in Inkscape, but the package 'color.sty' is not loaded}%
    \renewcommand\color[2][]{}%
  }%
  \providecommand\transparent[1]{%
    \errmessage{(Inkscape) Transparency is used (non-zero) for the text in Inkscape, but the package 'transparent.sty' is not loaded}%
    \renewcommand\transparent[1]{}%
  }%
  \providecommand\rotatebox[2]{#2}%
  \ifx\svgwidth\undefined%
    \setlength{\unitlength}{34.11590898bp}%
    \ifx\svgscale\undefined%
      \relax%
    \else%
      \setlength{\unitlength}{\unitlength * \real{\svgscale}}%
    \fi%
  \else%
    \setlength{\unitlength}{\svgwidth}%
  \fi%
  \global\let\svgwidth\undefined%
  \global\let\svgscale\undefined%
  \makeatother%
  \begin{picture}(1,0.99304428)%
    \put(0,0){\includegraphics[width=\unitlength,page=1]{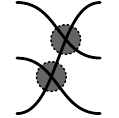}}%
    \put(0.89291447,0.01052523){\color[rgb]{0,0,0}\makebox(0,0)[lb]{\smash{\tiny$3$}}}%
    \put(0.89291447,0.47951463){\color[rgb]{0,0,0}\makebox(0,0)[lb]{\smash{\tiny$2$}}}%
    \put(0.89291447,0.94850403){\color[rgb]{0,0,0}\makebox(0,0)[lb]{\smash{\tiny$1$}}}%
    \put(-0.00486622,0.00861608){\color[rgb]{0,0,0}\makebox(0,0)[lb]{\smash{\tiny$3$}}}%
    \put(-0.00486622,0.47760547){\color[rgb]{0,0,0}\makebox(0,0)[lb]{\smash{\tiny$2$}}}%
    \put(-0.00486622,0.94659475){\color[rgb]{0,0,0}\makebox(0,0)[lb]{\smash{\tiny$1$}}}%
  \end{picture}%
\endgroup%
}}~;&
(\sL_{\bar{\beta^c}},\sC(\sL_{\bar{\beta^c}}))&=\vcenter{\hbox{
\begingroup%
  \makeatletter%
  \providecommand\color[2][]{%
    \errmessage{(Inkscape) Color is used for the text in Inkscape, but the package 'color.sty' is not loaded}%
    \renewcommand\color[2][]{}%
  }%
  \providecommand\transparent[1]{%
    \errmessage{(Inkscape) Transparency is used (non-zero) for the text in Inkscape, but the package 'transparent.sty' is not loaded}%
    \renewcommand\transparent[1]{}%
  }%
  \providecommand\rotatebox[2]{#2}%
  \ifx\svgwidth\undefined%
    \setlength{\unitlength}{34.29585689bp}%
    \ifx\svgscale\undefined%
      \relax%
    \else%
      \setlength{\unitlength}{\unitlength * \real{\svgscale}}%
    \fi%
  \else%
    \setlength{\unitlength}{\svgwidth}%
  \fi%
  \global\let\svgwidth\undefined%
  \global\let\svgscale\undefined%
  \makeatother%
  \begin{picture}(1,0.98982399)%
    \put(0,0){\includegraphics[width=\unitlength,page=1]{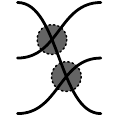}}%
    \put(0.89347634,0.01246015){\color[rgb]{0,0,0}\makebox(0,0)[lb]{\smash{\tiny$3$}}}%
    \put(0.89347634,0.47898879){\color[rgb]{0,0,0}\makebox(0,0)[lb]{\smash{\tiny$2$}}}%
    \put(0.89347634,0.94551744){\color[rgb]{0,0,0}\makebox(0,0)[lb]{\smash{\tiny$1$}}}%
    \put(-0.00484069,0.00857087){\color[rgb]{0,0,0}\makebox(0,0)[lb]{\smash{\tiny$3$}}}%
    \put(-0.00484069,0.47509951){\color[rgb]{0,0,0}\makebox(0,0)[lb]{\smash{\tiny$2$}}}%
    \put(-0.00484069,0.94162828){\color[rgb]{0,0,0}\makebox(0,0)[lb]{\smash{\tiny$1$}}}%
  \end{picture}%
\endgroup%
}}~.&
\end{align*}
Thus the resulting resolution $\sL^L\cdot \sL^B \cdot\sL^R=\sL^L\cdot( \sL_{\beta} \sL_{\beta^c} \sL_{\bar{\beta^c}})\cdot \sL^R$ with markings becomes 
\begin{align*}
(\sL_\sfv,\emptyset)&\stackrel{\phi}\longmapsto (\sL_\sfv^\phi, \sB_\sfv^\phi)=\vcenter{\hbox{\tiny
\begingroup%
  \makeatletter%
  \providecommand\color[2][]{%
    \errmessage{(Inkscape) Color is used for the text in Inkscape, but the package 'color.sty' is not loaded}%
    \renewcommand\color[2][]{}%
  }%
  \providecommand\transparent[1]{%
    \errmessage{(Inkscape) Transparency is used (non-zero) for the text in Inkscape, but the package 'transparent.sty' is not loaded}%
    \renewcommand\transparent[1]{}%
  }%
  \providecommand\rotatebox[2]{#2}%
  \ifx\svgwidth\undefined%
    \setlength{\unitlength}{154.76934428bp}%
    \ifx\svgscale\undefined%
      \relax%
    \else%
      \setlength{\unitlength}{\unitlength * \real{\svgscale}}%
    \fi%
  \else%
    \setlength{\unitlength}{\svgwidth}%
  \fi%
  \global\let\svgwidth\undefined%
  \global\let\svgscale\undefined%
  \makeatother%
  \begin{picture}(1,0.20482179)%
    \put(0,0){\includegraphics[width=\unitlength,page=1]{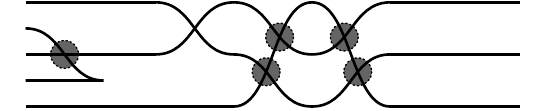}}%
    \put(0.97787037,0.00178054){\color[rgb]{0,0,0}\makebox(0,0)[lb]{\smash{$3$}}}%
    \put(0.97787037,0.09869896){\color[rgb]{0,0,0}\makebox(0,0)[lb]{\smash{$2$}}}%
    \put(0.97787037,0.19561738){\color[rgb]{0,0,0}\makebox(0,0)[lb]{\smash{$1$}}}%
    \put(-0.00100562,0.19561738){\color[rgb]{0,0,0}\makebox(0,0)[lb]{\smash{$4$}}}%
    \put(-0.00100562,0.14715817){\color[rgb]{0,0,0}\makebox(0,0)[lb]{\smash{$5$}}}%
    \put(-0.00100562,0.09869896){\color[rgb]{0,0,0}\makebox(0,0)[lb]{\smash{$6$}}}%
    \put(-0.00100562,0.05023975){\color[rgb]{0,0,0}\makebox(0,0)[lb]{\smash{$7$}}}%
    \put(-0.00100562,0.00178054){\color[rgb]{0,0,0}\makebox(0,0)[lb]{\smash{$8$}}}%
  \end{picture}%
\endgroup%
}}~.
\end{align*}
\end{example}

In general, for a marked Legendrian graph $(\sL,\sB)$, let $\phi$ be a perfect matching on the set of half-edges of a vertex $\sfv$, which we will call a perfect matching of $\sfv$. Then one can define the resolution of $(\sL,\sB)$ with respect to $\phi$ by the replacement of a small neighborhood of $\sfv$ with the resolution diagram.
\begin{align*}
(\sL,\sB)&\stackrel{\phi}{\longrightarrow}(\sL^{\phi}, \sB^{\phi}),&
\sB^{\phi}&=\sB\amalg\sB_\sfv^{\phi}.
\end{align*}

It is important to note that the result $\sL^{\phi}$ of the resolution is not necessarily equipped with a Maslov potential.
In the above example, unless the Maslov potentials of the pairs of arcs comprising $L$ --- 1st and 6th, 2nd and 4th, 3rd and 8th --- coincide, and the Maslov potentials of the 5th and 7th arcs have difference 1, 
a Maslov potential for $\sL$ will not extend to one for $\sL^\phi$.

\begin{definition}
Let $(\sL,\mu)$ be a Legendrian graph with Maslov potential. For $\rho\in\Z$, a perfect matching $\phi$ on a vertex $\sfv$ of type $(\ell,r)$ is {\em $\rho$-graded} with respect to $\mu$ if
\begin{enumerate}
\item for $\{i,\phi(i)\}\in\phi$ with $\phi(i)<i\le r$ or $i>\phi(i)>r$, the difference of Maslov potentials for the $i$-th and $\phi(i)$-th arcs is $1$ modulo $\rho$:
\[
\mu(\sfe_i)-\mu(\sfe_{\phi(i)})=1\in \fR/\rho\fR.
\]
\item for $\{i,\phi(i)\}\in\phi$ with $i\le r<\phi(i)$ or $\phi(i)\le r<i$, the difference of Maslov potentials for $i$-th and $\phi(i)$-th arcs is divisible by $\rho$:
\[
\mu(\sfe_i)-\mu(\sfe_{\phi(i)})=0\in \fR/\rho\fR.
\]
\end{enumerate}

We say that a resolution $\phi$ is a {\em$\rho$-graded resolution} if $\phi$ is {\em$\rho$-graded}, and denote the set of $\rho$-graded matchings at $\sfv$ by 
\[
\cP^\rho_\sfv\coloneqq\{
\phi\in\cP_\sfv\mid \phi\text{ is $\rho$-graded}\}.
\]
\end{definition}

\begin{remark}
If $\rho=1$, then the Maslov potential becomes trivial and all possible perfect matchings are $1$-graded.
\end{remark}

\begin{lemma}
Let $(\sL,\mu)$ be a Legendrian graph with Maslov potential and $\sfv$ be a vertex. Any $\rho$-graded resolution $\phi$ on $\sfv$ admits an induced Maslov potential $\mu^{\phi}$.
\end{lemma}

\begin{proof}
Since $\phi$ is $\rho$-graded, the arcs in each matched pair have Maslov potentials which are either the same or differ by 1 according to whether they form a smooth arc or a cusp after the resolution. This implies that $\mu$ induces a Maslov potential $\mu^{\phi}$ on the resolution $\sL^{\phi}$.
\end{proof}

Therefore, a $\rho$-graded resolution $\phi$ on a vertex $\sfv$ of a marked bordered Legendrian graph $\bL=((\sL,\mu),\sB)$ gives us $\bL^{\phi}$ defined as follows:
\[
\bL=((\sL, \mu), \sB)\stackrel{\phi}{\longrightarrow}
\bL^{\phi}\coloneqq((\sL^{\phi}, \mu^{\phi}), \sB^{\phi}).
\]

Let us denote the set of all collections of $\rho$-graded matchings $\Phi=\{\phi_\sfv\}_{\sfv\in\sV_\sL}$, one for each vertex, by $\cP^\rho_\bL$. In symbols,
\[
\cP^\rho_\bL\coloneqq\{ \Phi=\{\phi_\sfv\}_{\sfv\in\sV_\sL} \mid \forall i\,\phi_{\sfv_i}\in\cP^\rho_{\sfv_i}\}\simeq
\prod_{\sfv\in\sV_\sL} \cP^\rho_\sfv.
\]
Then one can define a simultaneous resolution $\bL^\Phi$ of $\bL$, via $\Phi$, in a canonical way as follows:
\begin{align*}
\bL^{\Phi}&\coloneqq
(\cdots((\bL^{\phi_{\sfv_1}})^{\phi_{\sfv_2}})\cdots)^{\phi_{\sfv_k}},& \text{where }
\sV&=\{\sfv_1,\dots,\sfv_k\}.
\end{align*}

\begin{definition}[Full resolutions]
Let $\bL=((\sL,\mu),\sB)$ be a marked bordered Legendrian graph. We define the set $\tilde \bL$ of all marked $\rho$-resolutions of $\bL$, simultaneously at all vertices, to consist of the following bordered Legendrians without vertices:
\[
\tilde\bL\coloneqq
\coprod_{\Phi\in\cP^\rho_\bL} \bL^\Phi.
\]
\end{definition}

\begin{example}\label{ex:resolution of 4-valent}
There are five types of four-valent vertices, each of which has three ($1$-graded) resolutions as follows:
\begin{align*}
(4,0)&\longmapsto \vcenter{\hbox{\includegraphics{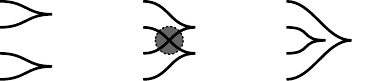}}}\\
(3,1)&\longmapsto \vcenter{\hbox{\includegraphics{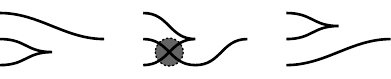}}}\\
(2,2)&\longmapsto \vcenter{\hbox{\includegraphics{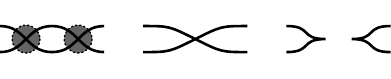}}}\\
(1,3)&\longmapsto \vcenter{\hbox{\includegraphics{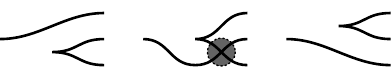}}}\\
(0,4)&\longmapsto \vcenter{\hbox{\includegraphics{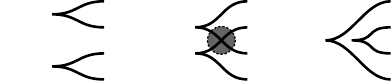}}}
\end{align*}
\end{example}

\begin{example}
For the readers' convenience, we list local resolutions for some vertices of valency six. 
Note that all types $(\ell,r)$ with $\ell+r=6$ have 15 possible ($1$-graded) resolutions which coincide with the number of possible pairings of six edges.
Here is the list of resolutions for the type $(0,6)$:
\begin{align*}
\vcenter{\hbox{\scalebox{1.5}{
\begingroup%
  \makeatletter%
  \providecommand\color[2][]{%
    \errmessage{(Inkscape) Color is used for the text in Inkscape, but the package 'color.sty' is not loaded}%
    \renewcommand\color[2][]{}%
  }%
  \providecommand\transparent[1]{%
    \errmessage{(Inkscape) Transparency is used (non-zero) for the text in Inkscape, but the package 'transparent.sty' is not loaded}%
    \renewcommand\transparent[1]{}%
  }%
  \providecommand\rotatebox[2]{#2}%
  \ifx\svgwidth\undefined%
    \setlength{\unitlength}{128.4676797bp}%
    \ifx\svgscale\undefined%
      \relax%
    \else%
      \setlength{\unitlength}{\unitlength * \real{\svgscale}}%
    \fi%
  \else%
    \setlength{\unitlength}{\svgwidth}%
  \fi%
  \global\let\svgwidth\undefined%
  \global\let\svgscale\undefined%
  \makeatother%
  \begin{picture}(1,1.06352258)%
    \put(0,0){\includegraphics[width=\unitlength,page=1]{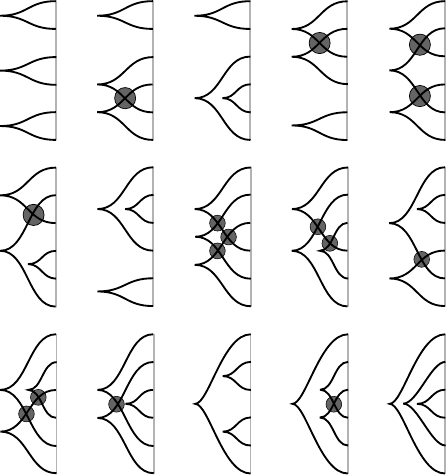}}%
  \end{picture}%
\endgroup%
}}}
\end{align*}
The following is the resolution list for type $(3,3)$:
\begin{align*}
\vcenter{\hbox{
\begingroup%
  \makeatletter%
  \providecommand\color[2][]{%
    \errmessage{(Inkscape) Color is used for the text in Inkscape, but the package 'color.sty' is not loaded}%
    \renewcommand\color[2][]{}%
  }%
  \providecommand\transparent[1]{%
    \errmessage{(Inkscape) Transparency is used (non-zero) for the text in Inkscape, but the package 'transparent.sty' is not loaded}%
    \renewcommand\transparent[1]{}%
  }%
  \providecommand\rotatebox[2]{#2}%
  \ifx\svgwidth\undefined%
    \setlength{\unitlength}{288.80008bp}%
    \ifx\svgscale\undefined%
      \relax%
    \else%
      \setlength{\unitlength}{\unitlength * \real{\svgscale}}%
    \fi%
  \else%
    \setlength{\unitlength}{\svgwidth}%
  \fi%
  \global\let\svgwidth\undefined%
  \global\let\svgscale\undefined%
  \makeatother%
  \begin{picture}(1,0.39063907)%
    \put(0,0){\includegraphics[width=\unitlength,page=1]{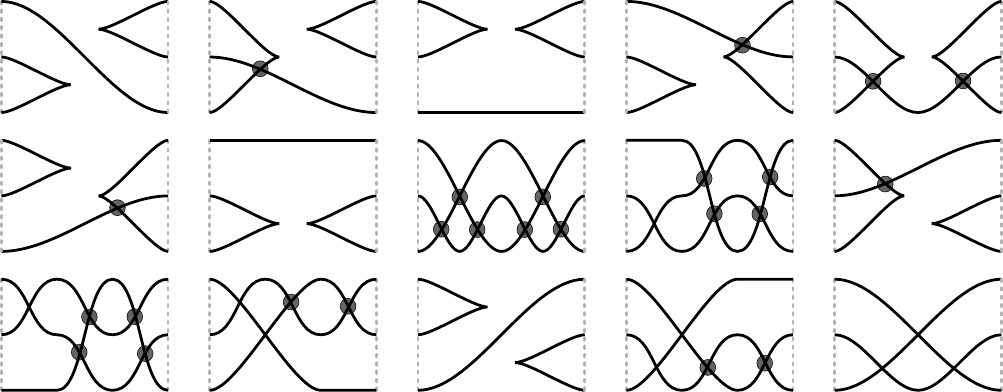}}%
  \end{picture}%
\endgroup%
}}
\end{align*}
We leave it for the reader to check other types, which can be done easily.
\end{example}

\section{Rulings for bordered Legendrian links and graphs}

\subsection{Rulings for marked bordered Legendrian graphs}
In this section we will define $\rho$-graded normal rulings for marked bordered Legendrian graphs. To this end, we first define normal rulings for bordered Legendrian links.

Let $\Lambda$ be a bordered Legendrian link and $\fr\subset\sC$ be a subset of its crossings. We denote the $0$-resolution of $\Lambda$ at every crossing $\sfc$ in $\fr$ by $\Lambda_\fr$:
\[
\begin{tikzpicture}[scale=0.5]
\begin{scope}[xshift=-0.5cm]
\draw (-1,0) node[left] {$\Lambda=$};
\draw[thick] (-1,-1) -- (1,1);
\draw[line width=10pt,white] (-1,1) -- (1,-1);
\draw[thick] (-1,1) -- (1,-1);
\draw (0,0) node[below] {$\sfc$};
\end{scope}
\draw[->] (1,0) -- (2,0);
\begin{scope}[xshift=3.5cm]
\draw[thick] (-1,1) -- (-0.5,0.5) to[out=-45,in=225] (0.5,0.5) -- (1,1);
\draw[thick] (-1,-1) -- (-0.5,-0.5) to[out=45,in=-225] (0.5,-0.5) -- (1,-1);
\draw (1,0) node[right] {$=\Lambda_\sfc$.};
\end{scope}
\end{tikzpicture}
\]

\begin{definition}
Let $\mathbf{\Lambda}=((\Lambda,\mu),\sB)$ be a marked bordered Legendrian link of type $(\ell,r)$, and $(\phi,\psi)$ be a pair of matchings in $\cP^\rho_{[\ell]} \times \cP^\rho_{[r]}$. A {\em $\rho$-graded normal ruling} of $\bL$ with $(\phi,\psi)$ is a subset $\fr$ of $\sC\setminus \sB$ with {\em decomposition $S_\fr$} such that
\begin{enumerate}
\item $|\sfc|\in \rho\fR$ for any $\sfc\in \fr$;
\item $S_\fr$ {\em decomposes} the 0-resolution $\Lambda_\fr$ into {\em eyes}, {\em left half-eyes}, {\em right half-eyes} and {\em parallels}, which are bordered Legendrian links of type $(0,0)$, $(0,2)$, $(2,0)$, and $(2,2)$, respectively, looking as follows:
\begin{align*}
\vcenter{\hbox{\includegraphics{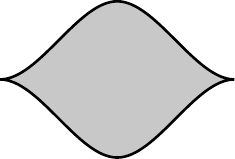}}}\qquad
\vcenter{\hbox{\includegraphics{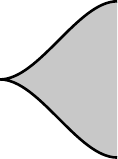}}}\qquad
\vcenter{\hbox{\includegraphics{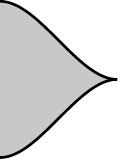}}}\qquad
\vcenter{\hbox{\includegraphics{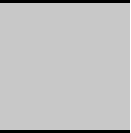}}}
\end{align*}

\item at each $\sfc\in \fr$, a {\em non-interlacing} condition is satisfied.
The following are the only possible decomposition configurations in narrow vertical regions containg some $\sfc\in \fr$:
\begin{align*}
\vcenter{\hbox{\includegraphics{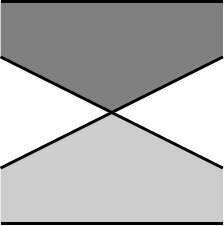}}}\qquad
\vcenter{\hbox{\includegraphics{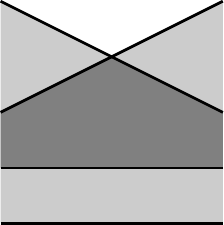}}}\qquad
\vcenter{\hbox{\includegraphics{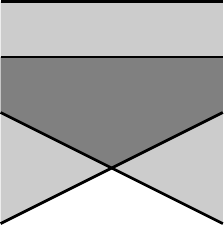}}}
\end{align*}
\item
$\phi=\iota_L^*(S_\fr)$ and $\psi=\iota^*_R(S_\fr)$.
\end{enumerate}
Let us denote the set of such $\rho$-graded normal rulings by $\bR^\rho_{\mathbf{\Lambda}}(\phi,\psi)$, and simply denote its element $(\fr, S_\fr)$ by $S_\fr$.
\end{definition}

Now we consider marked bordered Legendrian graphs and define $\rho$-graded normal rulings as rulings of resolutions of the graph, as follows:

\begin{definition}
Let $\bL$ be a marked bordered Legendrian graph of type $(\ell,r)$, and let $(\phi,\psi)$ be a pair of matchings in $\cP^\rho_{[\ell]} \times \cP^\rho_{[r]}$. 
Then we define the set of $\rho$-graded normal rulings of $\bL$ as follows:
\begin{align*}
\bR^\rho_\bL(\phi,\psi) \coloneqq \coprod_{\mathbf{\Lambda}\in\tilde\bL} \bR^\rho_{\mathbf{\Lambda}}(\phi,\psi),&&
\bR^\rho_\bL\coloneqq \coprod_{(\phi,\psi)\in \cP^\rho_{[\ell]} \times \cP^\rho_{[r]}}\bR^\rho_\bL(\phi,\psi).
\end{align*}
\end{definition}

Now for marked bordered Legendrian graphs $\bL_1$ and $\bL_2$ of type $(\ell,r)$ and $(r,s)$, respectively, we have the following lemma.

\begin{lemma}\label{lemma:fiber product}
Let $\bL=\bL_1\cdot\bL_2$. Then the set of $\rho$-graded normal rulings of $\bL$ is the following fiber product:
\[
\begin{tikzcd}[column sep=4pc]
\bR^\rho_\bL(\phi,\psi)\arrow[r,dashed]\ar[d,dashed] & \bR^\rho_{\bL_2}(\varphi,\psi)\arrow[d,"\iota^*_L"]\\
\bR^\rho_{\bL_1}(\phi,\varphi)\arrow[r,"\iota^*_R"] & \cP_{[r]}
\end{tikzcd}.
\]
\end{lemma}
\begin{proof}
Let $S_{\fr_1}$ and $S_{\fr_2}$ be normal rulings for $\bL_1$ and $\bL_2$, respectively. Then they can be glued in an obvious way if and only if the perfect matchings given by $\iota_R^*(S_{\fr_1})$ and $\iota_L^*(S_{\fr_2})$ coincide.

Conversely, the two maps from $\bR^\rho_\bL(\phi,\psi)$ are given by the restrictions and the universal property of the fiber product, which completes the proof.
\end{proof}

\begin{lemma}\label{lemma:vertex bijections}
Let $\ell$ and $r$ be even. Then there are canonical bijections
\[
\bR^\rho_{\iota_R^*(\mathbf{0}_\ell)}\simeq \bR^\rho_{\mathbf{0}_\ell}\quad\text{and}\quad
\bR^\rho_{\iota_L^*(\pmb{\infty}_r)}\simeq \bR^\rho_{\mathbf{\infty}_r}.
\]
\end{lemma}

\begin{proof}
Note that $\iota_R^*(\mathbf{0}_\ell)$ is the trivial Legendrian graph of $\ell$ strands, and it is easy to see that the set of $\rho$-graded normal rulings is the same as the set of $\rho$-graded matchings on the vertex $\mathbf{0}$
\[
\bR^\rho_{\iota_R^*(\mathbf{0}_\ell)} \simeq \cP_{\mathbf{0}}^\rho \simeq  \cP_{[\ell]}^\rho.
\]

On the other hand, $\bR^\rho_{\mathbf{0}_\ell}$ is the disjoint union of $\bR^\rho_{\mathbf{0}_\ell^\phi}(-,\eta)$ over all $\phi\in\cP_{\mathbf{0}}^\rho$ and $(-,\eta)\in \cP_{[0]}^\rho\times\cP_{[\ell]}^\rho$.
As seen earlier, all crossings in $\mathbf{0}_\ell^\phi$ are marked and there are no other choices of crossing for normal rulings, but $\fr=\emptyset$, so that the $0$-resolution on $\fr$ becomes $\mathbf{0}_\ell^\phi$ itself.
Since $\mathbf{0}_\ell^\phi$ is canonically decomposed into $\frac\ell2$ left half-eyes, $\fr=\emptyset$ becomes a normal ruling.
Namely, for each $\rho$-graded matching $\phi$, the ruling $\bR^\rho_{\mathbf{0}_\ell^\phi}(-,\eta)$ has a unique element only when $\phi=\eta$.
Hence we have
\begin{align*}
\bR^\rho_{\mathbf{0}_\ell} &=\coprod_{(-,\eta)\in \cP_{[0]}^\rho\times\cP_{[\ell]}^\rho} \left( \coprod_{\phi\in\cP_{\mathbf{0}}^\rho} \bR^\rho_{\mathbf{0}_\ell^\phi}(-,\eta)\right)= \coprod_{\phi\in\cP_{\mathbf{0}}^\rho} \bR^\rho_{\mathbf{0}_\ell^\phi}(-,\phi) \simeq \cP_{\mathbf{0}}^\rho\simeq \bR^\rho_{\iota_R^*(\mathbf{0}_\ell)}.
\end{align*}

Essentially the same proof applies for $\pmb{\infty}_r$ and we are done.
\end{proof}

\begin{corollary}\label{corollary:rulings for bordered Legendrian graphs}
Let $\bL$ be a marked bordered Legendrian graph of type $(\ell,r)$, where $\ell$ and $r$ are even. Then there is a bijection
$\bR^\rho_\bL\simeq \bR^\rho_{\hat\bL}$
given by 
$S_{\fr}\mapsto S_{\fr}$.
\end{corollary}

\begin{proof}
This is a direct consequence of Lemma~\ref{lemma:fiber product} and Lemma~\ref{lemma:vertex bijections}.
\end{proof}

\begin{definition}
Let $S_{\fr}\in\bR^\rho_\bL$. The {\em weight $\wt(S_{\fr})$} of $S_{\fr}$ is defined as
\begin{align*}
\wt(S_{\fr})&\coloneqq z^{n(S_{\fr})},&
&\text{where }n(S_{\fr})\coloneqq\#(\fr)-\frac{\#(\lcusp_{\sL_\fr})+\#(\rcusp_{\sL_\fr})}2\in\frac12\Z
\end{align*}
and $\sL_\fr$ is a bordered Legendrian link obtained by the $0$-resolution on $\fr$.
\end{definition}

It is obvious that 
\[
n(S_{\fr})=\#(\fr)- \#(\{\text{eyes in }S_\fr\})-\frac 12\#(\{\text{half-eyes in }S_\fr\}).
\]

\begin{definition}
The {\em $\rho$-graded ruling polynomial $R^\rho_\bL(\phi,\psi)$} of $\bL$ is the sum of weights of $\rho$-graded normal rulings of type $(\phi,\psi)$:
\[
R^\rho_\bL(\phi,\psi)\coloneqq
\sum_{S_{\fr}\in\bR^\rho_\bL(\phi,\psi)} \wt(S_{\fr})\in\Z[z^{\pm\frac12}].
\]
\end{definition}

\begin{corollary}\label{corollary:multiplicative under concatenation}
Suppose that $\bL_1$ and $\bL_2$ are of types $(\ell,r)$ and $(r,s)$, respectively and $\bL=\bL_1\cdot\bL_2$. Then
\[
R^\rho_\bL(\phi,\varphi)=\sum_{\psi\in\cP_{[r]}}
R^\rho_{\bL_1}(\phi,\psi)
R^\rho_{\bL_2}(\psi,\varphi).
\]
\end{corollary}
\begin{proof}
This is obvious by Lemma~\ref{lemma:fiber product}.
\end{proof}

\begin{remark}
Evidently, $R^\rho_\bL$ can be regarded as a linear transformation from $R^\rho_{\iota_L^*(\bL)}$ to $R^\rho_{\iota_R^*(\bL)}$, whose $(\phi,\psi)$-entry is given precisely by $R^\rho_\bL(\phi,\psi)$.
\end{remark}

The following theorem will be proven later.

\begin{theorem}[Invariance theorem]
\label{theorem:invariance for graphs}
The set of $\rho$-graded normal rulings $\bR^\rho_\bL(\phi,\psi)$ of type $(\phi,\psi)$ transforms bijectively under equivalences of marked bordered Legendrian graphs.
In particular, the polynomial $R^\rho_\bL(\phi,\psi)$ is invariant.
\end{theorem}

More precisely, there is a weight-preserving bijection between the sets $\bR^\rho_\bL(\phi,\psi)$ and $\bR^\rho_{\bL'}(\phi,\psi)$ of $\rho$-graded normal rulings for any two equivalent marked bordered Legendrian graphs $\bL$ and $\bL'$.


By the following corollary, whose proof is obvious, the $\rho$-graded ruling polynomial defined above recovers the earlier notion for Legendrian links.

\begin{corollary}
Let $\bL=(\sL,\mu)$ be a Legendrian link with a Maslov potential. Then the $\rho$-graded ruling polynomial $R^\rho_\bL(\emptyset,\emptyset)$ is the same as the $\rho$-graded ruling polynomial defined by Chekanov \cite{Chekanov2002}.
\end{corollary}

\begin{example}\label{ex:pinched trefoil}
Let us consider the following front diagram $\sL$ of a Legendrian graph $L$ having one $4$-valent vertex. Here the other double point of the projection is not a vertex but a regular crossing.
\begin{align*}
\sL=\vcenter{\hbox{\scalebox{1.5}{
\begingroup%
  \makeatletter%
  \providecommand\color[2][]{%
    \errmessage{(Inkscape) Color is used for the text in Inkscape, but the package 'color.sty' is not loaded}%
    \renewcommand\color[2][]{}%
  }%
  \providecommand\transparent[1]{%
    \errmessage{(Inkscape) Transparency is used (non-zero) for the text in Inkscape, but the package 'transparent.sty' is not loaded}%
    \renewcommand\transparent[1]{}%
  }%
  \providecommand\rotatebox[2]{#2}%
  \ifx\svgwidth\undefined%
    \setlength{\unitlength}{47.99999987bp}%
    \ifx\svgscale\undefined%
      \relax%
    \else%
      \setlength{\unitlength}{\unitlength * \real{\svgscale}}%
    \fi%
  \else%
    \setlength{\unitlength}{\svgwidth}%
  \fi%
  \global\let\svgwidth\undefined%
  \global\let\svgscale\undefined%
  \makeatother%
  \begin{picture}(1,0.51116667)%
    \put(0,0){\includegraphics[width=\unitlength,page=1]{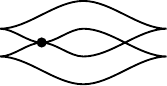}}%
    \put(0.29947885,0.30245929){\color[rgb]{0,0,0}\makebox(0,0)[lb]{\smash{$\scriptscriptstyle1$}}}%
    \put(0.29947148,0.1462091){\color[rgb]{0,0,0}\makebox(0,0)[lb]{\smash{$\scriptscriptstyle2$}}}%
    \put(0.14605767,0.14620882){\color[rgb]{0,0,0}\makebox(0,0)[lb]{\smash{$\scriptscriptstyle4$}}}%
    \put(0.14608348,0.30245831){\color[rgb]{0,0,0}\makebox(0,0)[lb]{\smash{$\scriptscriptstyle3$}}}%
  \end{picture}%
\endgroup%
}}}
\end{align*}
The possible resolutions are as follows:
\begin{align*}
\sL_-=\vcenter{\hbox{\scalebox{1.5}{
\begingroup%
  \makeatletter%
  \providecommand\color[2][]{%
    \errmessage{(Inkscape) Color is used for the text in Inkscape, but the package 'color.sty' is not loaded}%
    \renewcommand\color[2][]{}%
  }%
  \providecommand\transparent[1]{%
    \errmessage{(Inkscape) Transparency is used (non-zero) for the text in Inkscape, but the package 'transparent.sty' is not loaded}%
    \renewcommand\transparent[1]{}%
  }%
  \providecommand\rotatebox[2]{#2}%
  \ifx\svgwidth\undefined%
    \setlength{\unitlength}{48bp}%
    \ifx\svgscale\undefined%
      \relax%
    \else%
      \setlength{\unitlength}{\unitlength * \real{\svgscale}}%
    \fi%
  \else%
    \setlength{\unitlength}{\svgwidth}%
  \fi%
  \global\let\svgwidth\undefined%
  \global\let\svgscale\undefined%
  \makeatother%
  \begin{picture}(1,0.51116667)%
    \put(0,0){\includegraphics[width=\unitlength,page=1]{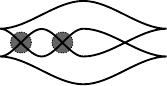}}%
  \end{picture}%
\endgroup%
}}}~,\qquad  \sL_0=\vcenter{\hbox{\scalebox{1.5}{
\begingroup%
  \makeatletter%
  \providecommand\color[2][]{%
    \errmessage{(Inkscape) Color is used for the text in Inkscape, but the package 'color.sty' is not loaded}%
    \renewcommand\color[2][]{}%
  }%
  \providecommand\transparent[1]{%
    \errmessage{(Inkscape) Transparency is used (non-zero) for the text in Inkscape, but the package 'transparent.sty' is not loaded}%
    \renewcommand\transparent[1]{}%
  }%
  \providecommand\rotatebox[2]{#2}%
  \ifx\svgwidth\undefined%
    \setlength{\unitlength}{48bp}%
    \ifx\svgscale\undefined%
      \relax%
    \else%
      \setlength{\unitlength}{\unitlength * \real{\svgscale}}%
    \fi%
  \else%
    \setlength{\unitlength}{\svgwidth}%
  \fi%
  \global\let\svgwidth\undefined%
  \global\let\svgscale\undefined%
  \makeatother%
  \begin{picture}(1,0.51116667)%
    \put(0,0){\includegraphics[width=\unitlength,page=1]{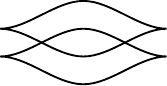}}%
  \end{picture}%
\endgroup%
}}}~,\qquad  \sL_\infty=\vcenter{\hbox{\scalebox{1.5}{
\begingroup%
  \makeatletter%
  \providecommand\color[2][]{%
    \errmessage{(Inkscape) Color is used for the text in Inkscape, but the package 'color.sty' is not loaded}%
    \renewcommand\color[2][]{}%
  }%
  \providecommand\transparent[1]{%
    \errmessage{(Inkscape) Transparency is used (non-zero) for the text in Inkscape, but the package 'transparent.sty' is not loaded}%
    \renewcommand\transparent[1]{}%
  }%
  \providecommand\rotatebox[2]{#2}%
  \ifx\svgwidth\undefined%
    \setlength{\unitlength}{48bp}%
    \ifx\svgscale\undefined%
      \relax%
    \else%
      \setlength{\unitlength}{\unitlength * \real{\svgscale}}%
    \fi%
  \else%
    \setlength{\unitlength}{\svgwidth}%
  \fi%
  \global\let\svgwidth\undefined%
  \global\let\svgscale\undefined%
  \makeatother%
  \begin{picture}(1,0.51116667)%
    \put(0,0){\includegraphics[width=\unitlength,page=1]{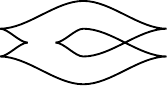}}%
  \end{picture}%
\endgroup%
}}}~.
\end{align*}
Since $R_1(\sL_-)=z^{-1}$, $R_1(\sL_0)=z^{-2}+1$, and $R_1(\sL_\infty)=0$, we have $R_1(\sL)=z^{-2}+z^{-1}+1$.

Let us also consider a different front diagram $\sL'$ of $L$.
\begin{align*}
\sL'=\vcenter{\hbox{
\begingroup%
  \makeatletter%
  \providecommand\color[2][]{%
    \errmessage{(Inkscape) Color is used for the text in Inkscape, but the package 'color.sty' is not loaded}%
    \renewcommand\color[2][]{}%
  }%
  \providecommand\transparent[1]{%
    \errmessage{(Inkscape) Transparency is used (non-zero) for the text in Inkscape, but the package 'transparent.sty' is not loaded}%
    \renewcommand\transparent[1]{}%
  }%
  \providecommand\rotatebox[2]{#2}%
  \ifx\svgwidth\undefined%
    \setlength{\unitlength}{74.40812202bp}%
    \ifx\svgscale\undefined%
      \relax%
    \else%
      \setlength{\unitlength}{\unitlength * \real{\svgscale}}%
    \fi%
  \else%
    \setlength{\unitlength}{\svgwidth}%
  \fi%
  \global\let\svgwidth\undefined%
  \global\let\svgscale\undefined%
  \makeatother%
  \begin{picture}(1,0.44081209)%
    \put(0,0){\includegraphics[width=\unitlength,page=1]{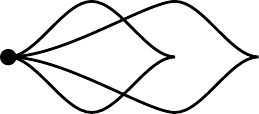}}%
  \end{picture}%
\endgroup%
}}
\end{align*}
The two are indeed equivalent through the following Legendrian isotopy: 
\begin{align*}
\vcenter{\hbox{
\begingroup%
  \makeatletter%
  \providecommand\color[2][]{%
    \errmessage{(Inkscape) Color is used for the text in Inkscape, but the package 'color.sty' is not loaded}%
    \renewcommand\color[2][]{}%
  }%
  \providecommand\transparent[1]{%
    \errmessage{(Inkscape) Transparency is used (non-zero) for the text in Inkscape, but the package 'transparent.sty' is not loaded}%
    \renewcommand\transparent[1]{}%
  }%
  \providecommand\rotatebox[2]{#2}%
  \ifx\svgwidth\undefined%
    \setlength{\unitlength}{324.8142749bp}%
    \ifx\svgscale\undefined%
      \relax%
    \else%
      \setlength{\unitlength}{\unitlength * \real{\svgscale}}%
    \fi%
  \else%
    \setlength{\unitlength}{\svgwidth}%
  \fi%
  \global\let\svgwidth\undefined%
  \global\let\svgscale\undefined%
  \makeatother%
  \begin{picture}(1,0.12561024)%
    \put(0,0){\includegraphics[width=\unitlength,page=1]{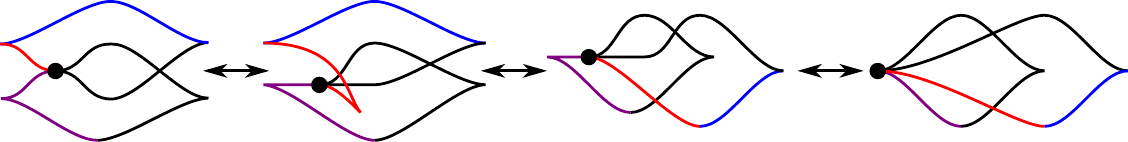}}%
  \end{picture}%
\endgroup%
}}~.
\end{align*}
Here the arcs of like color between consecutive front diagrams indicate arcs corresponding via Reidemeister moves.

For $\sL'$, the possible resolutions, cf.\ Example~\ref{ex:resolution of 4-valent}, are as follows:
\begin{align*}
\sL'_1=\vcenter{\hbox{
\begingroup%
  \makeatletter%
  \providecommand\color[2][]{%
    \errmessage{(Inkscape) Color is used for the text in Inkscape, but the package 'color.sty' is not loaded}%
    \renewcommand\color[2][]{}%
  }%
  \providecommand\transparent[1]{%
    \errmessage{(Inkscape) Transparency is used (non-zero) for the text in Inkscape, but the package 'transparent.sty' is not loaded}%
    \renewcommand\transparent[1]{}%
  }%
  \providecommand\rotatebox[2]{#2}%
  \ifx\svgwidth\undefined%
    \setlength{\unitlength}{68bp}%
    \ifx\svgscale\undefined%
      \relax%
    \else%
      \setlength{\unitlength}{\unitlength * \real{\svgscale}}%
    \fi%
  \else%
    \setlength{\unitlength}{\svgwidth}%
  \fi%
  \global\let\svgwidth\undefined%
  \global\let\svgscale\undefined%
  \makeatother%
  \begin{picture}(1,0.6)%
    \put(0,0){\includegraphics[width=\unitlength,page=1]{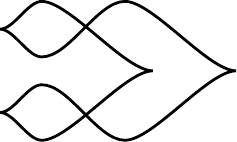}}%
  \end{picture}%
\endgroup%
}}~, \qquad  
\sL'_2=\vcenter{\hbox{
\begingroup%
  \makeatletter%
  \providecommand\color[2][]{%
    \errmessage{(Inkscape) Color is used for the text in Inkscape, but the package 'color.sty' is not loaded}%
    \renewcommand\color[2][]{}%
  }%
  \providecommand\transparent[1]{%
    \errmessage{(Inkscape) Transparency is used (non-zero) for the text in Inkscape, but the package 'transparent.sty' is not loaded}%
    \renewcommand\transparent[1]{}%
  }%
  \providecommand\rotatebox[2]{#2}%
  \ifx\svgwidth\undefined%
    \setlength{\unitlength}{72bp}%
    \ifx\svgscale\undefined%
      \relax%
    \else%
      \setlength{\unitlength}{\unitlength * \real{\svgscale}}%
    \fi%
  \else%
    \setlength{\unitlength}{\svgwidth}%
  \fi%
  \global\let\svgwidth\undefined%
  \global\let\svgscale\undefined%
  \makeatother%
  \begin{picture}(1,0.56666667)%
    \put(0,0){\includegraphics[width=\unitlength,page=1]{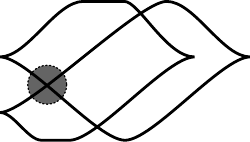}}%
  \end{picture}%
\endgroup%
}}~,\qquad  \sL'_3=\vcenter{\hbox{
\begingroup%
  \makeatletter%
  \providecommand\color[2][]{%
    \errmessage{(Inkscape) Color is used for the text in Inkscape, but the package 'color.sty' is not loaded}%
    \renewcommand\color[2][]{}%
  }%
  \providecommand\transparent[1]{%
    \errmessage{(Inkscape) Transparency is used (non-zero) for the text in Inkscape, but the package 'transparent.sty' is not loaded}%
    \renewcommand\transparent[1]{}%
  }%
  \providecommand\rotatebox[2]{#2}%
  \ifx\svgwidth\undefined%
    \setlength{\unitlength}{72bp}%
    \ifx\svgscale\undefined%
      \relax%
    \else%
      \setlength{\unitlength}{\unitlength * \real{\svgscale}}%
    \fi%
  \else%
    \setlength{\unitlength}{\svgwidth}%
  \fi%
  \global\let\svgwidth\undefined%
  \global\let\svgscale\undefined%
  \makeatother%
  \begin{picture}(1,0.45555556)%
    \put(0,0){\includegraphics[width=\unitlength,page=1]{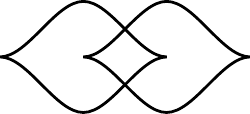}}%
  \end{picture}%
\endgroup%
}}~.
\end{align*}
It is straightforward to check that $R_1(\sL_-)=R_1(\sL'_2)$, $R_1(\sL_0)=R_1(\sL'_3)$, and $R_1(\sL_\infty)=R_1(\sL'_1)$ which implies \(R_1(\sL)=R_1(\sL')\).
\end{example}

\subsection{Ruling invariants for Legendrians in \texorpdfstring{$\#^k(S^1\times S^2)$}{the connected sums of the product of the circle and sphere}}

Let $\bn=(n_1,\dots,n_k)$ be a finite sequence of positive even integers and put $n=n_1+\cdots+n_k$.
We denote the product of the sets of perfect matchings $\cP_{[n_i]}$'s by 
\[
\cP_{[\bn]}\coloneqq \prod_{i=1}^k \cP_{[n_i]}.
\]
Via the identification
\begin{align*}
\coprod_{i=1}^k [n_i] &= \{1,\dots, n_1\}\coprod\{1,\dots, n_2\}\coprod\cdots \coprod\{1,\dots, n_k\}\\
&\simeq \{1,\dots,n_1, n_1+1,\dots, n_1+n_2,\dots,n\}= [n]
\end{align*}
we may regard $\cP_{[\bn]}$ as a subset of $\cP_{[n]}$.


Let $(\bL,\bn)$ be a pair of a marked bordered Legendrian graph $\bL$ of type $(n,n)$ and a finite sequence of integers $\bn=(n_i)$ with $n=\sum n_i$. 
We denote the set of $\rho$-graded normal rulings and tensors of type $(\phi,\phi)$ for some $\phi\in\cP_{[\bn]}$ as follows:
\begin{align*}
\bR^\rho_{\bL,\bn}&\coloneqq \coprod_{\phi\in\cP_{[\bn]}} \bR^\rho_\bL(\phi,\phi),&
\cR^\rho_{\bL,\bn}&\coloneqq  \sum_{\phi\in\cP_{[\bn]}} R^\rho_\bL(\phi,\phi)\cdot(\phi\otimes\phi^*).
\end{align*}

By assigning each $(\phi\otimes\phi^*)$ to the unit $1\in\Z[z^{\pm\frac 12}]$, we have the {\em $\rho$-graded ruling polynomial $R^\rho_{\bL,\bn}$} for the pair $(\bL,\bn)$
\[
R^\rho_{\bL,\bn}\coloneqq\sum_{\phi\in\cP_{[\bn]}} R^\rho_\bL(\phi,\phi) \in \Z[z^{\pm1}].
\]

\begin{corollary}\label{corollary:invariance for pair}
The set $\bR^\rho_{\bL,\bn}$ of $\rho$-graded normal rulings 
transfroms bijectively
and the polynomial $R^\rho_{\bL,\bn}$ is invariant under equivalences.
\end{corollary}
\begin{proof}
Since each $\bR^\rho_\bL(\phi,\phi)$ is invariant under equivalences by Theorem~\ref{theorem:invariance for graphs}, so is their union $\bR^\rho_{\bL,\bn}$.
\end{proof}

Let us consider the closure $\hat{(\bL,\bn)}$, which is a Legendrian graph that has $2k$ more vertices $\{\sfv_1,\dots, \sfv_k, \sfv_1',\dots, \sfv_k'\}$ than $\bL$ where each $\sfv_i$ and $\sfv_i'$ close $n_i$ borders from the left and the right, respectively. See Figure~\ref{figure:partialclosures}.
Then the set of normal rulings $\bR^\rho_{\bL,\bn}$ is the subset of the set of normal rulings in $\bR^\rho_{\hat{(\sL,\bn)}}$ such that the matchings on the $\sfv_i$ and on $\sfv_i'$ coincide for each $i$.

\begin{figure}[ht]
\[
(\bL,\bn)=\vcenter{\hbox{\input{L_n_input.tex}}}\quad\stackrel{\hat{}}\longrightarrow\quad
\hat{(\bL,\bn)}=\vcenter{\hbox{\input{L_n_closure_input.tex}}}
\]
\caption{The closure of a pair $(\bL,\bn)$ with $\bn=(2,4,4)$.}
\label{figure:partialclosures}
\end{figure}

On the other hand, the pair $(\bL,\bn)$ can be regarded as a {\em Gompf standard form} of a marked Legendrian link $[\bL,\bn]$ defined in \cite[Definition~2.1]{Gompf1998}, with a Maslov potential, in the $k$-fold connected sum $M_k\coloneqq\#^k(S^2\times S^1)$.
Here the contact manifold $M_k$ is the boundary of the four-manifold $W_k$ obtained from $\R^4$ by attaching $k$ 1-handles, or equivalently, $M_k$ is obtained by identifying pairs of boundary spheres in $\R^3$ with $2k$ balls removed.
In this description, each boundary component plays the role of the \emph{co-core} of the corresponding 1-handle.
\begin{align*}
M_k&=\#^k(S^2\times S^1)\cong \R^3\setminus \bigcup_{i=1}^{k} (\mathring{B}_{\mathsf{L},i}^3\cup\mathring{B}_{\mathsf{R},i}^3)\bigg/ S^2_{\mathsf{L},i}\sim S^2_{\mathsf{R},i},&
S^2_{*,i}&=\partial B^3_{*,i}\\
&\cong \partial W_k,\\
W_k&\coloneqq \R^4\cup\bigcup_{i=1}^k I\times D^3.
\end{align*}


\begin{figure}[ht]
\[
(\bL,\bn)=\vcenter{\hbox{\input{L_n_input.tex}}}\quad\stackrel{[\cdot]}\longrightarrow\quad
[\bL,\bn]=\vcenter{\hbox{\input{L_n_Gompf_input.tex}}}\subset M_k
\]
\caption{A Gompf standard form corresponding to a pair $(\bL,\bn)$ with $\bn=(2,4,4)$.}
\label{figure:gompf normal form}
\end{figure}

\begin{definition}[Ruling polynomials for Legendrians in $M_k$]
The $\rho$-graded ruling polynomial for a marked Legendrian link in $M_k$, given by the Gompf standard form $[\bL,\bn]$, is defined as the $\rho$-graded ruling polynomial for the pair $(\bL,\bn)$:
\[
R^\rho_{[\bL,\bn]}\coloneqq R^\rho_{\bL,\bn}.
\]
\end{definition}

\begin{theorem}
The ruling polynomial $R^\rho_{[\bL,\bn]}$ is well-defined.
\end{theorem}

\begin{proof}
It suffices to prove that the ruling polynomial $R^\rho_{[\bL,\bn]}$ is independent of the choice of a Gompf standard form.

Recall Theorem~2.2 from \cite{Gompf1998}
that two marked Legendrian links $[\bL,\bn]$ and $[\bL',\bn']$, with Maslov potentials, given in Gompf standard form in $M_k$ are isotopic if and only if the bordered Legendrian links $(\bL,\bn)$ and $(\bL',\bn')$ in $\R^3$ are related via Reidemeister moves away from the borders corresponding to co-cores --- i.e., Reimemeister moves in the boxed region of Figure~\ref{figure:gompf normal form} --- and {\em Gompf moves}, which are depicted in Figure~\ref{figure:gompf moves}.


The invariance under Reidemeister moves is already established in Corollary~\ref{corollary:invariance for pair}, and the invariance of ruling invariants under Gompf moves can be shown as follows.

For the Gompf move (GI), the invariance essentially comes from the bijection 
\[
\cP_{[n]} \simeq \cP_{[n],\{i,i+1\}}\coloneqq\{\phi\in\cP_{[n+2]}\mid \{i,i+1\}\in\phi\}
\]
between perfect matchings.
Namely, the Gompf move (GI) inserting the cusp at the $i$-th position forces the matching to have $\{i,i+1\}\in\phi$, and the above bijection induces the bijection between the sets of normal rulings.

For the Gompf move (GII), we use the bijection on $\cP_{[n]}$
\begin{align*}
\phi\in\cP_{[n]}&\mapsto \phi'\in\cP_{[n]},&
\phi'(j)&\coloneqq\begin{cases}
\phi(j) & j\neq i, i+1;\\
\phi(i+1) & j=i;\\
\phi(i) & j=i+1,
\end{cases}
\end{align*}
which directly induces the bijection between the sets of normal rulings again.

Finally, the Gompf move (GIII) is nothing but the composition of two Reidemeister (IV) moves and hence we have the invariance.
\end{proof}

\begin{figure}[ht]
\[
\setlength{\arraycolsep}{2pt}
\begin{array}{rcl}
\vcenter{\hbox{\includegraphics[scale=1]{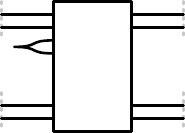}}}&\stackrel{\rm(GI)}\longleftrightarrow&
\vcenter{\hbox{\includegraphics[scale=1]{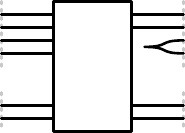}}}\\ \\
\vcenter{\hbox{\includegraphics[scale=1]{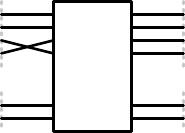}}}&\stackrel{\rm(GII)}\longleftrightarrow&
\vcenter{\hbox{\includegraphics[scale=1]{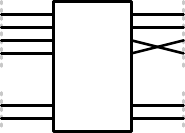}}}\\ \\
\vcenter{\hbox{\includegraphics[scale=1]{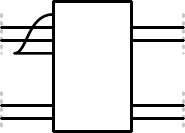}}}&\stackrel{\rm(GIII)}\longleftrightarrow&
\vcenter{\hbox{\includegraphics[scale=1]{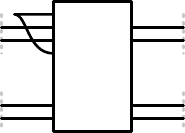}}}
\end{array}
\]
\caption{Gompf moves (GI), (GII) and (GIII).}
\label{figure:gompf moves}
\end{figure}

In the paper \cite{Leverson2017}, Leverson defined $\rho$-graded normal rulings for Legendrian links in $M_k$ by using Gompf standard forms. More precisely, if $[\bL,\bn]$ is a Gompf standard form of a Legendrian link in $M_k$ without markings, then the $\rho$-graded normal rulings for $[\bL,\bn]$ are those normal rulings of the bordered Legendrian $(\bL,\bn)$ whose matchings at the left and right ends coincide.

Notice that this definition is exactly the same as our definition for $\bR^\rho_{\bL,\bn}$ and therefore we have the following corollary.
\begin{corollary}\label{corollary:Leverson}
Let $[\bL=(\sL,\mu),\bn]$ be a Legendrian link in $M_k$ having no markings. Then the ruling polynomial $R^\rho_{[\bL,\bn]}$ coincides with the ruling invariant defined by Leverson in \cite[Definition~2.14]{Leverson2017}.
\end{corollary}

One way to go from connected Legendrian graphs in $\R^3$ to Legendrian links in $M_k$ is a {\em doubling construction} defined as follows:
For any marked {\em non-bordered} Legendrian graph $\bL=(\sL,\mu)$ with $k$ vertices and no markings, let us consider the {\em double} $D(\bL)=(D(\sL),D(\mu))$ of $\bL$ defined as the disjoint union of two copies, say $\bL_1\coloneqq(\sL_1,\mu_1)$ and $\bL_2\coloneqq(\sL_2,\mu_2)$ of $\bL$.
By applying Reidemeister moves, we may pull all vertices of $\bL_1$ to the left and pull all vertices of $\bL_2$ to the right so that the graph looks as follows:
\[
D(\bL)\coloneqq (D(\sL),D(\mu))=\vcenter{\hbox{
\begingroup%
  \makeatletter%
  \providecommand\color[2][]{%
    \errmessage{(Inkscape) Color is used for the text in Inkscape, but the package 'color.sty' is not loaded}%
    \renewcommand\color[2][]{}%
  }%
  \providecommand\transparent[1]{%
    \errmessage{(Inkscape) Transparency is used (non-zero) for the text in Inkscape, but the package 'transparent.sty' is not loaded}%
    \renewcommand\transparent[1]{}%
  }%
  \providecommand\rotatebox[2]{#2}%
  \ifx\svgwidth\undefined%
    \setlength{\unitlength}{104.68651235bp}%
    \ifx\svgscale\undefined%
      \relax%
    \else%
      \setlength{\unitlength}{\unitlength * \real{\svgscale}}%
    \fi%
  \else%
    \setlength{\unitlength}{\svgwidth}%
  \fi%
  \global\let\svgwidth\undefined%
  \global\let\svgscale\undefined%
  \makeatother%
  \begin{picture}(1,0.47284)%
    \put(0,0){\includegraphics[width=\unitlength,page=1]{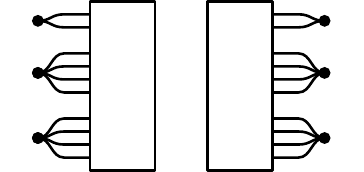}}%
    \put(-0.00289182,0.07522487){\color[rgb]{0,0,0}\makebox(0,0)[lb]{\smash{$\sfv_3$}}}%
    \put(-0.00289182,0.25433131){\color[rgb]{0,0,0}\makebox(0,0)[lb]{\smash{$\sfv_2$}}}%
    \put(-0.00289182,0.39761597){\color[rgb]{0,0,0}\makebox(0,0)[lb]{\smash{$\sfv_1$}}}%
    \put(0.28367804,0.21850952){\color[rgb]{0,0,0}\makebox(0,0)[lb]{\smash{$\sL_1$}}}%
    \put(0.60606913,0.21850952){\color[rgb]{0,0,0}\makebox(0,0)[lb]{\smash{$\sL_2$}}}%
    \put(0.93026494,0.39761584){\color[rgb]{0,0,0}\makebox(0,0)[lb]{\smash{$\sfv_1'$}}}%
    \put(0.9302948,0.25433138){\color[rgb]{0,0,0}\makebox(0,0)[lb]{\smash{$\sfv_2'$}}}%
    \put(0.9302202,0.07522532){\color[rgb]{0,0,0}\makebox(0,0)[lb]{\smash{$\sfv_3'$}}}%
  \end{picture}%
\endgroup%
}}=\bL_1\amalg \bL_2.
\]
Then it is obvious that $D(\bL)$ can be realized as the closure of a pair $(\tilde\bL,\bn)$ such that 
$\tilde\bL=\tilde\bL_1\cdot\tilde\bL_2$ for two bordered Legendrian graphs $\tilde\bL_1\coloneqq(\tilde\sL_1,\tilde\mu_1)$ and $\tilde\bL_2\coloneqq(\tilde\sL_2,\tilde\mu_2)$ of types $(n,0)$ and $(0,n)$, respectively, and $n=n_1+\cdots+n_k$ where $n_i$ is the valency of the vertex $\sfv_i$ in $\sL$:
\begin{align*}
D(\bL)&=\hat{(\tilde\bL,\bn)};&
(\tilde\bL,\bn)&=\vcenter{\hbox{
\begingroup%
  \makeatletter%
  \providecommand\color[2][]{%
    \errmessage{(Inkscape) Color is used for the text in Inkscape, but the package 'color.sty' is not loaded}%
    \renewcommand\color[2][]{}%
  }%
  \providecommand\transparent[1]{%
    \errmessage{(Inkscape) Transparency is used (non-zero) for the text in Inkscape, but the package 'transparent.sty' is not loaded}%
    \renewcommand\transparent[1]{}%
  }%
  \providecommand\rotatebox[2]{#2}%
  \ifx\svgwidth\undefined%
    \setlength{\unitlength}{71.51545383bp}%
    \ifx\svgscale\undefined%
      \relax%
    \else%
      \setlength{\unitlength}{\unitlength * \real{\svgscale}}%
    \fi%
  \else%
    \setlength{\unitlength}{\svgwidth}%
  \fi%
  \global\let\svgwidth\undefined%
  \global\let\svgscale\undefined%
  \makeatother%
  \begin{picture}(1,0.69215769)%
    \put(0,0){\includegraphics[width=\unitlength,page=1]{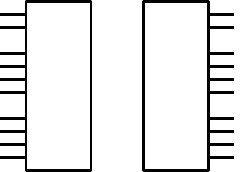}}%
    \put(0.15730865,0.31986134){\color[rgb]{0,0,0}\makebox(0,0)[lb]{\smash{$\tilde\sL_1$}}}%
    \put(0.62923462,0.31986056){\color[rgb]{0,0,0}\makebox(0,0)[lb]{\smash{$\tilde\sL_2$}}}%
  \end{picture}%
\endgroup%
}}=\tilde\bL_1\cdot \tilde\bL_2.
\end{align*}

By treating $(\tilde\bL,\bn)$ as a Gompf standard form, we obtain a Legendrian link $[\tilde\bL,\bn]$ in $M_k$. We denote this by $[D(\bL)]$. 

\begin{remark}
For singular Legendrian links, which are Legendrian 4-valent graphs, the double construction was considered in \cite[Section~6.2]{ABK2018}.
\end{remark}

We have the following further corollary:

\begin{corollary}\label{cor:double ruling}
Let $\bL=(\sL,\mu)$ be a Legendrian graph without markings. Then $\bL$ has a $\rho$-graded normal ruling if and only if the Legendrian link $[D(\bL)]$ in $M_k$ has a $\rho$-graded normal ruling in the sense of Leverson \cite{Leverson2017}.
\end{corollary}

\begin{proof}
Reidemeister moves do not affect whether $\bL$ has a normal ruling, and therefore $\bL$ has a $\rho$-graded normal ruling if and only if so does $D(\bL)=\bL_1\amalg\bL_2$. Moreover, it is obvious that 
\[
\bR^\rho_{\tilde\bL}=\bR^\rho_{D(\bL)}= \bR^\rho_{\bL_1}\times \bR^\rho_{\bL_2}=\bR^\rho_{\tilde\bL_1}\times\bR^\rho_{\tilde\bL_2}. 
\]
In addition, for each $\phi\in\cP_{[\bn]}$, we have
\[
\bR^\rho_{\tilde\bL}(\phi,\phi) = \bR^\rho_{\tilde\bL_1}(\phi,\emptyset) \times \bR^\rho_{\tilde\bL_2}(\emptyset,\phi).
\]

Since there is a one-to-one correspondence between $\bR^\rho_{\tilde\bL_1}(\phi,\emptyset)$ and $\bR^\rho_{\tilde\bL_2}(\emptyset,\phi)$, 
\[
\bR^\rho_{\tilde\bL,\bn}=\coprod_{\phi\in\cP_{[\bn]}} 
\bR^\rho_{\tilde\bL}(\phi,\phi) \neq\emptyset\Longleftrightarrow \coprod_{\phi\in\cP_{[\bn]}}\bR^\rho_{\tilde\bL_1}(\phi,\emptyset)\neq\emptyset.
\]

However, it is obvious that the right hand side is the same as $\bR^\rho_{\tilde\bL_1}=\bR^\rho_\bL$ and so we have
\[
\bR^\rho_{\tilde\bL,\bn}\neq\emptyset\Longleftrightarrow \bR^\rho_\bL\neq\emptyset.
\]
With this we are done since the left hand side is the same as the set of normal rulings of $[D(\bL)]$ in the sense of Leverson by Corollary~\ref{corollary:Leverson}.
\end{proof}

\section{Applications}

\subsection{Existence of rulings and augmentations}

In this section we briefly review the construction of the differential graded algebra (DGA for short) $\cA(\bL)$ for Legendrian graphs $\bL=(\sL,\mu)$ with Maslov potential, introduced by the first and second authors in \cite{AB2018}, and prove the equivalence between the existence of a normal ruling of $\bL$ and an augmentation of $\cA(\bL)$.
This result generalizes and unifies previous work for Legendrian links in $\R^3$ \cite{Sabloff2005, Leverson2016} and in $M_k=\#^k (S^2\times S^1)$ \cite{Leverson2017}.

\subsubsection{DGAs for Legendrian graphs}
To a Legendrian link $\bL=(\sL,\mu)$ with a Maslov potential, one can associate the Chekanov--Eliashberg DGA $\cA(\bL)$, which is a differential graded algebra generated by crossings in the Lagrangian projection (a.k.a.\ \emph{Reeb chords}) and whose differential comes from counting immersed polygons satisfying certain boundary conditions. 
Recently, the construction of the DGA invariant has been generalized to arbitrary Legendrian graphs \cite{AB2018}.

The main task was 
\begin{enumerate}
\item to handle algebraic behavior (or a DGA construction) at the vertices and 
\item to show the invariance under new (Lagrangian) Reidemeister moves which arise from the vertices.
\end{enumerate}
For the first issue, we assigned a DG-subalgebra $\cI_\sfv(\bL)$ for each vertex $\sfv\in \sV_\sL$, see Remark~\ref{rem:peripheral structure}.
For the second issue, it is needed to extend the notion of algebraic equivalence of DGAs from {\em stable-tame isomorphisms} to {\em generalized stable-tame isomorphisms}, see \cite{AB2018} for the precise definition. With these terminology, we have


\begin{theorem}\cite[Theorem~A,B]{AB2018}
Let $\bL=(\sL,\mu)$ be a Legendrian graph with Maslov potential. Then there is a pair $(\cA(\bL), \cP(\bL))$ consisting of a DGA $\cA(\bL)$ and a collection $\cP(\bL)$ of DG-subalgebras from vertices $\sV_\sL$.

Moreover, the pair $(\cA(\bL), \cP(\bL))$ up to generalized stable-tame isomorphisms is invariant under the Legendrian Reidemeister moves for $\bL=(\sL,\mu)$. In particular the induced homology $H_*(\cA(\bL),\partial)$ is an invariant.
\end{theorem}


To define the DGA $\cA(\bL)$, we use the Lagrangian projection $\pi_{xy}:\R^3 \to \R^2_{xy}$ of $\sL$. There is a {\em combinatorial} way to obtain a Lagrangian projection of $\sL$ from a front diagram, due to Ng \cite{Ng2003}, called {\em resolution}.

\begin{definition}\cite[Definition~2.1]{Ng2003}\label{definition:Ng's resolution}
Let $\sL$ be a regular front projection of a Legendrian graph. Then the {\em resolution} $\res(\sL)$ is a diagram in the $xy$-plane obtained  by performing on $\sL$ the operations
\begin{align*}
\Lcusp&\mapsto \leftarc~,&
\Rcusp&\mapsto\rightkink~,&
\crossing&\mapsto \crossingpositive~,
\end{align*}
along with, for each vertex $v$ of type $(\ell,r)$, the replacement
\[
\vcenter{\hbox{}}\mapsto\vcenter{\hbox{
\begingroup%
  \makeatletter%
  \providecommand\color[2][]{%
    \errmessage{(Inkscape) Color is used for the text in Inkscape, but the package 'color.sty' is not loaded}%
    \renewcommand\color[2][]{}%
  }%
  \providecommand\transparent[1]{%
    \errmessage{(Inkscape) Transparency is used (non-zero) for the text in Inkscape, but the package 'transparent.sty' is not loaded}%
    \renewcommand\transparent[1]{}%
  }%
  \providecommand\rotatebox[2]{#2}%
  \ifx\svgwidth\undefined%
    \setlength{\unitlength}{126.28139512bp}%
    \ifx\svgscale\undefined%
      \relax%
    \else%
      \setlength{\unitlength}{\unitlength * \real{\svgscale}}%
    \fi%
  \else%
    \setlength{\unitlength}{\svgwidth}%
  \fi%
  \global\let\svgwidth\undefined%
  \global\let\svgscale\undefined%
  \makeatother%
  \begin{picture}(1,0.38604261)%
    \put(0,0){\includegraphics[width=\unitlength,page=1]{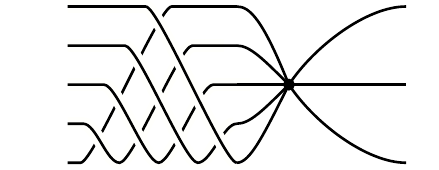}}%
    \put(0.64443659,0.24647418){\color[rgb]{0,0,0}\makebox(0,0)[lb]{\smash{$\sfv$}}}%
    \put(0.93494112,0.3697179){\color[rgb]{0,0,0}\makebox(0,0)[lb]{\smash{\tiny$1$}}}%
    \put(0.93794087,0.19048533){\color[rgb]{0,0,0}\makebox(0,0)[lb]{\smash{\tiny$2$}}}%
    \put(0.93944074,0.01050265){\color[rgb]{0,0,0}\makebox(0,0)[lb]{\smash{\tiny$r$}}}%
    \put(0.01374111,0.37121743){\color[rgb]{0,0,0}\makebox(0,0)[lb]{\smash{\tiny$r+1$}}}%
    \put(0.01374111,0.27672647){\color[rgb]{0,0,0}\makebox(0,0)[lb]{\smash{\tiny$r+2$}}}%
    \put(0.01374111,0.00900312){\color[rgb]{0,0,0}\makebox(0,0)[lb]{\smash{\tiny$r+\ell$}}}%
  \end{picture}%
\endgroup%
}}~.
\]
\end{definition}

The unital algebra $\cA(\bL)$ over $\Z$ is generated by the union of the set $\sC(\res(\sL))$ of crossings of the resolution $\res(\sL)$ and an infinite set of generators for each vertex, namely
\[
\cA(\bL)\coloneqq\Z\langle \sC(\res(\sL))\amalg \tilde\sV(\sL)\rangle,
\]
where
\begin{align}\label{eqn:vertex_generators}
\tilde\sV(\sL)\coloneqq\{\sfv_{i,\ell}\mid \sfv\in\sV(\sL),\,i\in\Zmod{\val(\sfv)},\, \ell\ge 1\}.
\end{align}
We assign so called Reeb and orientation signs to the four quadrants at each crossing $\sfc$ of $\res(\sL)$ as depicted in Figure~\ref{figure:sign}.
From now on, shaded regions indicate quadrants whose orientation sign, depending on the grading of the crossing, may be negative.

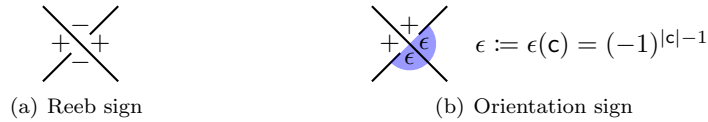
\begin{figure}[ht]
\subfigure[Reeb sign]{\makebox[0.4\textwidth]{
$\begin{tikzpicture}[scale=0.5,baseline=-.5ex]
\draw[thick] (-1,-1) -- (1,1);
\draw[line width=10pt,white] (-1,1) -- (1,-1);
\draw[thick] (-1,1) -- (1,-1);
\draw(0,0) node[above] {$-$} node[below] {$-$} node[left] {$+$} node[right] {$+$};
\end{tikzpicture}$}}
\subfigure[Orientation sign]{\makebox[0.4\textwidth]{
$\begin{tikzpicture}[scale=0.5,baseline=-.5ex]
\draw[thick] (-1,-1) -- (1,1);
\draw[line width=10pt,white] (-1,1) -- (1,-1);
\fill[blue,opacity=0.4](-0.5,-0.5)--(0.5,0.5) arc (45:-135:0.707);
\draw[thick] (-1,1) -- (1,-1);
\draw(0,0) node[above] {$+$} node[below] {$\epsilon$} node[left] {$+$} node[right] {$\epsilon$};
\draw(1.5,0) node[right] {$\epsilon\coloneqq \epsilon(\sfc)=(-1)^{|\sfc|-1}$};
\end{tikzpicture}$}}
\caption{Reeb signs and orientation signs.}
\label{figure:sign}
\end{figure}

Here the grading $|\sfc|$ of the crossing $\sfc\in\sC(\res(\sL))$ is as given in equation \eqref{equation:grading}.
For a generator $\sfv_{i,j}$ belonging to a vertex $\sfv$ of type $(\ell,r)$, the grading is defined as
\[
|\sfv_{i,j}|\coloneqq \mu(i)-\mu(i+j) + (n-1)\in\fR,
\]
where $n$ is the number of intersections between the vertical line passing through $\sfv$ and the spiral curve $\gamma(\sfv,i,j)$ that starts from the $i$-th half-edge, rotates clockwise around $\sfv$, and passes exactly $j$ minimal sectors.

\begin{figure}[ht]
\begin{align*}
\gamma(\sfv,1,3)&=
\begin{tikzpicture}[baseline=-.5ex]
\begin{scope}[scale=0.8]
\draw[dashed,gray] (0,-1)--(0,1);
\draw[fill] (0,0) circle (3pt);
\foreach \t in {1,...,6} {
	\draw[thick] (0,0) -- (120-\t * 60:1);
	\draw(120-\t*60:1.3) node {$\t$};
}
\draw[red,thick,-latex'] plot[variable=\t,domain=60:-120,samples=50] ({.5 * cos( \t )},{ 0.5 * sin( \t )});
\end{scope}
\end{tikzpicture}&
\gamma(\sfv,1,7)&=
\begin{tikzpicture}[baseline=-.5ex]
\begin{scope}[scale=0.8]
\draw[dashed,gray] (0,-1)--(0,1);
\draw[fill] (0,0) circle (3pt);
\foreach \t in {1,...,6} {
	\draw[thick] (0,0) -- (120-\t * 60:1);
	\draw(120-\t*60:1.3) node {$\t$};
}
\draw[red,thick,-latex'] plot[variable=\t,domain=60:-360,samples=50] ({ (.5 - (\t-60) / 1800 ) * cos( \t )},{ (.5 - (\t-60) / 1800) * sin( \t )});
\end{scope}
\end{tikzpicture}\\
|\sfv_{1,3}|&=\mu(1)-\mu(4)+(1-1),&
|\sfv_{1,7}|&=\mu(1)-\mu(2)+(2-1),&
\end{align*}
\caption{Examples of piral curves $\gamma(\sfv,i,j)$.}
\label{figure:spiral curve}
\end{figure}
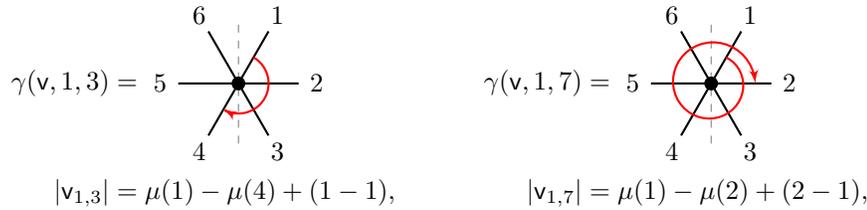

Let $\Pi_t$ be a $(t+1)$-gon and let us denote its boundary and the set of its vertices by $\partial\Pi_t$ and $V\Pi_t=\{\bx_0,\dots,\bx_t\}$, respectively. 
The differential for each crossing $\sfc$ is given by counting immersed polygons 
\[
f\colon (\Pi_t,\partial \Pi_t, V\Pi_t)\to(\R^2,\sL,\sC_\sL\cup\sV_\sL)
\]
which pass only one Reeb-positive quadrant at $f(\bx_0)=\sfc$ and several Reeb-negative quadrants and {\em vertex corners}.
When $f$ maps a vertex of $\Pi_t$ to a crossing then a neighborhood of the vertex is mapped to a single quadrant (positive for $\bx_0$, negative otherwise) at the crossing. There is no such local convexity requirement for vertices that are mapped to (projections of) vertices, cf.\ Figure \ref{figure:vertex corner}.

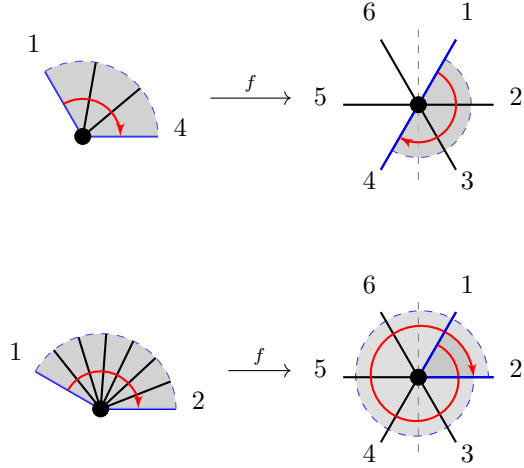
\begin{figure}[ht]
\[
\begin{tikzcd}
\begin{tikzpicture}[baseline=1pc]
\draw[thick, blue] (120:1)--(0,0)--(0:1);
\draw[dashed, blue] (120:1) arc (120:0:1);
\fill[lightgray, opacity=0.7] (0,0) -- (120:1) arc (120:0:1) -- cycle;
\draw[fill,black] (0,0) circle (3pt);
\draw(120:1.3) node {$1$};
\draw[thick] (0,0)--(80:1);
\draw[thick] (0,0)--(40:1);
\draw(0:1.3) node {$4$};
\draw[thick,red,-latex'] (120:0.5) arc (120:0:0.5);
\end{tikzpicture}\arrow[r,"f"]&
\begin{tikzpicture}
\draw[dashed,gray] (0,-1)--(0,1);
\draw[dashed, blue] plot[variable=\t,domain=60:-120,samples=50] ({ (.7 ) * cos( \t )},{ (.7) * sin( \t )});
\fill[lightgray,opacity=0.7] (0,0) -- (60:0.7) plot[variable=\t,domain=60:-120,samples=50] ({ (.7 ) * cos( \t )},{ (.7 ) * sin( \t )}) -- (0,0);
\foreach \t in {1,...,6} {
	\draw[thick] (0,0) -- (120-\t * 60:1);
	\draw(120-\t*60:1.3) node {$\t$};
}
\draw[red,thick,-latex'] plot[variable=\t,domain=60:-120,samples=50] ({.5 * cos( \t )},{ 0.5 * sin( \t )});
\draw[blue, thick] (0,0) -- (60:1) (0,0) -- (240:1);
\draw[fill] (0,0) circle (3pt);
\end{tikzpicture}\\
\begin{tikzpicture}[baseline=1pc]
\draw[thick, blue] (150:1)--(0,0)--(0:1);
\draw[dashed, blue] (150:1) arc (150:0:1);
\fill[lightgray, opacity=0.7] (0,0) -- (150:1) arc (150:0:1) -- cycle;
\draw[fill,black] (0,0) circle (3pt);
\draw(150:1.3) node {$1$};
\foreach \t in {1,...,6} {
	\draw[thick] (0,0) -- (150-\t*150/7:1);
}
\draw(0:1.3) node {$2$};
\draw[thick,red,-latex'] (150:0.5) arc (150:0:0.5);
\end{tikzpicture}\arrow[r,"f"] &
\begin{tikzpicture}
\draw[dashed, blue] plot[variable=\t,domain=60:-360,samples=50] ({ (.7 - (\t-60) / 1800 ) * cos( \t )},{ (.7 - (\t-60) / 1800) * sin( \t )});
\fill[lightgray,opacity=0.7] (0,0) -- (60:0.7) plot[variable=\t,domain=60:0,samples=50] ({ (.7 - (\t-60) / 1800 ) * cos( \t )},{ (.7 - (\t-60) / 1800) * sin( \t )}) -- (0,0);
\fill[lightgray,opacity=0.5] (0,0) -- (60:0.7) plot[variable=\t,domain=0:-360,samples=50] ({ (.7 - (\t-60) / 1800 ) * cos( \t )},{ (.7 - (\t-60) / 1800) * sin( \t )}) -- (0,0);
\draw[dashed,gray] (0,-1)--(0,1);
\foreach \t in {1,...,6} {
	\draw[thick] (0,0) -- (120-\t * 60:1);
	\draw(120-\t*60:1.3) node {$\t$};
}
\draw[red,thick,-latex'] plot[variable=\t,domain=60:-360,samples=50] ({ (.5 - (\t-60) / 1800 ) * cos( \t )},{ (.5 - (\t-60) / 1800) * sin( \t )});
\draw[blue, thick] (0,0) -- (60:1) (0,0) -- (1,0);
\draw[fill] (0,0) circle (3pt);
\end{tikzpicture}
\end{tikzcd}
\]
\caption{Vertex corners of immersed polygons}
\label{figure:vertex corner}
\end{figure}

\begin{definition}[Signs of polygons]
For an immersed polygon $f$ with domain $\Pi_t$ having the vertex $\bx\in V\Pi_t$, the sign $\sgn(f,\bx)$ is defined as follows:
\begin{itemize}
\item If $f(\bx)$ is a crossing $\sfc$, then $\sgn(f,\bx)$ is the orientation sign of the quadrant locally covered by the image of $f$, cf.\ Figure~\ref{figure:sign}.
\item If $f(\bx)$ is a vertex, then $\sgn(f,\bx)$ is defined to be $1$.
\end{itemize}
\end{definition}

\begin{definition}[Canonical label]\label{def:Canonical label}
Let 
\[
f \colon (\Pi_t,\partial \Pi_t, V\Pi_t)\to(\R^2,\sL,\sC_\sL\cup\sV_\sL)
\] 
be an orientation preserving immersed polygon as above.
Let us label the nearby edges $\bh_{\bv_+}$, $\bh_{\bv_-}$ on a neighborhood $\bU_\bv$ of $\bv\in V\Pi_t$ as follows:
\begin{align*}
\vcenter{\hbox{
\begingroup%
  \makeatletter%
  \providecommand\color[2][]{%
    \errmessage{(Inkscape) Color is used for the text in Inkscape, but the package 'color.sty' is not loaded}%
    \renewcommand\color[2][]{}%
  }%
  \providecommand\transparent[1]{%
    \errmessage{(Inkscape) Transparency is used (non-zero) for the text in Inkscape, but the package 'transparent.sty' is not loaded}%
    \renewcommand\transparent[1]{}%
  }%
  \providecommand\rotatebox[2]{#2}%
  \ifx\svgwidth\undefined%
    \setlength{\unitlength}{64.80938242bp}%
    \ifx\svgscale\undefined%
      \relax%
    \else%
      \setlength{\unitlength}{\unitlength * \real{\svgscale}}%
    \fi%
  \else%
    \setlength{\unitlength}{\svgwidth}%
  \fi%
  \global\let\svgwidth\undefined%
  \global\let\svgscale\undefined%
  \makeatother%
  \begin{picture}(1,0.80661243)%
    \put(0,0){\includegraphics[width=\unitlength,page=1]{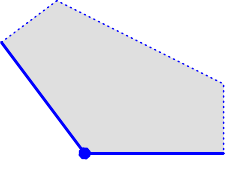}}%
    \put(0.28378866,0.00831651){\color[rgb]{0,0,0}\makebox(0,0)[lb]{\smash{$\bv$}}}%
    \put(0.59591606,0.00727795){\color[rgb]{0,0,0}\makebox(0,0)[lb]{\smash{$\bh_{\bv_+}$}}}%
    \put(0.0325906,0.25153242){\color[rgb]{0,0,0}\makebox(0,0)[lb]{\smash{$\bh_{\bv_-}$}}}%
    \put(0.40,0.40){\color[rgb]{0,0,0}\makebox(0,0)[lb]{\smash{$\bU_{\bv}$}}}
  \end{picture}%
\endgroup%
}}
\end{align*}

We define a function $\tilde f:\vPi\to \sG_\sL$, called the {\em canonical label} of $f$, as
\[
\tilde f(\bv)\coloneqq\begin{cases}
\sgn(f,\bv) \sfc & f(\bv)=\sfc\in\sC_\sL;\\
\sfv_{i,\ell} & f(\bv)=\sfv\in\sV_\sL,
\end{cases}
\]
where 
$f(\bh_{\bv_-}\cap \bU_\bv)\subset \sfh_{\sfv,i}$, and $f(\bU_\bv)$ is mapped to $\ell\coloneqq\ell_f(\bv)$ sectors near $\bv$.

\end{definition}

\begin{definition}[Grading of polygons]
The {\em grading} of the immersed $(t+1)$-gon $f$ is defined by
\[
|f|\coloneqq |\tilde f(\bx_0)|-\sum_{i=1}^t|\tilde f(\bx_i)|.
\]
\end{definition}


\begin{definition}[Differential]\label{def:differential}
For each $\sfc\in\sC_\sL$, let $\cM_t(\sfc)$ be the set of all immersed $(t+1)$-gons $f$ with $\tilde f(\bx_0)=\pm\sfc$ whose degree is $1$:
\[
\cM_t(\sfc)\coloneqq\{f\colon(\Pi_t,\partial\Pi_t,V\Pi_t)\to(\R^2,\sL,\sC_\sL\cup\sV_\sL)\mid \tilde f(\bx_0)=\pm\sfc,\,|f|=1\}.
\]
Then, the differential $\partial\sfc$ is defined as
\[
\partial\sfc\coloneqq \sum_{t\ge0}\sum_{f\in\cM_t(\sfc)} \sgn(f,\bx_0)\tilde f(\bx_1)\cdots \tilde f(\bx_t).
\]

On the other hand, for $\sfv_{i,j}$, the differential is given by the following formula.
\[
\partial \sfv_{i,j} \coloneqq
\delta_{j,\val(\sfv)}+\sum_{\substack{j_1+j_2=j}} (-1)^{|\sfv_{i,j_1}|-1}\sfv_{i,j_1}\sfv_{i+j_1, j_2}.
\]
\end{definition}

\begin{remark}\label{rem:peripheral structure}
Notice that $\partial\sfv_{i,j}$ involves only $\sfv_{i',j'}$'s and therefore, we have the DG-subalgebra $\cI_\sfv(\bL)$ for each vertex $\sfv$ generated by $\sfv_{i,j}$'s. Hence we have a DGA morphism $\bp_\sfv\colon \cI_\sfv(\bL)\to\cA(\bL)$, especially a DG-subalgebra.
\end{remark}

Furthermore, one can obtain Ekholm-Ng's DGA invariants for Legendrian links with Maslov potentials contained in $M_k$ defined in \cite{EN2015}.
Recall that a Legendrian link in $M_k$ can be represented by a pair $(\bL,\bn)$ of a bordered Legendrian graph $\bL=(\sL,\mu)$ of type $(n,n)$ without markings and a sequence $\bn=(n_1,\dots, n_k)$ of natural numbers with $n_1+\cdots+n_k=n$.
As before, we denote the set of vertices of the closure $\hat{(\bL,\bn)}$ by $\{\sfv_i, \sfv'_i\mid 1\le i\le k\}$.

\begin{theorem}\cite[Theorem~7.9]{AB2018}
Let $[\bL,\bn]$ be a Legendrian link with a Maslov potential in $M_k$ given as a Gompf standard form and let $\bar\bL\coloneqq\hat{(\bL,\bn)}$. The Ekholm-Ng's DGA $\cA^{\mathsf{EN}}([\bL,\bn])$ can be defined as the homotopy coequalizer
\[
\begin{tikzcd}\displaystyle
\coprod_{i=1}^k\cI_{\sfv_i}(\bar\bL)\arrow[r,"\coprod\bp_i",shift left=0.5ex]\arrow[r,"\coprod\bp_i'"',shift right=0.5ex] & \cA(\bar\bL)
\arrow[r]& \cA^{\mathsf{EN}}([\bL,\bn]),
\end{tikzcd}
\]
where $\bp_i$ and $\bp_i'$ are peripheral structures 
\begin{align*}
\bp_i &\colon  \cI_{\sfv_i}(\bar\bL) \stackrel{\bp_{\sfv_i}}\longrightarrow \cA(\bar\bL),&
\bp_i' &\colon  \cI_{\sfv_i} (\bar\bL) \simeq \cI_{\sfv_i'} \left(\bar\bL\right) \stackrel{\bp_{\sfv'_i}}\longrightarrow \cA(\bar\bL).
\end{align*}
\end{theorem}

In particular, for any non-bordered Legendrian graph $\bL=(\sL,\mu)$ with $k$ vertices $\{\sfv_1,\dots,\sfv_k\}$ without markings, the Ekholm-Ng's DGA $\cA^{\mathsf{EN}}([D(\bL)])$ for $[D(\bL)]\subset M_k$ can be defined as the quotient of $\cA(D(\bL))$ of the DGA for the double $D(\bL)=(D(\sL), D(\mu))$ in $\R^3$.
Indeed, we have the following (homotopy) pushout diagram consisting of injective homomorphisms between DGAs:
\[
\begin{tikzcd}
\displaystyle\coprod_{i=1}^k \cI_{\sfv_i}(\bL) \arrow[r,hook,"\coprod_{i=1}^k\bp_i"]\arrow[d,hook,"\coprod_{i=1}^k\bp'_i"] & \cA(\bL) \arrow[d,hook]\\
\cA(\bL)\arrow[r,hook] & \cA(D(\bL))\big/\left( \coprod_{i=1}^k\bp_i\sim\coprod_{i=1}^k\bp_i'\right)\arrow[r,equal,"\sim"] &\cA^{\mathsf{EN}}([D(\bL)])
\end{tikzcd}
\]

\subsubsection{DGAs for bordered Legendrian graphs}
Legendrian links in a bordered manifold and their associated DGAs were first considered in \cite{Sivek2011} via combinatorial methods and later in \cite{HS2014} with geometric interpretation.

Now we define a DGA for a bordered Legendrian graph $\bL=(\sL,\mu)$ of type $(\ell,r)$ with a Maslov potential.
Let $\hat\bL^{\mathsf{left}}$ be the concatenation 
\[
\hat\bL^{\mathsf{left}}\coloneqq\mathbf{0}_\ell(\iota_L^*(\mu))\cdot\bL
\]
called the {\em left closure} of $\bL$. Then we define the {\em Ng's resolution} $\res(\bL)$ for a bordered Legendrian graph $\bL$ as the resolution of the right-bordered Legendrian $\hat\bL^\mathsf{left}$, which can be regarded as a subdiagram of the resolution of the closure $\hat\bL$. See Figure~\ref{figure:left closure}.
\[
\res(\bL)\coloneqq\res(\hat\bL^\mathsf{left})\subset \res(\hat\bL).
\]

\begin{figure}[ht]
\[
\begin{tikzcd}
\vcenter{\hbox{\includegraphics{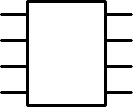}}}\arrow[r,"\hat{\cdot}^\mathsf{left}"] &
\vcenter{\hbox{\includegraphics{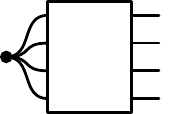}}}\arrow[r, "\res"] &
\res(\sL)=\vcenter{\hbox{\includegraphics{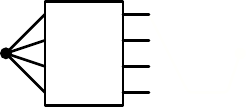}}}\\
\vcenter{\hbox{\includegraphics{L_front.pdf}}}\arrow[r,"\hat{\cdot}"] &
\vcenter{\hbox{\includegraphics{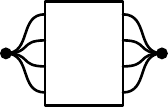}}}\arrow[r, "\res"] &
\res(\hat{\sL})=\vcenter{\hbox{\includegraphics{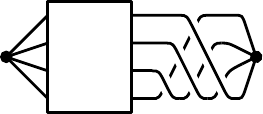}}}
\end{tikzcd}
\]
\caption{The closures and their resolutions}
\label{figure:left closure}
\end{figure}

\begin{definition}[DGAs for bordered Legendrian graphs]\label{definition:DGA for bordered Legendrian graphs}
Let $\bL=(\sL,\mu)$ be a bordered Legendrian graph with a Maslov potential. Then $\cA(\bL)$ is defined by the DGA construction for the Lagrangian projection $\res(\bL)$.
\end{definition}

Then, it is easy to see that $\cA(\bL)$ is generated by not only crossings and vertex generators in $\bL$, but also infinitely many generators $\{\mathsf{0}_{i,j}\mid i\in\Zmod{\ell},j>0\}$, where $\mathsf{0}$ is the vertex coming from the left-closure.

The right border of $\res(\bL)$ gives us an additional datum, a DGA morphism $\bp_\infty\colon \cI_\infty(\bL)\to\cA(\bL)$ of degree 0 defined as follows:
The DGA $\cI_\infty(\bL)$ is the DGA of the trivial bordered Legendrian $(\sfI_r, \iota_R^*(\mu))$, whose generators will be denoted by $\infty_{i,j}$'s
\[
\cI_\infty(\bL)\coloneqq\Z\langle
\infty_{i,j}\mid i\in\Zmod{r}, j>0
\rangle.
\]
The image of $\infty_{i,j}$ under $\bp_\infty$ is defined in a similar way to the differential $\partial$ so that $\bp_\infty$ counts immersed once-punctured $t$-gons contained in the neighborhood of $\res(\bL)$ as depicted in Figure~\ref{figure:disk at infinity}.
We regard that the spiral curves corresponding to $\infty_{i,j}$ are lying on the boundary of this neighborhood and each once-punctured immersed polygon converges to some $\infty_{i,j}$ near the puncture.

\begin{figure}[ht]
\subfigure[A Legendrian tangle]{\makebox[0.45\textwidth]{$
\vcenter{\hbox{\includegraphics{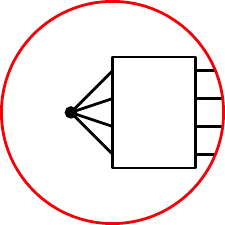}}}$}}
\subfigure[An immersed once-punctured polygon]{\makebox[0.45\textwidth]{$
\vcenter{\hbox{\includegraphics{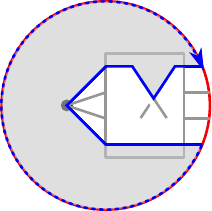}}}
$}}
\caption{Polygons in a Legendrian tangle}
\label{figure:disk at infinity}
\end{figure}

Then this disk counting defines a DGA morphism. See \cite[Section~6]{AB2018} for detail.
\begin{lemma}\cite[Lemma~6.10]{AB2018}
The map $\bp_\infty$ is a DGA morphism.
\end{lemma}

\begin{theorem}\label{theorem:equivalence of augmentations}
Let $\bL$ be a bordered Legendrian graph. Then $\cA(\bL)$ has a $\rho$-graded augmentation if and only if so does $\cA(\hat\bL)$.
\end{theorem}
\begin{proof}
In order to avoid the ambiguity, we denote the differential for $\cA(\hat\bL)$ by $\hat\partial$.

As seen in Figure~\ref{figure:left closure}, all crossings and vertices for $\res(\bL)$ are already contained in $\res(\hat\bL)$, which indeed induces the canonical DGA morphism
\[
\Phi\colon \cA(\bL)\to\cA(\hat\bL).
\]
It follows directly that $\aug(\cA(\bL),\Z)\neq\emptyset$ if $\aug(\cA(\hat\bL),\Z)\neq\emptyset$ by pre-composition of $\Phi$.

Suppose that we have an augmentation $\epsilon\colon \cA(\bL)\to\Z$. Then it suffices to extend $\epsilon$ to $\hat\epsilon\colon \cA(\widehat \bL)\to\Z$ by assigning values for the additional generators---both crossing $\{\sfa_{i,j}\mid 1\le i<j\le r\}$ and vertex generators $\{\infty_{i,j}\mid 1\le i\le r, j>0\}$ --- that come from the resolution part of the right closure part $\infty_r$. See Figure~\ref{figure:right closure resolution}.

As to the differential of $\sfa_{i,j}$, two types of disks --- indicated as $(A_k)$ and $(B_k)$ in Figure~\ref{figure:right closure resolution} --- contribute as follows:
\begin{align*}
{(A_k)}&:(-1)^{|\sfa_{i,j}|-1}\sfa_{i,k}\infty_{k,j-k};\\
{(B_k)}&:\bp_\infty(\infty_{i,k-i})\sfa_{k, j}.
\end{align*}
Therefore we have
\begin{align*}
\hat\partial\sfa_{i,j}&=
\bp_\infty(\infty_{i,j-i}) + (-1)^{|\sfa_{i,j}|-1}\infty_{i,j-i}\\
&\mathrel{\hphantom{=}}+\sum_{k=i+1}^{j-1} (-1)^{|\sfa_{i,j}|-1}\sfa_{i,k}\infty_{k,j-k} + \bp_\infty(\infty_{i,k-i})\sfa_{k, j}.
\end{align*}
Note that $|\sfa_{i,j}| = |\infty_{i,j-i}|+1$. 

On the other hand, the differential for $\infty_{i,j}$ is the same as before
\begin{align*}
\hat\partial\infty_{i,j}&=\delta_{j,r}+\sum_{j_1+j_2=j}(-1)^{|\infty_{i,j_1}|-1}\infty_{i,j_1}\infty_{i+j_1,j_2}\\
&=(-1)^{|\infty_{i,j}|-1}\delta_{j,r}+\sum_{j_1+j_2=j}(-1)^{|\infty_{i,j_1}|-1}\infty_{i,j_1}\infty_{i+j_1,j_2}.
\end{align*}
The last equality holds since $\delta_{j,r}=1$ if and only if $j=r$ and $|\infty_{i,r}|=1$.

We now extend $\epsilon$ to $\hat\epsilon$ by assigning values on $\sfa_{i,j}$ and $\infty_{i,j}$ as follows:
\begin{align*}
\hat\epsilon(\sfa_{i,j})&\coloneqq 0,&
\hat\epsilon(\infty_{i,j})&\coloneqq (-1)^{|\infty_{i,j}|-1}\epsilon(\bp_\infty(\infty_{i,j})).
\end{align*}

To show that $\hat\epsilon$ is an augmentation for $\cA(\hat\bL)$, it suffices to show that $\hat\epsilon$ commutes with differential. That is,
\[
\hat\epsilon\circ\hat\partial=0.
\]

From the direct computation, we have
\begin{align*}
(\hat\epsilon\circ\hat\partial)(\infty_{i,j})&=\delta_{j,r}+\sum_{j_1+j_2=j}(-1)^{|\infty_{i,j_1}|-1}\hat\epsilon(\infty_{i,j_1}\infty_{i+j_1,j_2})\\
&=(-1)^{|\infty_{i,j}|-1}\epsilon(\bp_\infty(\hat \partial\infty_{i,j}))\\
&=(-1)^{|\infty_{i,j}|-1}\epsilon(\partial\bp_\infty(\infty_{i,j}))=0.
\end{align*}
Here, we used that for $j_1+j_2=j$,
\[
(-1)^{|\infty_{i,j}|-1} = (-1)^{|\infty_{i,j_1}|-1}(-1)^{|\infty_{i+j_1,j_2}|-1}.
\]
Finally, for $\sfa_{i,j}$ we have
\begin{align*}
(\hat\epsilon\circ\hat\partial)(\sfa_{i,j})&=\epsilon(\bp_\infty(\infty_{i,j-i}))+(-1)^{|\sfa_{i,j}|-1}\hat\epsilon(\infty_{i,j-i})\\
&=\epsilon(\bp_\infty(\infty_{i,j-i}))+(-1)^{|\sfa_{i,j}|-1+|\infty_{i,j-i}|-1}\epsilon(\bp_\infty(\infty_{i,j-i}))\\
&=0
\end{align*}
since $|\sfa_{i,j}|=|\infty_{i,j-i}|+1$. Therefore $\hat\epsilon$ is a DGA morphism and we are done.
\end{proof}

\begin{figure}[ht]
\[
\begin{tikzpicture}[baseline=-.5ex]
\begin{scope}[xshift=-2cm]
\draw[thick](-1,1.5) node[left] {$1$} to[out=0,in=180] (1,0);
\draw[thick](-1,.5) to[out=0,in=180] (1,0);
\draw[thick](-1,-.5) to[out=0,in=180] (1,0);
\draw[thick](-1,-1.5) node[left] {$r$} to[out=0,in=180] (1,0);
\draw[fill] (1,0) circle (2pt) node[above] {$\infty$};
\end{scope}
\draw[->](-0.5,0)--(0.5,0);
\draw(0,0) node[above] {$\res$};
\begin{scope}[xshift=2cm]
\draw[thick,rounded corners](-1,-1.5) node[left] {$r$} -- (-0.5,-1.5) -- (1,1.5) -- (2.5,1.5) -- (4,0);
\draw[line width=8pt,white,rounded corners](-1,-.5) -- (-0.5,-.5) -- (0,-1.5) -- (0.5,-1.5) -- (1.5,.5) -- (2.5,0.5);
\draw[thick,rounded corners](-1,-.5) -- (-0.5,-.5) -- (0,-1.5) -- (0.5,-1.5) -- (1.5,.5) -- (2.5,0.5) -- (4,0);
\draw[line width=8pt,white,rounded corners](-1,.5) -- (0,.5) -- (1,-1.5) -- (1.5,-1.5) -- (2, -0.5) -- (2.5,-0.5);
\draw[thick,rounded corners](-1,.5) -- (0,.5) -- (1,-1.5) -- (1.5,-1.5) -- (2, -0.5) -- (2.5,-0.5) -- (4,0);
\draw[line width=8pt,white,rounded corners](-1,1.5) -- (0.5,1.5) -- (2,-1.5) -- (2.5,-1.5);
\draw[thick,rounded corners](-1,1.5) node[left] {$1$} -- (0.5,1.5) -- (2,-1.5) -- (2.5,-1.5) -- (4,0);
\draw[thick,dotted](-1,-1.2) -- (-1,-0.95);
\fill[blue,opacity=0.4](0.65,.8)--(0.85,1.2) arc (63.43:-116.57:0.224);
\fill[blue,opacity=0.4](1.65,-1.2)--(1.85,-.8) arc (63.43:-116.57:0.224);
\fill[blue,opacity=0.4](1.15,-.2)--(1.35,.2) arc (63.43:-116.57:0.224);
\fill[blue,opacity=0.4](-0.35,-1.2)--(-0.15,-.8) arc (63.43:-116.57:0.224);
\fill[blue,opacity=0.4](0.15,-.2)--(0.35,0.2) arc (63.43:-116.57:0.224);
\fill[blue,opacity=0.4](0.65,-1.2)--(0.85,-0.8) arc (63.43:-116.57:0.224);
\fill[lightgray,opacity=0.5,rounded corners] (-1,-0.5) --(-0.5,-0.5) -- (-0.25, -1) -- (0.25,0)-- (0,0.5) -- (-1,0.5);
\fill[lightgray,opacity=0.5,rounded corners] (4,0) -- (2.5,-0.5) -- (2,-0.5)-- (1.75,-1) -- (1.25, 0)  -- (1.5,0.5) -- (2.5,0.5) -- (4,0);
\draw (2.5,0) node {$(A_k)$};
\draw (-0.5,0) node {$(B_k)$};
\draw[fill] (4,0) circle (2pt) node[right] {$\infty$};
\draw (1.75,-1) node[right] {$\sfa_{1,2}$};
\draw (1.25,0) node[right] {$\sfa_{1,3}$};
\draw (0.75,1) node[right] {$\sfa_{1,r}$};
\draw (0.75,-1) node[right] {$\sfa_{2,3}$};
\draw (0.25,0) node[right] {$\sfa_{2,r}$};
\draw (-0.25,-1) node[right] {$\sfa_{3,r}$};
\draw[red,thick,-latex'] plot[variable=\t,domain=225:-225,samples=50] ({ (.5 - (\t-60) / 1800 ) * cos( \t ) + 4},{ (.5 - (\t-60) / 1800) * sin( \t )});
\draw (4,0.5) node[above] {$\infty_{i,j}$};
\end{scope}
\end{tikzpicture}
\]
\caption{The resolution of $\infty_r$ and generators}
\label{figure:right closure resolution}
\end{figure}
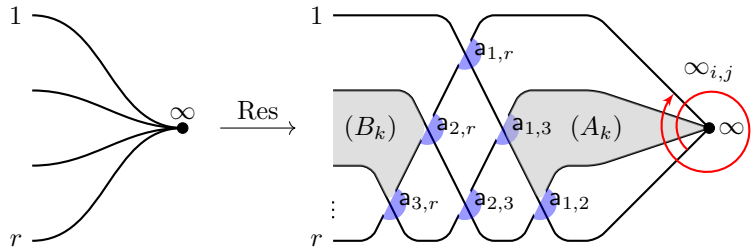

\subsubsection{Augmentations and rulings}
\begin{definition}[Augmentation]
An {\em augmentation} of a DGA $\cA$ over $\Z$ is a DGA morphism $\epsilon\colon \cA\to(\Z,|\cdot|\equiv0,\partial\equiv0)$. We say that $\epsilon$ is {\em $\rho$-graded} if $\fR=\Zmod{\rho}$.

We denote the set of all $\rho$-graded augmentations for $\cA$ over $\Z$ by $\aug^\rho(\cA,\Z)$.
\end{definition}

As mentioned earlier in this section, the existence of augmentation is related with the existence of normal rulings as follows:
\begin{theorem}\cite{Sabloff2005, Leverson2016, Leverson2017}\label{theorem:equivalence of existence for links}
For a Legendrian link $\bL=(\sL,\mu)$ in $\R^3$ or $M_k$, the $\rho$-graded normal ruling exists if and only if the $\rho$-graded augmentation exists for $\cA(\bL)$ or $\cA^{\mathsf{EN}}(\bL)$, respectively.
\end{theorem}

\begin{lemma}\label{lemma:equivalence of augmentations}
Let $\bL=(\sL,\mu)$ be a Legendrian graph in $\R^3$ with a Maslov potential with $k$ vertices. Then $\cA(\bL)$ has a $\rho$-graded augmentation if and only if so does the DGA $\cA([D(\bL)])$ for $[D(\bL)]\subset M_k$
\end{lemma}
\begin{proof}
This is obvious from the universal property of the pushout diagram.
\end{proof}

\begin{theorem}\label{theorem:equivalence of existence for graphs}
Let $\bL=(\sL,\mu)$ be a bordered Legendrian graph with a Maslov potential. Then the $\rho$-graded normal ruling for $\bL$ exists if and only if $\rho$-graded augmentation for $\cA(\bL)$ exists.
\end{theorem}
\begin{proof}
The theorem follows from Corollaries~\ref{corollary:rulings for bordered Legendrian graphs} and \ref{cor:double ruling}, Theorems~\ref{theorem:equivalence of augmentations} and \ref{theorem:equivalence of existence for links}, and Lemma~\ref{lemma:equivalence of augmentations}.
\end{proof}

Diagrammatically, one can present Theorem~\ref{theorem:equivalence of existence for graphs} as follows:
\[
\begin{tikzcd}[column sep=3pc, row sep=3pc]
\bR^\rho_\bL\neq\emptyset\arrow[d,Leftrightarrow,"\text{Cor.~\ref{corollary:rulings for bordered Legendrian graphs}}"'] 
\arrow[r,Leftrightarrow,dashed,"\text{Thm.~\ref{theorem:equivalence of existence for graphs}}"']&
\aug^\rho(\cA(\bL),\Z)\neq\emptyset\arrow[d,Leftrightarrow,"\text{Thm.~\ref{theorem:equivalence of augmentations}}"]\\
\bR^\rho_{\hat\bL}\neq\emptyset\arrow[d,Leftrightarrow,"\text{Cor.~\ref{cor:double ruling}}"']&
\aug^\rho(\cA(\hat\bL),\Z)\neq\emptyset\arrow[d,Leftrightarrow,"\text{Lem.~\ref{lemma:equivalence of augmentations}}"]\\
\bR^\rho_{[D(\hat\bL)]}\neq\emptyset\arrow[r,Leftrightarrow,"\text{Thm.~\ref{theorem:equivalence of existence for links}}"]&
\aug^\rho(\cA^{\mathsf{EN}}([D(\hat\bL)]),\Z)\neq\emptyset
\end{tikzcd}
\]

\subsection{Four-valent graphs and the Kauffman polynomial}

Now let us focus on four valent Legendrian graphs, which are the same as Legendrian singular links which have been studied in \cite{ABK2018}.

\begin{lemma}\label{lemma:skein_relation_of_ruling}
The 1-graded normal ruling polynomial $R_1$ satisfies the following skein relation:
\[
R_1\left(\Lcuspfour\right) = R_1\left(\Ldoublecusp\right) - (z-1) R_1\left(\Lstackedcusp\right) + R_1\left(\Lnestedcusp\right)
\]
\end{lemma}
\begin{proof}
As seen before, the full resolutions for $L_{0_4}=\Lcuspfour$ are
\[
\tilde L_{0_4}=\left\{\Ldoublecuspblack, \Lnestedcusp, \Lstackedcusp \right\}
\]
and so
\[
R_1\left(\Lcuspfour\right) = R_1\left(\Ldoublecuspblack\right) + R_1\left(\Lnestedcusp\right) + R_1\left(\Lstackedcusp\right).
\]
For a given crossing $\sfc \in \sC$, the set of normal ruling can be decomposed into two sets whether a normal ruling contains $\sfc$ or not.
Thus we have
\begin{align*}
R_1\left(\Ldoublecusp\right)=R_1\left(\Ldoublecuspblack\right) + z R_1\left(\Lstackedcusp\right)
\end{align*}
and therefore the claim is proved.
\end{proof}

\begin{definition}\cite{Kauffman1990}
Let $K=K_1\sqcup\cdots\sqcup K_n$ be an unoriented $n$-component link. The {\em unnormalized Kauffman polynomial} $[K]$ for a link $K$ is a polynomial of two variables $(a,z)$ which satisfies the following skein relation:
\begin{align*}
\left[\theunknot\right] &= 1,&
\left[\crossingpositive\right]-\left[ \crossingnegative\right] &= z\left(\left[\crossinghorizontal\right] - \left[\crossingvertical\right]\right),\\
\left[\positivekink\right] &= a \left[ \horizontalline \right],&
\left[\negativekink\right] &= a^{-1} \left[ \horizontalline \right].
\end{align*}

The {\em (normalized) Kauffman polynomial} $F_K$ for a link $K$ is defined to be
\begin{align*}
F_K&\coloneqq a^{-\mathbf{w}(K)}[K],&
\mathbf{w}(K)&\coloneqq \sum_{i=1}^n w(K_i),
\end{align*}
where $w(K_i)$ is the writhe of the component $K_i$ of $K$.
\end{definition}

Usually, the Kauffman polynomial $F_K$ is defined only for (unoriented) knots or oriented links since the notion of total writhe for unoriented link is ambiguious. However, it is still well-defined that the sum of component-wise writhes.
Therefore it is easy to see that $F_K$ is invariant under the ambient isotopy.
\begin{remark}
The polynomial $F_K$ is originally defined by Kauffman but denoted by $U_K$. See Page 13 in \cite{Kauffman1990}.
\end{remark}

For Legendrian links, there is a known degree bound of the Kauffman polynomial with respect to the variable $a$.
\begin{lemma}\cite{Rutherford2006}\label{lemma:upper bound for links}
For any Legendrian link $\sK$, the degree $\deg_a[\sK]$ is at most $-1$. Equivalently, 
\[
\deg_a F_\sK \le -1 - tb(\sK).
\]
\end{lemma}
\begin{remark}
Here, we are using a slightly different convention for the Kauffman polynomial of Legendrian links from \cite{Rutherford2006} since we consider the additional kink for each right cusp.

One of the benefit of our convention is that the upper bound is always $-1$. Compare this with Lemma~2.2 in \cite{Rutherford2006}, where the upper bound is given as $\#(\lcusp_\sL) -1$.
\end{remark}

\begin{theorem}\cite[Theorem~3.1]{Rutherford2006}\label{theorem:ruling and kauffman polynomial for links}
For a Legendrian knot $\sK$, the ungraded ruling polynomial is the same as the coefficient of $a^{-1-tb(\sK)}$ in the shifted Kauffman polynomial $z^{-1}F_\sK$.
\end{theorem}
\begin{remark}
Note that in \cite{Rutherford2006}, the weight convention for each normal ruling is
\[
z^{\#\{switches\}-\#\{eyes\}+1}
\]
and so there is no need to consider the shifted Kauffman polynomial.
\end{remark}

\begin{definition}\label{definition:4-valent spatial skein}\cite{Kauffman1989}
The unnormalized Kauffman polynomial $[\Gamma]$ for a 4-valent spatial graph $\Gamma$ is a polynomial of three variables $(a,A,B)$ which satisfies the additional skein relation:
\begin{align}\label{eq:4-valent spatial skein}
\left[\crossingsingular\right] &= \left[ \crossingpositive\right] - A\left[\crossinghorizontal\right] - B\left[\crossingvertical\right],
\end{align}
where $z=A-B$.

The unnormalized Kauffman polynomial $[\sL]$ for a 4-valent Legendrian graph $\sL$ is given by the Kauffman polynomial of the Ng's resolution of $\sL$.
\end{definition}
\begin{remark}
One can use the following skein relation for Kauffman polynomial for 4-valent graphs instead:
\begin{align*}
\left[\crossingsingular\right] &= 
\left[ \crossingnegative\right] - B\left[\crossinghorizontal\right] - A\left[\crossingvertical\right].
\end{align*}
\end{remark}

To define the (normalized) Kauffman polynomial for 4-valent graphs, we first resolve all vertices in a {\em virtual way}, that is, all 4-valent vertices will be regarded as virtual transverse crossings. The result will be a virtual link and denoted by 
\[
\sL^\otimes=\sL^\otimes_1\sqcup\cdots\sqcup \sL^\otimes_n.
\]
Then the component-wise writhes $w(\sL^\otimes_i)$ for the virtual link $\sL^\otimes$ are well-defined again. 
In practice, for each component $\sL^\otimes_i$, $w(\sL^\otimes_i)$ is the sum of signed {\em real} crossings.

\begin{definition}[Total writhe for Legendrian 4-valent graphs]
Let $\sL$ be a Legendrian 4-valent graph and $\sL^\otimes=\sL^\otimes_1\sqcup\cdots\sqcup\sL^\otimes_n$ be the virtual link obtained by the virtual resolution. The total writhe is defined as follows:
\[
\mathbf{w}(\sL)\coloneqq \sum_{i=1}^n w(\sL^\otimes_i).
\]
\end{definition}
Notice that if $\sL$ is a Legendrian knot, then the writhe of the Ng's resolution of $\sL$ is the same as Thurston-Bennequin number $tb(\sL)$. Therefore we may regard the total writhe $\mathbf{w}(\sL)$ as the total Thurston-Bennequin number $\mathbf{tb}(\sL)$.

\begin{example}
Consider the following $4$-valent Legendrian graph and its virtual resolution which is depicted by  the diagram with red circles:
\[
\begin{tikzcd}
\sL=\vcenter{\hbox{\includegraphics[scale=1]{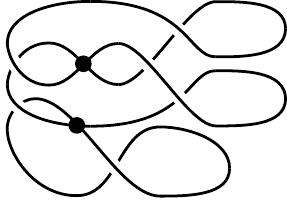}}}
\arrow[r]&
\sL^\otimes=\vcenter{\hbox{\includegraphics[scale=1]{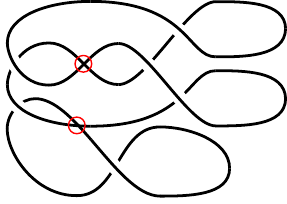}}}
\end{tikzcd}
\]
Let us denote the upper and lower component of $\sL^\otimes$ by $\sL_1^\otimes$ and $\sL_2^\otimes$, respectively. Then we have $w(\sL_1^\otimes )=0$, $w(\sL_2^\otimes)=-1$, and hence $\bw(\sL)=-1$.
\end{example}

\begin{definition}[Kauffman polynomials for spatial 4-valent graphs]
The (normalized) Kauffman polynomial $F_\sL$ for a spatial 4-valent graph $\sL$ is defined as
\[
F_\sL\coloneqq a^{-\mathbf{w}(\sL)}[\sL].
\]
\end{definition}
Then one can see that the Kauffman polynomial is invariant under the ambient isotopy and the above two results can be generalized to 4-valent Legendrian graphs as follows:
\begin{lemma}
The following holds: for any Legendrian 4-valent graph $\sL$,
\[
\deg_a [\sL] \le -1,
\]
or equivalently,
\[
\deg_a F_\sL \le -1 -\mathbf{tb}(\sL).
\]
\end{lemma}
\begin{proof}
Due to the skein relation \eqref{eq:4-valent spatial skein} for 4-valent graphs, we have
\begin{align*}
\deg_a\left[\crossingsingular\right]&= \deg_a\left(\left[\crossingpositive\right]-A\left[\crossinghorizontal\right] -B\left[\crossingvertical\right]\right)\\
&\le \max\left\{
\deg_a\left[\crossingpositive\right], \deg_a\left[\crossinghorizontal\right], \deg_a\left[\crossingvertical\right]
\right\}.
\end{align*}
By the induction on the number of vertices and Lemma~\ref{lemma:upper bound for links}, we are done.
\end{proof}

\begin{theorem}\label{theorem:ruling and kauffman polynomial for graphs}
Let $\sL$ be a regular front projection of a 4-valent Legendrian graph.
The ungraded $(\rho=1)$ ruling polynomial $R^1_\sL$ for $\sL$ is the same as the coefficient of $a^{-\mathbf{tb}(\sL)-1}$ ($a^{-1}$, resp.) in the shifted Kauffman polynomial $z^{-1}F_\sL$ (unnormalized polynomial $z^{-1}[\sL]$, resp.) after replacing $A$ and $B$ with $(z-1)$ and $-1$, respectively.
\end{theorem}

Simply speaking, this theorem implies the existence of a topological invariant for 4-valent spatial graphs which is a two variable polynomial of $a$ and $z$ whose certain coefficient of $a$ coincides with the ruling polynomial.

\begin{proof}[Proof of Theorem~\ref{theorem:ruling and kauffman polynomial for graphs}]
As seen in Lemma~\ref{lemma:skein_relation_of_ruling} and definition of Kauffman polynomial for 4-valent graphs, both satisfy the same skein relation after replacing $A$ and $B$ as above.
Therefore by using the induction on the number of vertices, we only need to consider Legendrian links which has been already covered by Theorem~\ref{theorem:ruling and kauffman polynomial for links}.
\end{proof}

\begin{example}
Let us consider the following Legendrian graph $\sL$ having a valency 4-vertex with three vertex resolutions as follows:
\[
\begin{tikzcd}[column sep=4pc]
\sL=\vcenter{\hbox{\scalebox{1.5}{}}}
\arrow[r,"\scriptstyle{\{\{1,3\},\{2,4\}\}}"]
\arrow[rd,"\scriptstyle{\{\{1,2\},\{3,4\}\}}"]
\arrow[d,"\scriptstyle{\{\{1,4\},\{2,3\}\}}"']
&
\sL_-=\vcenter{\hbox{\scalebox{1.5}{\includegraphics{ex_front_1.pdf}}}}\\
\sL_0=\vcenter{\hbox{\scalebox{1.5}{\includegraphics{ex_front_2.pdf}}}}&
\sL_\infty=\vcenter{\hbox{\scalebox{1.5}{\includegraphics{ex_front_3.pdf}}}}
\end{tikzcd}
\]
Since $R^1_{\sL_-}=z^{-1}$, $R^1_{\sL_0}=z^{-2}+1$, and $R^1_{\sL_\infty}=0$, we have $R^1_\sL=z^{-2}+z^{-1}+1$.

On the other hand, the corresponding Lagrangian projection of $\sL$ with its resolutions are the following:
\[
\let\decsc\DecorationScale
\renewcommand{\DecorationScale}{0.15}
\begin{tikzcd}[column sep=4pc]
K=\vcenter{\hbox{\includegraphics[scale=1]{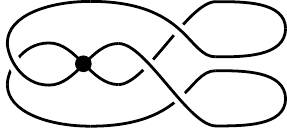}}}
\arrow[r,"\crossingsingular\mapsto\crossingnegative"]
\arrow[rd,"\crossingsingular\mapsto\crossingvertical"]
\arrow[d,"\crossingsingular\mapsto\crossinghorizontal"']
&
K_-=\vcenter{\hbox{\includegraphics[scale=1]{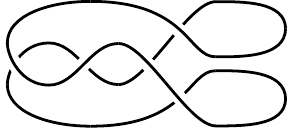}}}\\
K_0=\vcenter{\hbox{\includegraphics[scale=1]{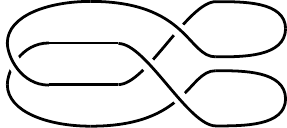}}}&
K_\infty=\vcenter{\hbox{\includegraphics[scale=1]{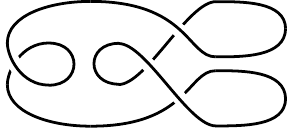}}}
\end{tikzcd}
\renewcommand{\DecorationScale}{\decsc}
\]
Thus we have
\begin{align*}
[K]&=[K_-] - B[K_0] - A[K_\infty]\\
&=a^{-1} - B(-z^{-1}a^{-3}-z a^{-3}+a^{-2} + z^{-1} a^{-1} + za^{-1})  - A a^{-4}.
\end{align*}
If we regard $K$ as a virtual knot, then it has two positive crossings and two negative crossings and so ${\bf w}(K)={\bf tb}(\sL)=0$. 
This implies that the shifted Kauffman polynomial $z^{-1}F_K$ with $A=z-1$ and $B=-1$ becomes
\[
(z^{-1}-1)a^{-4}+(-z^{-2}-z^{-1})a^{-3}+ z^{-1} a^{-2} + \boxed{(z^{-2}+ z^{-1} +1)}a^{-1}.
\]
Here we can check that $R^1_\sL$ appear in the coefficient of $a^{-\mathbf{tb}(\sL)-1}$.
\end{example}

\begin{question}
Can Theorem~\ref{theorem:ruling and kauffman polynomial for graphs} be generalized to arbitrary Legendrian graphs? 
Namely, does there exist a topological invariant for spatial graphs which is two variable polynomial of $a$ and $z$ whose certain coefficient of $a$ gives us the ruling polynomial $R^1$?
\end{question}

There is a partial answer to the above question for spatial graphs with vertices of valency at most six. However, for spatial graphs with a vertex of valency eight or higher, one should consider hundreds of resolutions and so it is not easy to determine the coefficients of the skein relation that resolving vertices, for example, the coefficients $A$ and $B$ in the skein relation \eqref{eq:4-valent spatial skein}.

\section{Proof of the invariance theorem}\label{sec:invariance}

Recall that the Legendrian half-twist braid $\Delta_b$ is defined inductively as follows:
\[
\Delta_b\coloneqq\vcenter{\hbox{\tiny
\begingroup%
  \makeatletter%
  \providecommand\color[2][]{%
    \errmessage{(Inkscape) Color is used for the text in Inkscape, but the package 'color.sty' is not loaded}%
    \renewcommand\color[2][]{}%
  }%
  \providecommand\transparent[1]{%
    \errmessage{(Inkscape) Transparency is used (non-zero) for the text in Inkscape, but the package 'transparent.sty' is not loaded}%
    \renewcommand\transparent[1]{}%
  }%
  \providecommand\rotatebox[2]{#2}%
  \ifx\svgwidth\undefined%
    \setlength{\unitlength}{45.00000073bp}%
    \ifx\svgscale\undefined%
      \relax%
    \else%
      \setlength{\unitlength}{\unitlength * \real{\svgscale}}%
    \fi%
  \else%
    \setlength{\unitlength}{\svgwidth}%
  \fi%
  \global\let\svgwidth\undefined%
  \global\let\svgscale\undefined%
  \makeatother%
  \begin{picture}(1,0.76666407)%
    \put(0,0){\includegraphics[width=\unitlength,page=1]{Delta_induction.pdf}}%
    \put(0.24999999,0.25833074){\color[rgb]{0,0,0}\makebox(0,0)[lb]{\smash{$\Delta_{b-1}$}}}%
  \end{picture}%
\endgroup%
}}~.
\]
By adding $(2n-b)$ trivial strands below $\Delta_b$, we obtain the bordered Legendrian link $\Delta_b^{2n}$ of type $(2n,2n)$.

\begin{lemma}
Let $\mathbf{\Delta}_b^{2n}\coloneqq((\Delta_b^{2n},\mu),\emptyset)$ be a bordered Legendrian link of type $(2n,2n)$ with a Maslov potential, where $\Delta_b$ is the half-twist braid of the upper $b$ strands among $2n$ strands. Then 
\[
R^\rho_{\mathbf{\Delta}^{2n}_b}(\phi,\psi)=0
\]
if $\phi$ or $\psi$ matches at least one pair of the first $b$ strands.
\end{lemma}
\begin{proof}
One can prove that if a normal ruling whose boundary matches the first strand with a strand in $\Delta_{b-1}$, then it implies that the existence of a normal ruling whose boundary matches two strands in $\Delta_{b-1}$, which is a contradiction. We omit the detail.
\end{proof}

Let $\beta$ be a permutation braid. Then $\bar\beta\beta$ is {\em palindromic}, that is the same as its reverse, and moreover it is {\em pure}, meaning that its induced permutation is the identity. 
Moreover, $\tau(\beta)\coloneqq\Delta_b\beta\Delta_b^{-1}$ is again a permutation braid and can be regarded as a braid obtained from $\beta$ by horizontal reflection.
\[
\tau\colon\beta=\vcenter{\hbox{\includegraphics{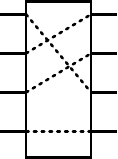}}}\longmapsto
\tau(\beta)=\vcenter{\hbox{\includegraphics{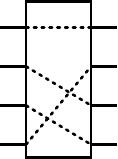}}}
\]

Recall that the complement $\beta^c$ of $\beta$ is defined as $\beta^{-1}\Delta$. It is easy to check that
\[
\beta\beta^c\bar{\beta^c} = \Delta \bar{\beta^c} = \tau(\bar{\beta^c})\Delta = \tau(\bar{\beta^c}) \tau(\beta^c) \beta.
\]
Indeed, the two Legendrian graphs $\sL$ and $\sL'$ defined as 
\[
\sL\coloneqq\sL_{\beta}\sL_{\beta^c \bar{\beta^c}},\quad\text{and}\quad
\sL'\coloneqq\sL_{\tau(\bar{\beta^c})\tau(\beta^c)}\sL_{\beta}
\]
are Legendrian isotopic.
Therefore any Maslov potential $\mu$ on $\sL$ induces a Maslov potential $\mu'$ on $\sL'$ and {\it vice versa}.
Let 
\[
\bL\coloneqq((\sL, \mu), \sC(\sL_{\beta^c \bar{\beta^c}})),
\quad\text{and}\quad
\bL'\coloneqq((\sL', \mu'), \sC(\sL_{\tau(\bar{\beta^c})\tau(\beta^c)})).
\]

\begin{lemma}\label{lemma:delta conjugate}
For any $\phi,\psi\in\cP_{[2n]}^\rho$, there is a weight-preserving bijection between the sets of normal rulings of $\sL\coloneqq\sL_{\beta}\sL_{\beta^c \bar{\beta^c}}$ and $\sL'\coloneqq\sL_{\tau(\bar{\beta^c})\tau(\beta^c)}\sL_{\beta}$.
\begin{align*}
R^\rho_\bL(\phi,\psi)
&\simeq R^\rho_{\bL'}(\phi,\psi)\\
S_\fr&\mapsto S_\fr
\end{align*}
\end{lemma}
\begin{proof}
It is easy to see
\[
\beta \beta^c\bar{\beta^c}=\Delta_b\bar{\beta^c}=\tau(\bar{\beta^c})\Delta_b=\tau(\bar{\beta^c})\tau(\beta^c)\beta,
\]
and therefore $\sL_{\beta}\sL_{\beta^c \bar{\beta^c}}=\sL_{\tau(\bar{\beta^c})\tau(\beta^c)}\sL_{\beta}$ as bordered Legendrian links.

For a convenience sake, let 
\begin{align*}
R&\coloneqq R_{\rho,(\phi,\psi)}(\sL_{\beta}\sL_{\beta^c \bar{\beta^c}},\sC({\sL_{\beta^c \bar{\beta^c}}})),\\
R'&\coloneqq R_{\rho,(\phi,\psi)}(\sL_{\tau(\bar{\beta^c})\tau(\beta^c)}\sL_{\beta}, \sC({\sL_{\tau(\bar{\beta^c})\tau(\beta^c)}}))
\end{align*}
Then since both $\sL=\Delta_b \sL_{\bar{\beta^c}}$ and $\sL'=\sL_{\tau(\bar{\beta^c})}\Delta_b$ contain $\Delta_b$, $R$ and $R'$ should be emptyset if $\phi$ or $\psi$ match two braid strands, and so we assume that both $\phi$ and $\psi$ have no $\{i,j\}$ with $i,j\in[b]$.

For these choices of $\phi$ and $\psi$, all crossings of the pure braids $\beta^c \bar{\beta^c}$ and $\tau(\bar{\beta^c})\tau(\beta^c)$ are marked and it plays exactly the same role as the identity. Therefore we have bijections
\[
R\simeq R_{\rho,(\phi,\psi)}(\sL_{\beta},\emptyset)\simeq R'.\qedhere
\]
\end{proof}

The following proposition is equivalent to Theorem~\ref{theorem:invariance for graphs}.
\begin{proposition}
Let $\sL'$ and $\sL''$ be two bordered Legendrian graphs different by one of the Reidemeister move. Then there is a weight-preserving bijection between $\cR_{\rho,(\phi,\psi)}(\sL')$ and $\cR_{\rho,(\phi,\psi)}(\sL'')$ for each $(\phi,\psi)\in\cP_{[\ell]}\times \cP_{[r]}$.
\end{proposition}
Since invariance under the usual Reidemeister moves is already established, it suffices to consider Reidemeister moves involving vertices. that is, $\rm(0_c)$, $\rm(0_e)$, $\rm(0_f)$, $\rm(IV)$, and $\rm(V)$.
Moreover, since we consider only resolutions of vertices, we need to prove that these Reidemeister moves commute with vertex resolutions.
For example, we need to prove that for each perfect pairing $\phi\in\cP_{[2n]}$ with $2n=\val(\sfv)$, the induced move $\rm(0_c)_*$ below yields a bijection between sets of normal rulings:
\[
\begin{tikzcd}[column sep=3pc]
\vcenter{\hbox{\includegraphics{R0_NV_1.pdf}}}\arrow[r,leftrightarrow,"\rm(0_c)"]\arrow[d,"\phi"'] & \vcenter{\hbox{\includegraphics{R0_NV_2.pdf}}}\arrow[d,"\phi"]\\
\vcenter{\hbox{\tiny
\begingroup%
  \makeatletter%
  \providecommand\color[2][]{%
    \errmessage{(Inkscape) Color is used for the text in Inkscape, but the package 'color.sty' is not loaded}%
    \renewcommand\color[2][]{}%
  }%
  \providecommand\transparent[1]{%
    \errmessage{(Inkscape) Transparency is used (non-zero) for the text in Inkscape, but the package 'transparent.sty' is not loaded}%
    \renewcommand\transparent[1]{}%
  }%
  \providecommand\rotatebox[2]{#2}%
  \ifx\svgwidth\undefined%
    \setlength{\unitlength}{37.49999906bp}%
    \ifx\svgscale\undefined%
      \relax%
    \else%
      \setlength{\unitlength}{\unitlength * \real{\svgscale}}%
    \fi%
  \else%
    \setlength{\unitlength}{\svgwidth}%
  \fi%
  \global\let\svgwidth\undefined%
  \global\let\svgscale\undefined%
  \makeatother%
  \begin{picture}(1,0.72)%
    \put(0,0){\includegraphics[width=\unitlength,page=1]{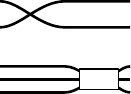}}%
    \put(0.68,0.09000001){\color[rgb]{0,0,0}\makebox(0,0)[lb]{\smash{$\sfv_\phi$}}}%
  \end{picture}%
\endgroup%
}}\ar[r,leftrightarrow,"\rm(0_c)_*"] & \vcenter{\hbox{\tiny
\begingroup%
  \makeatletter%
  \providecommand\color[2][]{%
    \errmessage{(Inkscape) Color is used for the text in Inkscape, but the package 'color.sty' is not loaded}%
    \renewcommand\color[2][]{}%
  }%
  \providecommand\transparent[1]{%
    \errmessage{(Inkscape) Transparency is used (non-zero) for the text in Inkscape, but the package 'transparent.sty' is not loaded}%
    \renewcommand\transparent[1]{}%
  }%
  \providecommand\rotatebox[2]{#2}%
  \ifx\svgwidth\undefined%
    \setlength{\unitlength}{37.49999906bp}%
    \ifx\svgscale\undefined%
      \relax%
    \else%
      \setlength{\unitlength}{\unitlength * \real{\svgscale}}%
    \fi%
  \else%
    \setlength{\unitlength}{\svgwidth}%
  \fi%
  \global\let\svgwidth\undefined%
  \global\let\svgscale\undefined%
  \makeatother%
  \begin{picture}(1,0.72)%
    \put(0,0){\includegraphics[width=\unitlength,page=1]{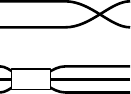}}%
    \put(0.16031087,0.09000001){\color[rgb]{0,0,0}\makebox(0,0)[lb]{\smash{$\sfv_\phi$}}}%
  \end{picture}%
\endgroup%
}}
\end{tikzcd}
\]

From now on, we simply use $(L,B)=(L',B')$ if there is a weight-preserving bijection between sets of normal rulings. 

\begin{lemma}
All marked Reidemeister moves induce a weight-preserving bijection between sets of normal rulings.
In other words, for each move $(\mathrm{M})$ between $(L,B)$ and $(L',B')$, we have
\[
(L,B) = (L',B').
\]
\end{lemma}
\begin{proof}
This is easy to check and we omit the proof.
\end{proof}

As seen in Remark~\ref{remark:marked Reidemeister move T}, the diagram with two markings is not the same as the diagram without any crossings.
Indeed, the difference between no crossings and two markings is whether the eye involving two strands is allowed or not.
\[
({\rm O})\quad\vcenter{\hbox{\includegraphics{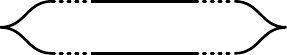}}}\qquad\qquad
({\rm X})\quad\vcenter{\hbox{\includegraphics{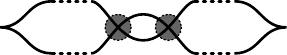}}}
\]

One of the direct consequence is that the moves similar to Reidemeister moves $(0)$ and $\rm(II)$ holds for two consecutive markings in the following sense:
\begin{align*}
\vcenter{\hbox{\includegraphics{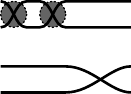}}}&=
\vcenter{\hbox{\includegraphics{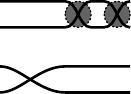}}}&
\vcenter{\hbox{\includegraphics{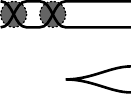}}}&=
\vcenter{\hbox{\includegraphics{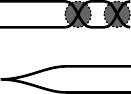}}}&
\vcenter{\hbox{\includegraphics{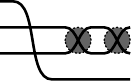}}}&=
\vcenter{\hbox{\includegraphics{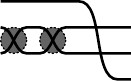}}}
\end{align*}

In general, any positive pure braid $\beta$ consisting of markings has the similar property.

\begin{lemma}\label{lemma:black pure braid commutes}
Suppose that $\beta$ is a positive pure braid consisting of markings. The following holds.
\begin{align*}
\vcenter{\hbox{
\begingroup%
  \makeatletter%
  \providecommand\color[2][]{%
    \errmessage{(Inkscape) Color is used for the text in Inkscape, but the package 'color.sty' is not loaded}%
    \renewcommand\color[2][]{}%
  }%
  \providecommand\transparent[1]{%
    \errmessage{(Inkscape) Transparency is used (non-zero) for the text in Inkscape, but the package 'transparent.sty' is not loaded}%
    \renewcommand\transparent[1]{}%
  }%
  \providecommand\rotatebox[2]{#2}%
  \ifx\svgwidth\undefined%
    \setlength{\unitlength}{37.50000006bp}%
    \ifx\svgscale\undefined%
      \relax%
    \else%
      \setlength{\unitlength}{\unitlength * \real{\svgscale}}%
    \fi%
  \else%
    \setlength{\unitlength}{\svgwidth}%
  \fi%
  \global\let\svgwidth\undefined%
  \global\let\svgscale\undefined%
  \makeatother%
  \begin{picture}(1,1.21999998)%
    \put(0,0){\includegraphics[width=\unitlength,page=1]{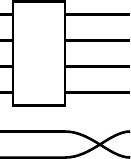}}%
    \put(0.19999999,0.80999999){\color[rgb]{0,0,0}\makebox(0,0)[lb]{\smash{$\beta$}}}%
  \end{picture}%
\endgroup%
}}&=
\vcenter{\hbox{
\begingroup%
  \makeatletter%
  \providecommand\color[2][]{%
    \errmessage{(Inkscape) Color is used for the text in Inkscape, but the package 'color.sty' is not loaded}%
    \renewcommand\color[2][]{}%
  }%
  \providecommand\transparent[1]{%
    \errmessage{(Inkscape) Transparency is used (non-zero) for the text in Inkscape, but the package 'transparent.sty' is not loaded}%
    \renewcommand\transparent[1]{}%
  }%
  \providecommand\rotatebox[2]{#2}%
  \ifx\svgwidth\undefined%
    \setlength{\unitlength}{37.50000006bp}%
    \ifx\svgscale\undefined%
      \relax%
    \else%
      \setlength{\unitlength}{\unitlength * \real{\svgscale}}%
    \fi%
  \else%
    \setlength{\unitlength}{\svgwidth}%
  \fi%
  \global\let\svgwidth\undefined%
  \global\let\svgscale\undefined%
  \makeatother%
  \begin{picture}(1,1.21999997)%
    \put(0,0){\includegraphics[width=\unitlength,page=1]{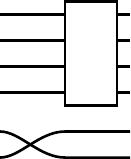}}%
    \put(0.59542886,0.80999998){\color[rgb]{0,0,0}\makebox(0,0)[lb]{\smash{$\beta$}}}%
  \end{picture}%
\endgroup%
}}&
\vcenter{\hbox{
\begingroup%
  \makeatletter%
  \providecommand\color[2][]{%
    \errmessage{(Inkscape) Color is used for the text in Inkscape, but the package 'color.sty' is not loaded}%
    \renewcommand\color[2][]{}%
  }%
  \providecommand\transparent[1]{%
    \errmessage{(Inkscape) Transparency is used (non-zero) for the text in Inkscape, but the package 'transparent.sty' is not loaded}%
    \renewcommand\transparent[1]{}%
  }%
  \providecommand\rotatebox[2]{#2}%
  \ifx\svgwidth\undefined%
    \setlength{\unitlength}{37.50000006bp}%
    \ifx\svgscale\undefined%
      \relax%
    \else%
      \setlength{\unitlength}{\unitlength * \real{\svgscale}}%
    \fi%
  \else%
    \setlength{\unitlength}{\svgwidth}%
  \fi%
  \global\let\svgwidth\undefined%
  \global\let\svgscale\undefined%
  \makeatother%
  \begin{picture}(1,1.21999998)%
    \put(0,0){\includegraphics[width=\unitlength,page=1]{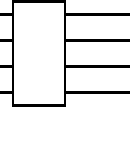}}%
    \put(0.19999999,0.80999999){\color[rgb]{0,0,0}\makebox(0,0)[lb]{\smash{$\beta$}}}%
    \put(0,0){\includegraphics[width=\unitlength,page=2]{R0_BrC_1.pdf}}%
  \end{picture}%
\endgroup%
}}&=
\vcenter{\hbox{
\begingroup%
  \makeatletter%
  \providecommand\color[2][]{%
    \errmessage{(Inkscape) Color is used for the text in Inkscape, but the package 'color.sty' is not loaded}%
    \renewcommand\color[2][]{}%
  }%
  \providecommand\transparent[1]{%
    \errmessage{(Inkscape) Transparency is used (non-zero) for the text in Inkscape, but the package 'transparent.sty' is not loaded}%
    \renewcommand\transparent[1]{}%
  }%
  \providecommand\rotatebox[2]{#2}%
  \ifx\svgwidth\undefined%
    \setlength{\unitlength}{37.50000006bp}%
    \ifx\svgscale\undefined%
      \relax%
    \else%
      \setlength{\unitlength}{\unitlength * \real{\svgscale}}%
    \fi%
  \else%
    \setlength{\unitlength}{\svgwidth}%
  \fi%
  \global\let\svgwidth\undefined%
  \global\let\svgscale\undefined%
  \makeatother%
  \begin{picture}(1,1.21999997)%
    \put(0,0){\includegraphics[width=\unitlength,page=1]{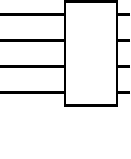}}%
    \put(0.59542886,0.80999998){\color[rgb]{0,0,0}\makebox(0,0)[lb]{\smash{$\beta$}}}%
    \put(0,0){\includegraphics[width=\unitlength,page=2]{R0_BrC_2.pdf}}%
  \end{picture}%
\endgroup%
}}&
\vcenter{\hbox{
\begingroup%
  \makeatletter%
  \providecommand\color[2][]{%
    \errmessage{(Inkscape) Color is used for the text in Inkscape, but the package 'color.sty' is not loaded}%
    \renewcommand\color[2][]{}%
  }%
  \providecommand\transparent[1]{%
    \errmessage{(Inkscape) Transparency is used (non-zero) for the text in Inkscape, but the package 'transparent.sty' is not loaded}%
    \renewcommand\transparent[1]{}%
  }%
  \providecommand\rotatebox[2]{#2}%
  \ifx\svgwidth\undefined%
    \setlength{\unitlength}{37.50000006bp}%
    \ifx\svgscale\undefined%
      \relax%
    \else%
      \setlength{\unitlength}{\unitlength * \real{\svgscale}}%
    \fi%
  \else%
    \setlength{\unitlength}{\svgwidth}%
  \fi%
  \global\let\svgwidth\undefined%
  \global\let\svgscale\undefined%
  \makeatother%
  \begin{picture}(1,1.02)%
    \put(0,0){\includegraphics[width=\unitlength,page=1]{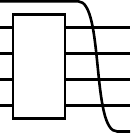}}%
    \put(0.19999999,0.50999999){\color[rgb]{0,0,0}\makebox(0,0)[lb]{\smash{$\beta$}}}%
  \end{picture}%
\endgroup%
}}&=
\vcenter{\hbox{
\begingroup%
  \makeatletter%
  \providecommand\color[2][]{%
    \errmessage{(Inkscape) Color is used for the text in Inkscape, but the package 'color.sty' is not loaded}%
    \renewcommand\color[2][]{}%
  }%
  \providecommand\transparent[1]{%
    \errmessage{(Inkscape) Transparency is used (non-zero) for the text in Inkscape, but the package 'transparent.sty' is not loaded}%
    \renewcommand\transparent[1]{}%
  }%
  \providecommand\rotatebox[2]{#2}%
  \ifx\svgwidth\undefined%
    \setlength{\unitlength}{37.50000006bp}%
    \ifx\svgscale\undefined%
      \relax%
    \else%
      \setlength{\unitlength}{\unitlength * \real{\svgscale}}%
    \fi%
  \else%
    \setlength{\unitlength}{\svgwidth}%
  \fi%
  \global\let\svgwidth\undefined%
  \global\let\svgscale\undefined%
  \makeatother%
  \begin{picture}(1,1.02)%
    \put(0,0){\includegraphics[width=\unitlength,page=1]{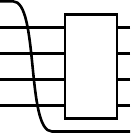}}%
    \put(0.58925607,0.50736269){\color[rgb]{0,0,0}\makebox(0,0)[lb]{\smash{$\beta$}}}%
  \end{picture}%
\endgroup%
}}
\end{align*}
\end{lemma}
\begin{proof}
This is obvious.
\end{proof}

On the other hand, marked cusps with markings have exactly the same property.

\begin{lemma}\label{lemma:marked cusp commutes}
Let $n\ge 0$. Then the following holds.
\begin{align*}
\vcenter{\hbox{\def\svgscale{0.9}
\begingroup%
  \makeatletter%
  \providecommand\color[2][]{%
    \errmessage{(Inkscape) Color is used for the text in Inkscape, but the package 'color.sty' is not loaded}%
    \renewcommand\color[2][]{}%
  }%
  \providecommand\transparent[1]{%
    \errmessage{(Inkscape) Transparency is used (non-zero) for the text in Inkscape, but the package 'transparent.sty' is not loaded}%
    \renewcommand\transparent[1]{}%
  }%
  \providecommand\rotatebox[2]{#2}%
  \ifx\svgwidth\undefined%
    \setlength{\unitlength}{44.01742442bp}%
    \ifx\svgscale\undefined%
      \relax%
    \else%
      \setlength{\unitlength}{\unitlength * \real{\svgscale}}%
    \fi%
  \else%
    \setlength{\unitlength}{\svgwidth}%
  \fi%
  \global\let\svgwidth\undefined%
  \global\let\svgscale\undefined%
  \makeatother%
  \begin{picture}(1,0.9541676)%
    \put(0,0){\includegraphics[width=\unitlength,page=1]{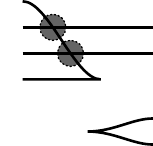}}%
    \put(-0.15707173,0.64902402){\color[rgb]{0,0,0}\makebox(0,0)[lb]{\smash{$n\left\{\vphantom{\rule{1cm}{.7pc}}\right.$}}}%
  \end{picture}%
\endgroup%
}}&=\ 
\vcenter{\hbox{\def\svgscale{0.9}
\begingroup%
  \makeatletter%
  \providecommand\color[2][]{%
    \errmessage{(Inkscape) Color is used for the text in Inkscape, but the package 'color.sty' is not loaded}%
    \renewcommand\color[2][]{}%
  }%
  \providecommand\transparent[1]{%
    \errmessage{(Inkscape) Transparency is used (non-zero) for the text in Inkscape, but the package 'transparent.sty' is not loaded}%
    \renewcommand\transparent[1]{}%
  }%
  \providecommand\rotatebox[2]{#2}%
  \ifx\svgwidth\undefined%
    \setlength{\unitlength}{44.01742448bp}%
    \ifx\svgscale\undefined%
      \relax%
    \else%
      \setlength{\unitlength}{\unitlength * \real{\svgscale}}%
    \fi%
  \else%
    \setlength{\unitlength}{\svgwidth}%
  \fi%
  \global\let\svgwidth\undefined%
  \global\let\svgscale\undefined%
  \makeatother%
  \begin{picture}(1,0.95416759)%
    \put(0,0){\includegraphics[width=\unitlength,page=1]{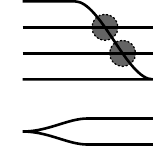}}%
    \put(-0.15707173,0.64902402){\color[rgb]{0,0,0}\makebox(0,0)[lb]{\smash{$n\left\{\vphantom{\rule{1cm}{.7pc}}\right.$}}}%
  \end{picture}%
\endgroup%
}}&
\vcenter{\hbox{\def\svgscale{0.9}
\begingroup%
  \makeatletter%
  \providecommand\color[2][]{%
    \errmessage{(Inkscape) Color is used for the text in Inkscape, but the package 'color.sty' is not loaded}%
    \renewcommand\color[2][]{}%
  }%
  \providecommand\transparent[1]{%
    \errmessage{(Inkscape) Transparency is used (non-zero) for the text in Inkscape, but the package 'transparent.sty' is not loaded}%
    \renewcommand\transparent[1]{}%
  }%
  \providecommand\rotatebox[2]{#2}%
  \ifx\svgwidth\undefined%
    \setlength{\unitlength}{44.01742348bp}%
    \ifx\svgscale\undefined%
      \relax%
    \else%
      \setlength{\unitlength}{\unitlength * \real{\svgscale}}%
    \fi%
  \else%
    \setlength{\unitlength}{\svgwidth}%
  \fi%
  \global\let\svgwidth\undefined%
  \global\let\svgscale\undefined%
  \makeatother%
  \begin{picture}(1,0.95416761)%
    \put(0,0){\includegraphics[width=\unitlength,page=1]{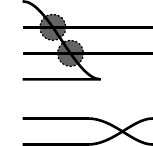}}%
    \put(-0.15707173,0.64902402){\color[rgb]{0,0,0}\makebox(0,0)[lb]{\smash{$n\left\{\vphantom{\rule{1cm}{.7pc}}\right.$}}}%
  \end{picture}%
\endgroup%
}}&=\ 
\vcenter{\hbox{\def\svgscale{0.9}
\begingroup%
  \makeatletter%
  \providecommand\color[2][]{%
    \errmessage{(Inkscape) Color is used for the text in Inkscape, but the package 'color.sty' is not loaded}%
    \renewcommand\color[2][]{}%
  }%
  \providecommand\transparent[1]{%
    \errmessage{(Inkscape) Transparency is used (non-zero) for the text in Inkscape, but the package 'transparent.sty' is not loaded}%
    \renewcommand\transparent[1]{}%
  }%
  \providecommand\rotatebox[2]{#2}%
  \ifx\svgwidth\undefined%
    \setlength{\unitlength}{44.01742448bp}%
    \ifx\svgscale\undefined%
      \relax%
    \else%
      \setlength{\unitlength}{\unitlength * \real{\svgscale}}%
    \fi%
  \else%
    \setlength{\unitlength}{\svgwidth}%
  \fi%
  \global\let\svgwidth\undefined%
  \global\let\svgscale\undefined%
  \makeatother%
  \begin{picture}(1,0.9541676)%
    \put(0,0){\includegraphics[width=\unitlength,page=1]{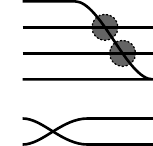}}%
    \put(-0.15707173,0.64902402){\color[rgb]{0,0,0}\makebox(0,0)[lb]{\smash{$n\left\{\vphantom{\rule{1cm}{.7pc}}\right.$}}}%
  \end{picture}%
\endgroup%
}}&
\vcenter{\hbox{\def\svgscale{0.9}
\begingroup%
  \makeatletter%
  \providecommand\color[2][]{%
    \errmessage{(Inkscape) Color is used for the text in Inkscape, but the package 'color.sty' is not loaded}%
    \renewcommand\color[2][]{}%
  }%
  \providecommand\transparent[1]{%
    \errmessage{(Inkscape) Transparency is used (non-zero) for the text in Inkscape, but the package 'transparent.sty' is not loaded}%
    \renewcommand\transparent[1]{}%
  }%
  \providecommand\rotatebox[2]{#2}%
  \ifx\svgwidth\undefined%
    \setlength{\unitlength}{44.01742448bp}%
    \ifx\svgscale\undefined%
      \relax%
    \else%
      \setlength{\unitlength}{\unitlength * \real{\svgscale}}%
    \fi%
  \else%
    \setlength{\unitlength}{\svgwidth}%
  \fi%
  \global\let\svgwidth\undefined%
  \global\let\svgscale\undefined%
  \makeatother%
  \begin{picture}(1,0.86897406)%
    \put(0,0){\includegraphics[width=\unitlength,page=1]{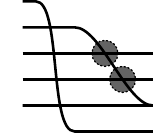}}%
    \put(-0.15707173,0.3834434){\color[rgb]{0,0,0}\makebox(0,0)[lb]{\smash{$n\left\{\vphantom{\rule{1cm}{.7pc}}\right.$}}}%
  \end{picture}%
\endgroup%
}}&=\ 
\vcenter{\hbox{\def\svgscale{0.9}
\begingroup%
  \makeatletter%
  \providecommand\color[2][]{%
    \errmessage{(Inkscape) Color is used for the text in Inkscape, but the package 'color.sty' is not loaded}%
    \renewcommand\color[2][]{}%
  }%
  \providecommand\transparent[1]{%
    \errmessage{(Inkscape) Transparency is used (non-zero) for the text in Inkscape, but the package 'transparent.sty' is not loaded}%
    \renewcommand\transparent[1]{}%
  }%
  \providecommand\rotatebox[2]{#2}%
  \ifx\svgwidth\undefined%
    \setlength{\unitlength}{44.01742448bp}%
    \ifx\svgscale\undefined%
      \relax%
    \else%
      \setlength{\unitlength}{\unitlength * \real{\svgscale}}%
    \fi%
  \else%
    \setlength{\unitlength}{\svgwidth}%
  \fi%
  \global\let\svgwidth\undefined%
  \global\let\svgscale\undefined%
  \makeatother%
  \begin{picture}(1,0.86897406)%
    \put(0,0){\includegraphics[width=\unitlength,page=1]{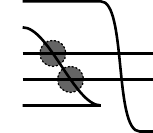}}%
    \put(-0.15707173,0.3834434){\color[rgb]{0,0,0}\makebox(0,0)[lb]{\smash{$n\left\{\vphantom{\rule{1cm}{.7pc}}\right.$}}}%
  \end{picture}%
\endgroup%
}}
\end{align*}
\end{lemma}
\begin{proof}
This is obvious.
\end{proof}

\begin{corollary}
All Reidemeister moves of types $(0)$ and $\rm(IV)$ induce weight-preserving bijections between sets of normal rulings.
\end{corollary}
\begin{proof}
Any resolution of a vertex is a product of marked left cusps with markings, a normal braid, a pure braid with markings and marked right cusps, and these commute with cusps, regular crossings, and long arc passing over (or under) the vertex by Lemma~\ref{lemma:black pure braid commutes} and Lemma~\ref{lemma:marked cusp commutes}.
\end{proof}

\subsection{\texorpdfstring{Reidemeister move $\rm(V)$}{The fourth Reidemeister move}}
Now, let us consider the Reidemeister move $\rm(V)$.
\[
\vcenter{\hbox{\includegraphics{R5_1.pdf}}}\stackrel{\rm(V)}\longleftrightarrow
\vcenter{\hbox{\includegraphics{R5_1_2.pdf}}}
\]

Let $\sfv$ be a vertex of type $(\ell,r)$ with $\ell+r=2n$, and $\phi\in\cP_{[2n]}$ with $\val\sfv = 2n$.
We denote the top-left arc by $\alpha$.
Then according to where the top-left arc $\alpha$ is matched, we have two cases:
\begin{enumerate}
\item $\alpha$ is a {\em strand} of a braid if it matches with an arc in the right, or
\[
\vcenter{\hbox{\includegraphics{R5_1.pdf}}}\stackrel{\phi}\longrightarrow
\vcenter{\hbox{
\begingroup%
  \makeatletter%
  \providecommand\color[2][]{%
    \errmessage{(Inkscape) Color is used for the text in Inkscape, but the package 'color.sty' is not loaded}%
    \renewcommand\color[2][]{}%
  }%
  \providecommand\transparent[1]{%
    \errmessage{(Inkscape) Transparency is used (non-zero) for the text in Inkscape, but the package 'transparent.sty' is not loaded}%
    \renewcommand\transparent[1]{}%
  }%
  \providecommand\rotatebox[2]{#2}%
  \ifx\svgwidth\undefined%
    \setlength{\unitlength}{157.50000148bp}%
    \ifx\svgscale\undefined%
      \relax%
    \else%
      \setlength{\unitlength}{\unitlength * \real{\svgscale}}%
    \fi%
  \else%
    \setlength{\unitlength}{\svgwidth}%
  \fi%
  \global\let\svgwidth\undefined%
  \global\let\svgscale\undefined%
  \makeatother%
  \begin{picture}(1,0.29047638)%
    \put(0,0){\includegraphics[width=\unitlength,page=1]{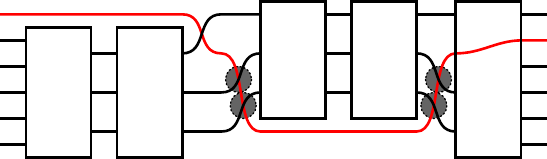}}%
    \put(0.07142858,0.07380969){\color[rgb]{0,0,0}\makebox(0,0)[lb]{\smash{$\sL'_L$}}}%
    \put(0.85714279,0.07380969){\color[rgb]{0,0,0}\makebox(0,0)[lb]{\smash{$\sL_R$}}}%
    \put(0.23809524,0.07380969){\color[rgb]{0,0,0}\makebox(0,0)[lb]{\smash{$\sL'_{\beta}$}}}%
    \put(0.49999987,0.12143002){\color[rgb]{0,0,0}\makebox(0,0)[lb]{\smash{$\sL'_{\beta^c}$}}}%
    \put(0.66666669,0.12142892){\color[rgb]{0,0,0}\makebox(0,0)[lb]{\smash{$\sL'_{\bar{\beta^c}}$}}}%
  \end{picture}%
\endgroup%
}}
\]
\item $\alpha$ is a marked cusp contained in $\sL_L$.
\end{enumerate}
\[
\vcenter{\hbox{\includegraphics{R5_1.pdf}}}\stackrel{\phi}\longrightarrow
\vcenter{\hbox{
\begingroup%
  \makeatletter%
  \providecommand\color[2][]{%
    \errmessage{(Inkscape) Color is used for the text in Inkscape, but the package 'color.sty' is not loaded}%
    \renewcommand\color[2][]{}%
  }%
  \providecommand\transparent[1]{%
    \errmessage{(Inkscape) Transparency is used (non-zero) for the text in Inkscape, but the package 'transparent.sty' is not loaded}%
    \renewcommand\transparent[1]{}%
  }%
  \providecommand\rotatebox[2]{#2}%
  \ifx\svgwidth\undefined%
    \setlength{\unitlength}{149.9999997bp}%
    \ifx\svgscale\undefined%
      \relax%
    \else%
      \setlength{\unitlength}{\unitlength * \real{\svgscale}}%
    \fi%
  \else%
    \setlength{\unitlength}{\svgwidth}%
  \fi%
  \global\let\svgwidth\undefined%
  \global\let\svgscale\undefined%
  \makeatother%
  \begin{picture}(1,0.33000004)%
    \put(0,0){\includegraphics[width=\unitlength,page=1]{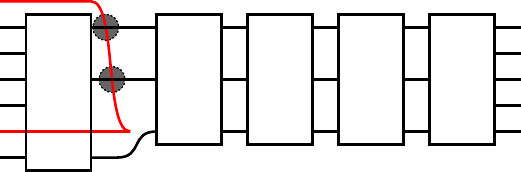}}%
    \put(0.075,0.07749998){\color[rgb]{0,0,0}\makebox(0,0)[lb]{\smash{$\sL'_L$}}}%
    \put(0.325,0.12750017){\color[rgb]{0,0,0}\makebox(0,0)[lb]{\smash{$\sL_{\beta}$}}}%
    \put(0.49999999,0.12750017){\color[rgb]{0,0,0}\makebox(0,0)[lb]{\smash{$\sL_{\beta^c}$}}}%
    \put(0.84999993,0.12750017){\color[rgb]{0,0,0}\makebox(0,0)[lb]{\smash{$\sL_R$}}}%
    \put(0.67499998,0.12750075){\color[rgb]{0,0,0}\makebox(0,0)[lb]{\smash{$\sL_{\bar{\beta^c}}$}}}%
  \end{picture}%
\endgroup%
}}
\]

Suppose that $\alpha$ is a strand of a braid. Then without loss of any generalities, we may assume that $\sL_L'$ is trivial, and the proof is as follows:
\begin{enumerate}
\item Make a small kink by using $\rm(I)$ on the top-left arc between $\sL_{\beta}$ and $\sL_{\beta^c}$.
\begin{align*}
\sfv_\phi=\vcenter{\hbox{}}
&\stackrel{\rm(I)}\longrightarrow
\vcenter{\hbox{
\begingroup%
  \makeatletter%
  \providecommand\color[2][]{%
    \errmessage{(Inkscape) Color is used for the text in Inkscape, but the package 'color.sty' is not loaded}%
    \renewcommand\color[2][]{}%
  }%
  \providecommand\transparent[1]{%
    \errmessage{(Inkscape) Transparency is used (non-zero) for the text in Inkscape, but the package 'transparent.sty' is not loaded}%
    \renewcommand\transparent[1]{}%
  }%
  \providecommand\rotatebox[2]{#2}%
  \ifx\svgwidth\undefined%
    \setlength{\unitlength}{172.50000238bp}%
    \ifx\svgscale\undefined%
      \relax%
    \else%
      \setlength{\unitlength}{\unitlength * \real{\svgscale}}%
    \fi%
  \else%
    \setlength{\unitlength}{\svgwidth}%
  \fi%
  \global\let\svgwidth\undefined%
  \global\let\svgscale\undefined%
  \makeatother%
  \begin{picture}(1,0.26521756)%
    \put(0,0){\includegraphics[width=\unitlength,page=1]{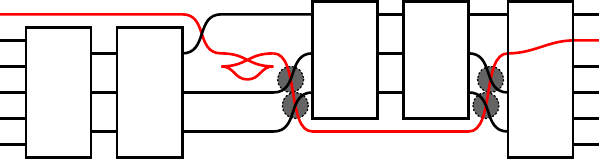}}%
    \put(0.0652174,0.06739146){\color[rgb]{0,0,0}\makebox(0,0)[lb]{\smash{$\sL'_L$}}}%
    \put(0.86956517,0.06739146){\color[rgb]{0,0,0}\makebox(0,0)[lb]{\smash{$\sL_R$}}}%
    \put(0.2173913,0.06739146){\color[rgb]{0,0,0}\makebox(0,0)[lb]{\smash{$\sL'_{\beta}$}}}%
    \put(0.54347816,0.11087089){\color[rgb]{0,0,0}\makebox(0,0)[lb]{\smash{$\sL'_{\beta^c}$}}}%
    \put(0.69565222,0.11086989){\color[rgb]{0,0,0}\makebox(0,0)[lb]{\smash{$\sL'_{\bar{\beta^c}}$}}}%
  \end{picture}%
\endgroup%
}}
\end{align*}
\item Pull-down the kink by using $\rm(II)$.
\begin{align*}
\vcenter{\hbox{}}
&\stackrel{\rm(II)}\longrightarrow
\vcenter{\hbox{
\begingroup%
  \makeatletter%
  \providecommand\color[2][]{%
    \errmessage{(Inkscape) Color is used for the text in Inkscape, but the package 'color.sty' is not loaded}%
    \renewcommand\color[2][]{}%
  }%
  \providecommand\transparent[1]{%
    \errmessage{(Inkscape) Transparency is used (non-zero) for the text in Inkscape, but the package 'transparent.sty' is not loaded}%
    \renewcommand\transparent[1]{}%
  }%
  \providecommand\rotatebox[2]{#2}%
  \ifx\svgwidth\undefined%
    \setlength{\unitlength}{172.50000238bp}%
    \ifx\svgscale\undefined%
      \relax%
    \else%
      \setlength{\unitlength}{\unitlength * \real{\svgscale}}%
    \fi%
  \else%
    \setlength{\unitlength}{\svgwidth}%
  \fi%
  \global\let\svgwidth\undefined%
  \global\let\svgscale\undefined%
  \makeatother%
  \begin{picture}(1,0.28695659)%
    \put(0,0){\includegraphics[width=\unitlength,page=1]{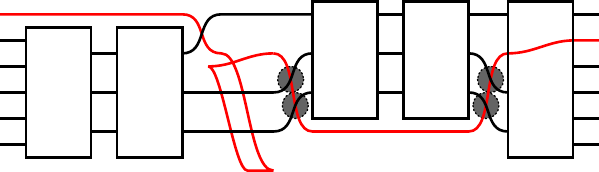}}%
    \put(0.0652174,0.08913049){\color[rgb]{0,0,0}\makebox(0,0)[lb]{\smash{$\sL'_L$}}}%
    \put(0.86956517,0.08913049){\color[rgb]{0,0,0}\makebox(0,0)[lb]{\smash{$\sL_R$}}}%
    \put(0.2173913,0.08913049){\color[rgb]{0,0,0}\makebox(0,0)[lb]{\smash{$\sL'_{\beta}$}}}%
    \put(0.54347816,0.13260992){\color[rgb]{0,0,0}\makebox(0,0)[lb]{\smash{$\sL'_{\beta^c}$}}}%
    \put(0.69565222,0.13260892){\color[rgb]{0,0,0}\makebox(0,0)[lb]{\smash{$\sL'_{\bar{\beta^c}}$}}}%
  \end{picture}%
\endgroup%
}}
\end{align*}
\item Move the right long arc to the right most position by Lemma~\ref{lemma:black pure braid commutes} and Lemma~\ref{lemma:marked cusp commutes}.
\begin{align*}
\vcenter{\hbox{}}
&\stackrel{\rm(A)}\longrightarrow
\vcenter{\hbox{
\begingroup%
  \makeatletter%
  \providecommand\color[2][]{%
    \errmessage{(Inkscape) Color is used for the text in Inkscape, but the package 'color.sty' is not loaded}%
    \renewcommand\color[2][]{}%
  }%
  \providecommand\transparent[1]{%
    \errmessage{(Inkscape) Transparency is used (non-zero) for the text in Inkscape, but the package 'transparent.sty' is not loaded}%
    \renewcommand\transparent[1]{}%
  }%
  \providecommand\rotatebox[2]{#2}%
  \ifx\svgwidth\undefined%
    \setlength{\unitlength}{168.7500031bp}%
    \ifx\svgscale\undefined%
      \relax%
    \else%
      \setlength{\unitlength}{\unitlength * \real{\svgscale}}%
    \fi%
  \else%
    \setlength{\unitlength}{\svgwidth}%
  \fi%
  \global\let\svgwidth\undefined%
  \global\let\svgscale\undefined%
  \makeatother%
  \begin{picture}(1,0.31555566)%
    \put(0,0){\includegraphics[width=\unitlength,page=1]{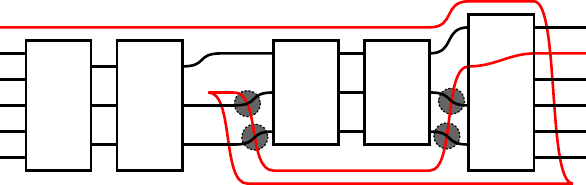}}%
    \put(0.06666667,0.09111117){\color[rgb]{0,0,0}\makebox(0,0)[lb]{\smash{$\sL'_L$}}}%
    \put(0.82222229,0.09111117){\color[rgb]{0,0,0}\makebox(0,0)[lb]{\smash{$\sL_R$}}}%
    \put(0.22222221,0.09111117){\color[rgb]{0,0,0}\makebox(0,0)[lb]{\smash{$\sL'_{\beta}$}}}%
    \put(0.48888887,0.11333424){\color[rgb]{0,0,0}\makebox(0,0)[lb]{\smash{$\sL'_{\beta^c}$}}}%
    \put(0.6444446,0.11333322){\color[rgb]{0,0,0}\makebox(0,0)[lb]{\smash{$\sL'_{\bar{\beta^c}}$}}}%
  \end{picture}%
\endgroup%
}}
\end{align*}
\item Reduce the left cusp by applying $\rm(II)$ and $(0)$.
\begin{align*}
\vcenter{\hbox{}}
&\substack{{\rm(II)}\\\longrightarrow\\\rm(0)}
\vcenter{\hbox{
\begingroup%
  \makeatletter%
  \providecommand\color[2][]{%
    \errmessage{(Inkscape) Color is used for the text in Inkscape, but the package 'color.sty' is not loaded}%
    \renewcommand\color[2][]{}%
  }%
  \providecommand\transparent[1]{%
    \errmessage{(Inkscape) Transparency is used (non-zero) for the text in Inkscape, but the package 'transparent.sty' is not loaded}%
    \renewcommand\transparent[1]{}%
  }%
  \providecommand\rotatebox[2]{#2}%
  \ifx\svgwidth\undefined%
    \setlength{\unitlength}{153.7500022bp}%
    \ifx\svgscale\undefined%
      \relax%
    \else%
      \setlength{\unitlength}{\unitlength * \real{\svgscale}}%
    \fi%
  \else%
    \setlength{\unitlength}{\svgwidth}%
  \fi%
  \global\let\svgwidth\undefined%
  \global\let\svgscale\undefined%
  \makeatother%
  \begin{picture}(1,0.34634158)%
    \put(0,0){\includegraphics[width=\unitlength,page=1]{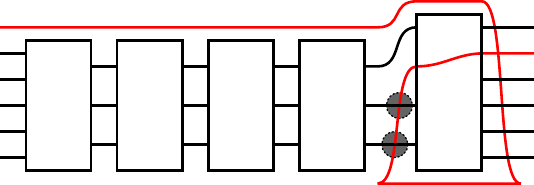}}%
    \put(0.07317073,0.10000006){\color[rgb]{0,0,0}\makebox(0,0)[lb]{\smash{$\sL'_L$}}}%
    \put(0.80487796,0.10000006){\color[rgb]{0,0,0}\makebox(0,0)[lb]{\smash{$\sL_R$}}}%
    \put(0.24390242,0.10000006){\color[rgb]{0,0,0}\makebox(0,0)[lb]{\smash{$\sL'_{\beta}$}}}%
    \put(0.41463394,0.10000062){\color[rgb]{0,0,0}\makebox(0,0)[lb]{\smash{$\sL'_{\beta^c}$}}}%
    \put(0.58536584,0.0999995){\color[rgb]{0,0,0}\makebox(0,0)[lb]{\smash{$\sL'_{\bar{\beta^c}}$}}}%
  \end{picture}%
\endgroup%
}}
\end{align*}
\item Apply $\rm(S)$ to make a standard form.
\begin{align*}
\vcenter{\hbox{}}
&\stackrel{\rm(S)}\longrightarrow
\vcenter{\hbox{
\begingroup%
  \makeatletter%
  \providecommand\color[2][]{%
    \errmessage{(Inkscape) Color is used for the text in Inkscape, but the package 'color.sty' is not loaded}%
    \renewcommand\color[2][]{}%
  }%
  \providecommand\transparent[1]{%
    \errmessage{(Inkscape) Transparency is used (non-zero) for the text in Inkscape, but the package 'transparent.sty' is not loaded}%
    \renewcommand\transparent[1]{}%
  }%
  \providecommand\rotatebox[2]{#2}%
  \ifx\svgwidth\undefined%
    \setlength{\unitlength}{165.00000217bp}%
    \ifx\svgscale\undefined%
      \relax%
    \else%
      \setlength{\unitlength}{\unitlength * \real{\svgscale}}%
    \fi%
  \else%
    \setlength{\unitlength}{\svgwidth}%
  \fi%
  \global\let\svgwidth\undefined%
  \global\let\svgscale\undefined%
  \makeatother%
  \begin{picture}(1,0.36590868)%
    \put(0,0){\includegraphics[width=\unitlength,page=1]{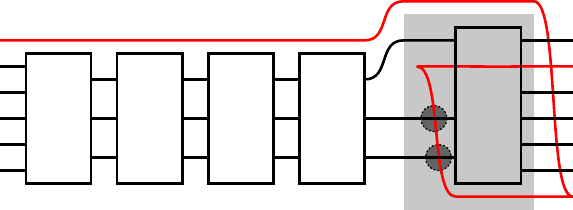}}%
    \put(0.06818182,0.11363601){\color[rgb]{0,0,0}\makebox(0,0)[lb]{\smash{$\sL'_L$}}}%
    \put(0.81818181,0.11363601){\color[rgb]{0,0,0}\makebox(0,0)[lb]{\smash{$\sL_R$}}}%
    \put(0.22727271,0.11363601){\color[rgb]{0,0,0}\makebox(0,0)[lb]{\smash{$\sL'_{\beta}$}}}%
    \put(0.38636344,0.11363654){\color[rgb]{0,0,0}\makebox(0,0)[lb]{\smash{$\sL'_{\beta^c}$}}}%
    \put(0.54545454,0.11363549){\color[rgb]{0,0,0}\makebox(0,0)[lb]{\smash{$\sL'_{\bar{\beta^c}}$}}}%
  \end{picture}%
\endgroup%
}}
\end{align*}
\item Then we can regard the result as a $\phi$-resolution $\sfv'_\phi$ of a vertex $\sfv'$ of type $(\ell-1,r+1)$.
\begin{align*}
\sfv'_\phi=\vcenter{\hbox{
\begingroup%
  \makeatletter%
  \providecommand\color[2][]{%
    \errmessage{(Inkscape) Color is used for the text in Inkscape, but the package 'color.sty' is not loaded}%
    \renewcommand\color[2][]{}%
  }%
  \providecommand\transparent[1]{%
    \errmessage{(Inkscape) Transparency is used (non-zero) for the text in Inkscape, but the package 'transparent.sty' is not loaded}%
    \renewcommand\transparent[1]{}%
  }%
  \providecommand\rotatebox[2]{#2}%
  \ifx\svgwidth\undefined%
    \setlength{\unitlength}{176.62500181bp}%
    \ifx\svgscale\undefined%
      \relax%
    \else%
      \setlength{\unitlength}{\unitlength * \real{\svgscale}}%
    \fi%
  \else%
    \setlength{\unitlength}{\svgwidth}%
  \fi%
  \global\let\svgwidth\undefined%
  \global\let\svgscale\undefined%
  \makeatother%
  \begin{picture}(1,0.38641188)%
    \put(0,0){\includegraphics[width=\unitlength,page=1]{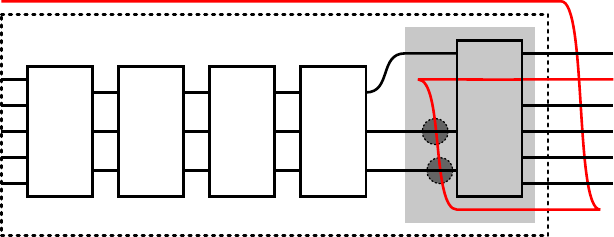}}%
    \put(0.06581741,0.12951168){\color[rgb]{0,0,0}\makebox(0,0)[lb]{\smash{$\sL'_L$}}}%
    \put(0.76645434,0.12951168){\color[rgb]{0,0,0}\makebox(0,0)[lb]{\smash{$\sL_R$}}}%
    \put(0.21443735,0.12951168){\color[rgb]{0,0,0}\makebox(0,0)[lb]{\smash{$\sL'_{\beta}$}}}%
    \put(0.36305715,0.12951217){\color[rgb]{0,0,0}\makebox(0,0)[lb]{\smash{$\sL'_{\beta^c}$}}}%
    \put(0.51167728,0.12951119){\color[rgb]{0,0,0}\makebox(0,0)[lb]{\smash{$\sL'_{\bar{\beta^c}}$}}}%
  \end{picture}%
\endgroup%
}}
\stackrel{\phi}\longleftarrow
\vcenter{\hbox{\includegraphics{R5_1_2.pdf}}}.
\end{align*}
\end{enumerate}

Now suppose that $\alpha$ is a marked cusp.
Then the proof is exactly the same as the reverse of the above by using Lemma~\ref{lemma:delta conjugate} as follows:
\begin{enumerate}
\item Apply Lemma~\ref{lemma:delta conjugate} to $\sL_{\beta}\sL_{\beta^c}\sL_{\bar{\beta^c}}$.
\begin{align*}
\sfv_\phi=\vcenter{\hbox{}}
&\stackrel{\rm(I)}\longrightarrow
\vcenter{\hbox{
\begingroup%
  \makeatletter%
  \providecommand\color[2][]{%
    \errmessage{(Inkscape) Color is used for the text in Inkscape, but the package 'color.sty' is not loaded}%
    \renewcommand\color[2][]{}%
  }%
  \providecommand\transparent[1]{%
    \errmessage{(Inkscape) Transparency is used (non-zero) for the text in Inkscape, but the package 'transparent.sty' is not loaded}%
    \renewcommand\transparent[1]{}%
  }%
  \providecommand\rotatebox[2]{#2}%
  \ifx\svgwidth\undefined%
    \setlength{\unitlength}{153.75000631bp}%
    \ifx\svgscale\undefined%
      \relax%
    \else%
      \setlength{\unitlength}{\unitlength * \real{\svgscale}}%
    \fi%
  \else%
    \setlength{\unitlength}{\svgwidth}%
  \fi%
  \global\let\svgwidth\undefined%
  \global\let\svgscale\undefined%
  \makeatother%
  \begin{picture}(1,0.32195127)%
    \put(0,0){\includegraphics[width=\unitlength,page=1]{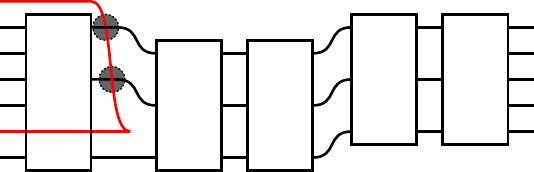}}%
    \put(0.07317073,0.07560975){\color[rgb]{0,0,0}\makebox(0,0)[lb]{\smash{$\sL'_L$}}}%
    \put(0.29268291,0.0512197){\color[rgb]{0,0,0}\makebox(0,0)[lb]{\smash{\tiny$\sL_{\tau(\bar{\beta^c})}$}}}%
    \put(0.46341461,0.0512197){\color[rgb]{0,0,0}\makebox(0,0)[lb]{\smash{\tiny$\sL_{\tau(\beta^c)}$}}}%
    \put(0.85365873,0.10000037){\color[rgb]{0,0,0}\makebox(0,0)[lb]{\smash{$\sL_R$}}}%
    \put(0.6829268,0.09999981){\color[rgb]{0,0,0}\makebox(0,0)[lb]{\smash{$\sL_{\beta}$}}}%
  \end{picture}%
\endgroup%
}}
\end{align*}
\item Apply $\rm(S)$ to the cusp.
\begin{align*}
\vcenter{\hbox{}}
&\stackrel{\rm(I)}\longrightarrow
\vcenter{\hbox{
\begingroup%
  \makeatletter%
  \providecommand\color[2][]{%
    \errmessage{(Inkscape) Color is used for the text in Inkscape, but the package 'color.sty' is not loaded}%
    \renewcommand\color[2][]{}%
  }%
  \providecommand\transparent[1]{%
    \errmessage{(Inkscape) Transparency is used (non-zero) for the text in Inkscape, but the package 'transparent.sty' is not loaded}%
    \renewcommand\transparent[1]{}%
  }%
  \providecommand\rotatebox[2]{#2}%
  \ifx\svgwidth\undefined%
    \setlength{\unitlength}{153.75000631bp}%
    \ifx\svgscale\undefined%
      \relax%
    \else%
      \setlength{\unitlength}{\unitlength * \real{\svgscale}}%
    \fi%
  \else%
    \setlength{\unitlength}{\svgwidth}%
  \fi%
  \global\let\svgwidth\undefined%
  \global\let\svgscale\undefined%
  \makeatother%
  \begin{picture}(1,0.32195127)%
    \put(0,0){\includegraphics[width=\unitlength,page=1]{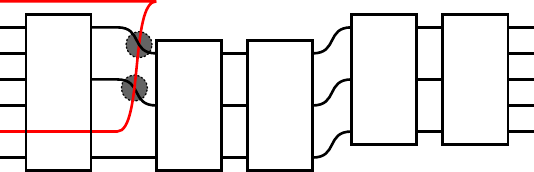}}%
    \put(0.07317073,0.07560975){\color[rgb]{0,0,0}\makebox(0,0)[lb]{\smash{$\sL'_L$}}}%
    \put(0.29268291,0.0512197){\color[rgb]{0,0,0}\makebox(0,0)[lb]{\tiny\smash{$\sL_{\tau(\bar{\beta^c})}$}}}%
    \put(0.46341461,0.0512197){\color[rgb]{0,0,0}\makebox(0,0)[lb]{\tiny\smash{$\sL_{\tau(\beta^c)}$}}}%
    \put(0.85365873,0.10000037){\color[rgb]{0,0,0}\makebox(0,0)[lb]{\smash{$\sL_R$}}}%
    \put(0.6829268,0.09999981){\color[rgb]{0,0,0}\makebox(0,0)[lb]{\smash{$\sL_{\beta}$}}}%
  \end{picture}%
\endgroup%
}}
\end{align*}
\item Push and pull down the cusp via $(0)$ and $\rm(II)$.
\begin{align*}
\vcenter{\hbox{}}
&\substack{{\rm(0)}\\\longrightarrow\\\rm(II)}
\vcenter{\hbox{
\begingroup%
  \makeatletter%
  \providecommand\color[2][]{%
    \errmessage{(Inkscape) Color is used for the text in Inkscape, but the package 'color.sty' is not loaded}%
    \renewcommand\color[2][]{}%
  }%
  \providecommand\transparent[1]{%
    \errmessage{(Inkscape) Transparency is used (non-zero) for the text in Inkscape, but the package 'transparent.sty' is not loaded}%
    \renewcommand\transparent[1]{}%
  }%
  \providecommand\rotatebox[2]{#2}%
  \ifx\svgwidth\undefined%
    \setlength{\unitlength}{168.75004123bp}%
    \ifx\svgscale\undefined%
      \relax%
    \else%
      \setlength{\unitlength}{\unitlength * \real{\svgscale}}%
    \fi%
  \else%
    \setlength{\unitlength}{\svgwidth}%
  \fi%
  \global\let\svgwidth\undefined%
  \global\let\svgscale\undefined%
  \makeatother%
  \begin{picture}(1,0.29333332)%
    \put(0,0){\includegraphics[width=\unitlength,page=1]{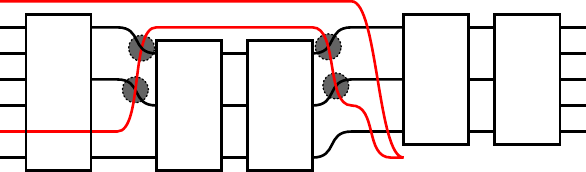}}%
    \put(0.06666665,0.06888887){\color[rgb]{0,0,0}\makebox(0,0)[lb]{\smash{$\sL'_L$}}}%
    \put(0.2666666,0.04666683){\color[rgb]{0,0,0}\makebox(0,0)[lb]{\tiny\smash{$\sL_{\tau(\bar{\beta^c})}$}}}%
    \put(0.42222212,0.04666683){\color[rgb]{0,0,0}\makebox(0,0)[lb]{\tiny\smash{$\sL_{\tau(\beta^c)}$}}}%
    \put(0.86666688,0.09111246){\color[rgb]{0,0,0}\makebox(0,0)[lb]{\smash{$\sL_R$}}}%
    \put(0.71111115,0.09111194){\color[rgb]{0,0,0}\makebox(0,0)[lb]{\smash{$\sL_{\beta}$}}}%
  \end{picture}%
\endgroup%
}}
\end{align*}
\item Move the long right arc to the right most position by Lemma~\ref{lemma:black pure braid commutes} and Lemma~\ref{lemma:marked cusp commutes}.
\begin{align*}
\vcenter{\hbox{}}
&\stackrel{\rm(B)}\longrightarrow
\vcenter{\hbox{
\begingroup%
  \makeatletter%
  \providecommand\color[2][]{%
    \errmessage{(Inkscape) Color is used for the text in Inkscape, but the package 'color.sty' is not loaded}%
    \renewcommand\color[2][]{}%
  }%
  \providecommand\transparent[1]{%
    \errmessage{(Inkscape) Transparency is used (non-zero) for the text in Inkscape, but the package 'transparent.sty' is not loaded}%
    \renewcommand\transparent[1]{}%
  }%
  \providecommand\rotatebox[2]{#2}%
  \ifx\svgwidth\undefined%
    \setlength{\unitlength}{176.25000398bp}%
    \ifx\svgscale\undefined%
      \relax%
    \else%
      \setlength{\unitlength}{\unitlength * \real{\svgscale}}%
    \fi%
  \else%
    \setlength{\unitlength}{\svgwidth}%
  \fi%
  \global\let\svgwidth\undefined%
  \global\let\svgscale\undefined%
  \makeatother%
  \begin{picture}(1,0.3021276)%
    \put(0,0){\includegraphics[width=\unitlength,page=1]{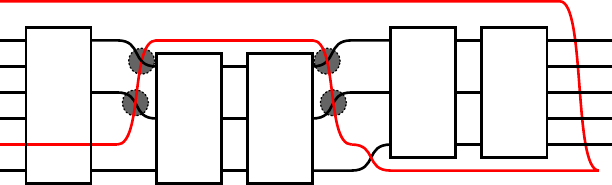}}%
    \put(0.06382979,0.06595745){\color[rgb]{0,0,0}\makebox(0,0)[lb]{\smash{$\sL'_L$}}}%
    \put(0.25531914,0.04468101){\color[rgb]{0,0,0}\makebox(0,0)[lb]{\tiny\smash{$\sL_{\tau(\bar{\beta^c})}$}}}%
    \put(0.40425531,0.04468101){\color[rgb]{0,0,0}\makebox(0,0)[lb]{\tiny\smash{$\sL_{\tau(\beta^c)}$}}}%
    \put(0.80851082,0.08723535){\color[rgb]{0,0,0}\makebox(0,0)[lb]{\smash{$\sL_R$}}}%
    \put(0.65957464,0.08723486){\color[rgb]{0,0,0}\makebox(0,0)[lb]{\smash{$\sL_{\beta}$}}}%
  \end{picture}%
\endgroup%
}}
\end{align*}
\item Apply Lemma~\ref{lemma:delta conjugate} again.
\begin{align*}
\vcenter{\hbox{
\begingroup%
  \makeatletter%
  \providecommand\color[2][]{%
    \errmessage{(Inkscape) Color is used for the text in Inkscape, but the package 'color.sty' is not loaded}%
    \renewcommand\color[2][]{}%
  }%
  \providecommand\transparent[1]{%
    \errmessage{(Inkscape) Transparency is used (non-zero) for the text in Inkscape, but the package 'transparent.sty' is not loaded}%
    \renewcommand\transparent[1]{}%
  }%
  \providecommand\rotatebox[2]{#2}%
  \ifx\svgwidth\undefined%
    \setlength{\unitlength}{176.25000398bp}%
    \ifx\svgscale\undefined%
      \relax%
    \else%
      \setlength{\unitlength}{\unitlength * \real{\svgscale}}%
    \fi%
  \else%
    \setlength{\unitlength}{\svgwidth}%
  \fi%
  \global\let\svgwidth\undefined%
  \global\let\svgscale\undefined%
  \makeatother%
  \begin{picture}(1,0.32553184)%
    \put(0,0){\includegraphics[width=\unitlength,page=1]{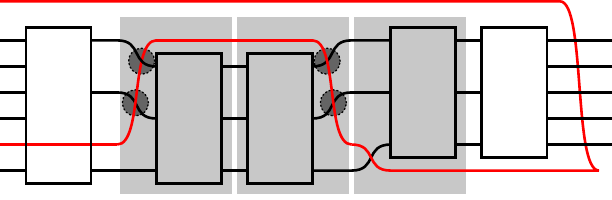}}%
    \put(0.06382979,0.08936169){\color[rgb]{0,0,0}\makebox(0,0)[lb]{\smash{$\sL'_L$}}}%
    \put(0.25531914,0.06808526){\color[rgb]{0,0,0}\makebox(0,0)[lb]{\tiny\smash{$\sL_{\tau(\bar{\beta^c})}$}}}%
    \put(0.40425531,0.06808526){\color[rgb]{0,0,0}\makebox(0,0)[lb]{\tiny\smash{$\sL_{\tau(\beta^c)}$}}}%
    \put(0.80851082,0.11063959){\color[rgb]{0,0,0}\makebox(0,0)[lb]{\smash{$\sL_R$}}}%
    \put(0.65957464,0.1106391){\color[rgb]{0,0,0}\makebox(0,0)[lb]{\smash{$\sL_{\beta}$}}}%
  \end{picture}%
\endgroup%
}}
&\stackrel{\rm(B)}\longrightarrow
\vcenter{\hbox{
\begingroup%
  \makeatletter%
  \providecommand\color[2][]{%
    \errmessage{(Inkscape) Color is used for the text in Inkscape, but the package 'color.sty' is not loaded}%
    \renewcommand\color[2][]{}%
  }%
  \providecommand\transparent[1]{%
    \errmessage{(Inkscape) Transparency is used (non-zero) for the text in Inkscape, but the package 'transparent.sty' is not loaded}%
    \renewcommand\transparent[1]{}%
  }%
  \providecommand\rotatebox[2]{#2}%
  \ifx\svgwidth\undefined%
    \setlength{\unitlength}{176.25000398bp}%
    \ifx\svgscale\undefined%
      \relax%
    \else%
      \setlength{\unitlength}{\unitlength * \real{\svgscale}}%
    \fi%
  \else%
    \setlength{\unitlength}{\svgwidth}%
  \fi%
  \global\let\svgwidth\undefined%
  \global\let\svgscale\undefined%
  \makeatother%
  \begin{picture}(1,0.32553203)%
    \put(0,0){\includegraphics[width=\unitlength,page=1]{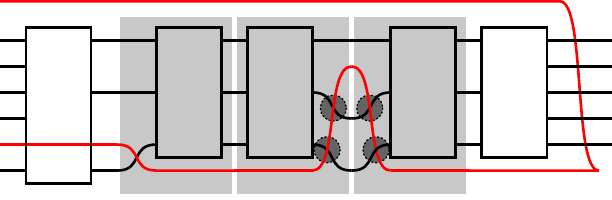}}%
    \put(0.06382979,0.08936187){\color[rgb]{0,0,0}\makebox(0,0)[lb]{\smash{$\sL'_L$}}}%
    \put(0.6595747,0.11063831){\color[rgb]{0,0,0}\makebox(0,0)[lb]{\smash{$\sL_{\bar{\beta^c}}$}}}%
    \put(0.42553189,0.11063831){\color[rgb]{0,0,0}\makebox(0,0)[lb]{\smash{$\sL_{\beta^c}$}}}%
    \put(0.80851082,0.11063978){\color[rgb]{0,0,0}\makebox(0,0)[lb]{\smash{$\sL_R$}}}%
    \put(0.27659595,0.1106388){\color[rgb]{0,0,0}\makebox(0,0)[lb]{\smash{$\sL_{\beta}$}}}%
  \end{picture}%
\endgroup%
}}
\end{align*}
\item Regard the result as a $\phi$-resolution $\sfv'_\phi$ of a vertex $\sfv'$ of type $(\ell-1,r+1)$ as before.
\begin{align*}
\vcenter{\hbox{
\begingroup%
  \makeatletter%
  \providecommand\color[2][]{%
    \errmessage{(Inkscape) Color is used for the text in Inkscape, but the package 'color.sty' is not loaded}%
    \renewcommand\color[2][]{}%
  }%
  \providecommand\transparent[1]{%
    \errmessage{(Inkscape) Transparency is used (non-zero) for the text in Inkscape, but the package 'transparent.sty' is not loaded}%
    \renewcommand\transparent[1]{}%
  }%
  \providecommand\rotatebox[2]{#2}%
  \ifx\svgwidth\undefined%
    \setlength{\unitlength}{176.62500375bp}%
    \ifx\svgscale\undefined%
      \relax%
    \else%
      \setlength{\unitlength}{\unitlength * \real{\svgscale}}%
    \fi%
  \else%
    \setlength{\unitlength}{\svgwidth}%
  \fi%
  \global\let\svgwidth\undefined%
  \global\let\svgscale\undefined%
  \makeatother%
  \begin{picture}(1,0.32271758)%
    \put(0,0){\includegraphics[width=\unitlength,page=1]{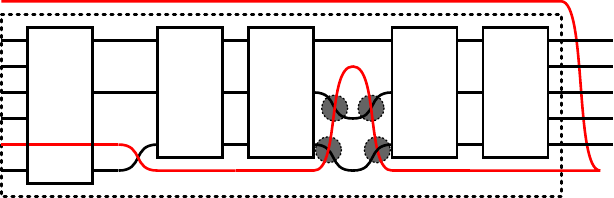}}%
    \put(0.06581741,0.08704885){\color[rgb]{0,0,0}\makebox(0,0)[lb]{\smash{$\sL'_L$}}}%
    \put(0.66029747,0.10828011){\color[rgb]{0,0,0}\makebox(0,0)[lb]{\smash{$\sL_{\bar{\beta^c}}$}}}%
    \put(0.42675157,0.10828011){\color[rgb]{0,0,0}\makebox(0,0)[lb]{\smash{$\sL_{\beta^c}$}}}%
    \put(0.80891738,0.10828158){\color[rgb]{0,0,0}\makebox(0,0)[lb]{\smash{$\sL_R$}}}%
    \put(0.27813184,0.1082806){\color[rgb]{0,0,0}\makebox(0,0)[lb]{\smash{$\sL_{\beta}$}}}%
  \end{picture}%
\endgroup%
}}
\stackrel{\phi}\longleftarrow
\vcenter{\hbox{\includegraphics{R5_1_2.pdf}}}.
\end{align*}
\end{enumerate}

\bibliographystyle{abbrv}
\bibliography{references}

\end{document}